\documentclass[12pt, reqno, oneside]{amsart}
\usepackage{amsmath, amsthm, amsopn, amssymb, microtype}
\usepackage{mathrsfs} 
\usepackage{mathscinet}
\usepackage{bbm}
\usepackage{enumerate, etoolbox, intcalc, geometry, caption, subcaption, afterpage}
\usepackage[english]{babel}
\usepackage{float}
\usepackage[all]{xy}
\usepackage{graphicx}

\usepackage[colorlinks=true,urlcolor=blue,pdfborder={0 0 0}]{hyperref}
\hypersetup{linkcolor=[rgb]{0,0,0.6}}
\hypersetup{citecolor=[rgb]{0,0.6,0}}
\usepackage{color}

\xyoption{color}

\title{BiLipschitz and bounded displacement equivalence of Delone sets}

\author{Yotam Smilansky}
\address{University of Manchester. Department of Mathematics. Oxford Rd, Manchester M13 9PL, United Kingdom.
	{\tt yotam.smilansky@manchester.ac.uk}
}

\author{Yaar Solomon}
\address{Ben-Gurion University of the Negev.
	Department of Mathematics.
	Beer-Sheva, 8410501, Israel. 
	{\tt yaars@bgu.ac.il}
}

\newcommand{\N}{{\mathbb{N}}}
\newcommand{\Z}{{\mathbb{Z}}}
\newcommand{\Q}{{\mathbb{Q}}}
\newcommand{\R}{{\mathbb{R}}}
\newcommand{\C}{{\mathbb{C}}}

\newcommand{\HHH}{{\mathscr{H}^{(d)}}}
\newcommand{\CC}{{\mathscr{C}}}
\renewcommand{\AA}{\mathcal{T}}
\newcommand{\LL}{\mathcal{L}}
\newcommand{\OO}{{\mathcal{O}}}
\newcommand{\PP}{{\mathcal{P}}}
\renewcommand{\SS}{{\mathcal{S}}}

\newcommand{\WW}{\mathcal{W}}
\newcommand{\bd}{\stackrel{\rm BD}{\sim}}

\newcommand{\minus}{\smallsetminus}

\newcommand{\absolute}[1] {\left|{#1}\right|}

\newcommand{\norm}[1]{\left\|{#1}\right\|}
\newcommand{\inpro}[2]{\langle{#1},{#2}\rangle}
\newcommand{\disc}[3]{\mathbf{disc}_{#1}\left({#2},{#3}\right)}

\newcommand{\dist}{\operatorname{dist}}
\newcommand{\vol}{\operatorname{vol}}

\newcommand{\diam}{\operatorname{diam}}
\newcommand{\BD}{\operatorname{BD}}
\newcommand{\inter}{\operatorname{int}}
\newcommand{\Jac}{\operatorname{Jac}}
\newcommand{\spn}{\operatorname{span}}
\newcommand{\bL}{\operatorname{biLip}}
\newcommand{\supp}{\operatorname{supp}}
\newcommand{\op}{\operatorname{op}}
\newcommand{\covol}{\operatorname{covol}}

\newcommand {\ignore}[1]  {}

\theoremstyle{plain}
\newtheorem{thm}{Theorem}[section]
\newtheorem*{thm*}{Theorem}
\newtheorem{lem}[thm]{Lemma}
\newtheorem{prop}[thm]{Proposition}
\newtheorem*{prop*}{Proposition}
\newtheorem{cor}[thm]{Corollary}

\theoremstyle{definition}
\newtheorem{definition}[thm]{Definition}

\newtheorem{example}[thm]{Example}
\newtheorem{remark}[thm]{Remark}
\newtheorem{open}[thm]{Open Problem}
\newtheorem{exe}[thm]{Exercise}

\newtheorem*{remark*}{Remark}

\numberwithin{equation}{section}
\swapnumbers

\newif\ifdraft\drafttrue

\begin{document}

\begin{abstract}
    We survey biLipschitz (BL) and bounded displacement (BD) equivalence of Delone sets, with an emphasis on examples arising in aperiodic order. After recalling the classical results of Burago--Kleiner, McMullen, and Laczkovich, we discuss criteria for rectifiability, uniform spreadness, and BD equivalence, and how these relate to discrepancy and point-counting estimates. We then focus on Delone sets associated with substitution tilings, where these questions can often be studied through the combinatorial and spectral properties of the underlying substitution rules. We also consider selected classes of cut-and-project sets and briefly discuss extensions beyond BL and BD equivalence.
\end{abstract}

\maketitle

\section{Introduction}\label{sec:introduction}
Delone sets provide a flexible but structured mathematical framework for modeling atomic configurations in crystals and quasicrystals. Unlike lattices, Delone sets need not exhibit any periodicity, yet they retain enough geometric regularity to capture essential features of ordered materials. Quantifying how closely a Delone set resembles a lattice can be approached by examining equivalence relations that preserve large-scale structure. 
This contribution is dedicated to the study of two such equivalence relations, namely, the bounded displacement and biLipschitz equivalence relations. Bounded displacement (BD) equivalence considers whether there exists a bijection between two Delone sets that displaces each point by at most a fixed constant. In contrast, biLipschitz (BL) equivalence requires a bijection between the sets that distorts distances only up to uniform multiplicative bounds. In this way, BD and BL equivalence relations provide rigorous means of comparing Delone sets to lattices and each other, allowing a classification of sets that are geometrically close to lattices despite a lack of periodicity.

The outline of this survey is as follows. We will introduce the basic relevant definitions and properties of Delone sets in \S\ref{sec:introduction} and \S\ref{sec:FLC_topology_etc} before discussing fundamental results about BL and BD equivalence in \S\ref{sec:BL} and \S\ref{sec:BD}, respectively. We will then consider constructions in aperiodic order, namely tilings defined via substitution rules in \S \ref{sec:substitution} and \S \ref{sec:multiscale}, and cut-and-project sets in \S \ref{sec:C&P}. Finally, we will discuss additional equivalence relations and properties that extend the study of BL and BD equivalence in \S \ref{sec:beyond BL and BD}. 

Throughout this document, $\norm{\cdot}$ denotes the Euclidean norm
on $\R^d$. Since all norms on a finite-dimensional vector space are
equivalent, this choice does not affect the notions of a
Delone set, BD equivalence, or BL equivalence. The Euclidean
assumption will, however, be used in statements concerning Voronoi
cells, convexity, and Euclidean volume.

\begin{definition}
    Let $\norm{\cdot}$ be a norm on $\R^d$. A set $\Lambda \subset \R^d$ is a \emph{Delone set} if it is both 
    \begin{enumerate}
        \item 
        \emph{uniformly discrete}: $r:=\inf\{\norm{x_1-x_2} ~\mid~ x_1, x_2\in \Lambda,\ x_1\neq x_2\} >0$, and 
        \item 
        \emph{relatively dense}: $R:=\sup\{\dist(y,\Lambda)~\mid~ y\in \R^d\} <\infty$, 
    \end{enumerate}
    where $\dist(y,\Lambda)=\inf_{x\in \Lambda}\norm{x-y}$ stands for the distance between a point and a set.
\end{definition}
    Observe that every open ball of radius $r/2$ contains at most one point of $\Lambda$ and every closed ball of radius $R$ contains at least one point of $\Lambda$. The constants $r$ and $R$ are often called the \emph{separation constant} and the \emph{covering radius} of $\Lambda$, respectively, while $r/2$ is referred to as the \emph{packing radius} in some contexts. 
    
    Basic examples of Delone sets in $\R^d$ include the integer lattice $\Z^d$, or more generally, any translated lattice of the form $x+g\cdot\Z^d$ for $x\in\R^d, g\in {\rm GL}_d(\R)$. In subsequent sections, we will introduce and discuss other interesting constructions, with a focus on those that come up in the study of aperiodic order, namely non-periodic sets that possess some long-range order. Examples of sets which are not Delone sets include the set of primitive lattice points $\Z^d_{\rm prim}=\{x=(x_1,\ldots,x_d)\mid \gcd(x_1,\ldots,x_d)=1\}$, for which relative denseness fails, and the union of lattices $\Z^d\cup\sqrt{2}\Z^d$, which fails to be uniformly discrete.

    A Delone set is defined purely by its geometric spacing conditions, and beyond these, it carries no further built-in structure. In particular, nothing in the definition requires any form of order, such as periodicity, a fixed density, or any restriction on the number of distinct local patches. Thus, a Delone set may exhibit a wide range of behaviors, from highly ordered lattices to completely irregular or even pathological configurations, with the only constraints being the two basic separation and covering bounds.

\subsection{The BL and BD equivalence relations} 
The BD and BL equivalence relations of Delone sets provide quantitative methods for comparing Delone sets by measuring the extent to which one set must be deformed or shifted to match another. We focus in particular on understanding when a Delone set is BD or BL equivalent to a lattice, and what structural or combinatorial properties govern these equivalence relations. 

Let $\Lambda_1, \Lambda_2 \subset \mathbb{R}^d$ be Delone sets.
\begin{definition}
We say that $\Lambda_1$ and $\Lambda_2$ are \emph{biLipschitz (BL) equivalent}, and write $\Lambda_1 \stackrel{\mathrm{BL}}\sim \Lambda_2$, if there exists a biLipschitz bijection $\psi : \Lambda_1 \to \Lambda_2$. That is, a bijection for which there exists a constant $L\ge 1$ so that
\[
\forall x_1\neq x_2\in \Lambda_1:\quad \frac{1}{L} \le \frac{\norm{\psi(x_1) - \psi(x_2)}}{\norm{x_1-x_2}} \le L.
\] 
The bijection $\psi$ is called a \emph{BL-map}, and it distorts distances only within fixed multiplicative bounds.
A Delone set $\Lambda\subset\R^d$ is \emph{rectifiable} if it is BL-equivalent to $\Z^d$.
\end{definition}

\begin{definition}\label{def:BD}
We say that $\Lambda_1$ and $\Lambda_2$ are \emph{bounded displacement (BD) equivalent}, and write $\Lambda_1 \stackrel{\mathrm{BD}}\sim \Lambda_2$, if there exists a bijection $\phi : \Lambda_1 \to \Lambda_2$ such that
\[
\sup_{x \in \Lambda_1} \norm{\phi(x) - x} < \infty. 
\]
The bijection $\phi$ is called a \emph{BD-map}, and it moves each point of $\Lambda_1$ to its image in $\Lambda_2$ by at most a uniformly bounded distance. 
A Delone set $\Lambda\subset\R^d$ is \emph{uniformly spread} if it is BD equivalent to $\alpha\Z^d$ for some $\alpha>0$.  
\end{definition}

\begin{figure}[h!]
    \centering
    \includegraphics[scale=0.58]{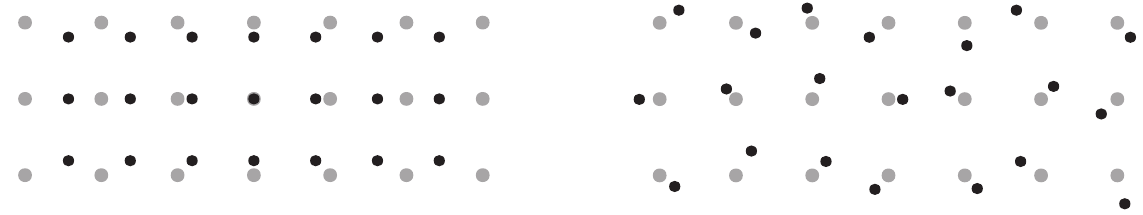}
    \caption{An illustration of a BL-map (left) and a BD-map (right). In the BL setting, pairwise distances are changed by at most a uniform multiplicative factor, although individual displacements may be arbitrarily large. In the BD setting, every point is displaced by a uniformly bounded amount. The original set is shown in gray and its image in black.}
    \label{fig:BDBL}
\end{figure}

\begin{exe}
    Verify that the above definitions define equivalence relations on the set of Delone sets in $\R^d$. 
\end{exe}

\begin{exe}
    Given a Delone set $\Lambda\subset\R^d$, find a BD-map $\phi:\Lambda\to\Lambda$ satisfying $\phi(x)\neq x$ for every $x\in\Lambda$.
\end{exe}

\begin{example}\label{exam:rectifiability_in_R}
    Any Delone set $\Lambda\subset\R$ is BL equivalent to $\Z$. This fact follows from the natural order on $\R$, since if we enumerate the elements of $\Lambda$ as  $(x_n)_{n\in\Z}$, so that $x_n<x_{n+1}$ for every $n$, and define $\psi:\Z\to \Lambda$ by $\psi(n) = x_n$, we have 
    $r\le \frac{\norm{x_n-x_m}}{\norm{n-m}}\le 2R$. Then $L=\max\{1/r, 2R\}$ is as required. 
\end{example}

\begin{exe}\label{ex:-N_cup_2N}
    Prove that $\Lambda = (-\N)\cup2\N$ is not BD equivalent to any lattice in $\R$. 
\begin{figure}[h!]
    \centering
    \includegraphics[scale=0.5]{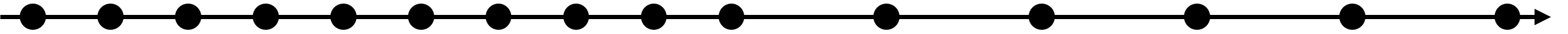}
\end{figure}
\end{exe}

\begin{prop}\label{prop:BD_implies_BL}
    For uniformly discrete sets $\Lambda_1,\Lambda_2\subset \R^d$, every BD-map from $\Lambda_1$ to $\Lambda_2$ is biLipschitz. In particular, BD equivalence implies BL equivalence for Delone sets. 
\end{prop}

\begin{proof}
    Let $\phi:\Lambda_1\to \Lambda_2$ be a bijection satisfying $\sup_{x \in \Lambda_1} \norm{\phi(x) - x} = M < \infty$. Let $r_1$ and $r_2$ denote the separation constants of $\Lambda_1$ and $\Lambda_2$, respectively.  Then for every $x\neq x'$ in $\Lambda_1$ we have 
    \[
    \max\{r_2,\norm{x-x'}-2M\}\le \norm{\phi(x)-\phi(x')}
    \le\norm{x-x'}+2M.
    \]
    On the one hand, for $L_1=\frac{2M}{r_1}+1$ we have $\norm{x-x'}+2M\le L_1\norm{x-x'}$. On the other hand, for $L_2=\frac{r_2}{2M+r_2}$ 
    we have $L_2\norm{x-x'}\le \max\{r_2,\norm{x-x'}-2M\}$. Thus, $\phi$ is $L$-biLipschitz with $L=\max\{L_1,\frac{1}{L_2}\}$, as required. 
\end{proof}

\begin{prop}
      All lattices in $\R^d$ are BL equivalent to each other.   
\end{prop}
\begin{proof}
    It suffices to prove that every lattice $\Lambda\subset\R^d$ is BL equivalent to $\Z^d$. By definition, for every such $\Lambda$ there exists an invertible $d\times d$ matrix $A$ satisfying $\Lambda = A\Z^d$. Then $A:\Z^d\to \Lambda$ is a bijection, and for every $x_1\neq x_2$ in $\Z^d$ we have  
    \[
    \frac{\norm{Ax_1-Ax_2}}{\norm{x_1-x_2}} = \frac{\norm{A(x_1-x_2)}}{\norm{x_1-x_2}} \le \norm{A}_{\op},
    \]
    where $\norm{A}_{\op}$ denotes the operator norm of $A$. To complete the proof, repeating this argument with $A^{-1}:\Lambda\to\Z^d$ yields the lower Lipschitz constant. 
\end{proof}

\subsection{Delone sets associated with tilings}

    A natural connection between discrete point sets and space-filling structures is provided by the relation between Delone sets and tilings. 

\begin{definition}\label{def:tiling}
    A \emph{tile} $T$ in $\R^d$ is a compact subset of $\mathbb{R}^d$ with nonempty interior and boundary of measure zero.
    A \emph{tiling} $\tau$ of $\R^d$ is a countable collection of tiles $\{T_i\}_{i \in I}$ that satisfies the following:
\begin{enumerate}
    \item 
    \emph{Covering:} $\displaystyle \bigcup_{i \in I} T_i = \R^d$.
    \item 
    \emph{Non-overlapping interiors:} For all $i \neq j$, $\inter(T_i) \cap \inter(T_j) = \emptyset$.
\end{enumerate} 
\end{definition}

A Delone set may be obtained from a tiling by selecting one representative point from each tile. To guarantee that such a selection can be made so that the resulting set is Delone, it is enough for the tiling to satisfy two uniform geometric bounds: (1) that every tile is contained in a ball of radius $R$, and (2) that every tile contains a ball of radius $r$. Without any further topological assumptions, these two basic hypotheses ensure that one can choose a point per tile so that the resulting set is both uniformly discrete and relatively dense, as illustrated in Figure \ref{fig:Tiling_to_Delone}.

\begin{figure}[h!]
    \centering
    \includegraphics[scale=0.22]{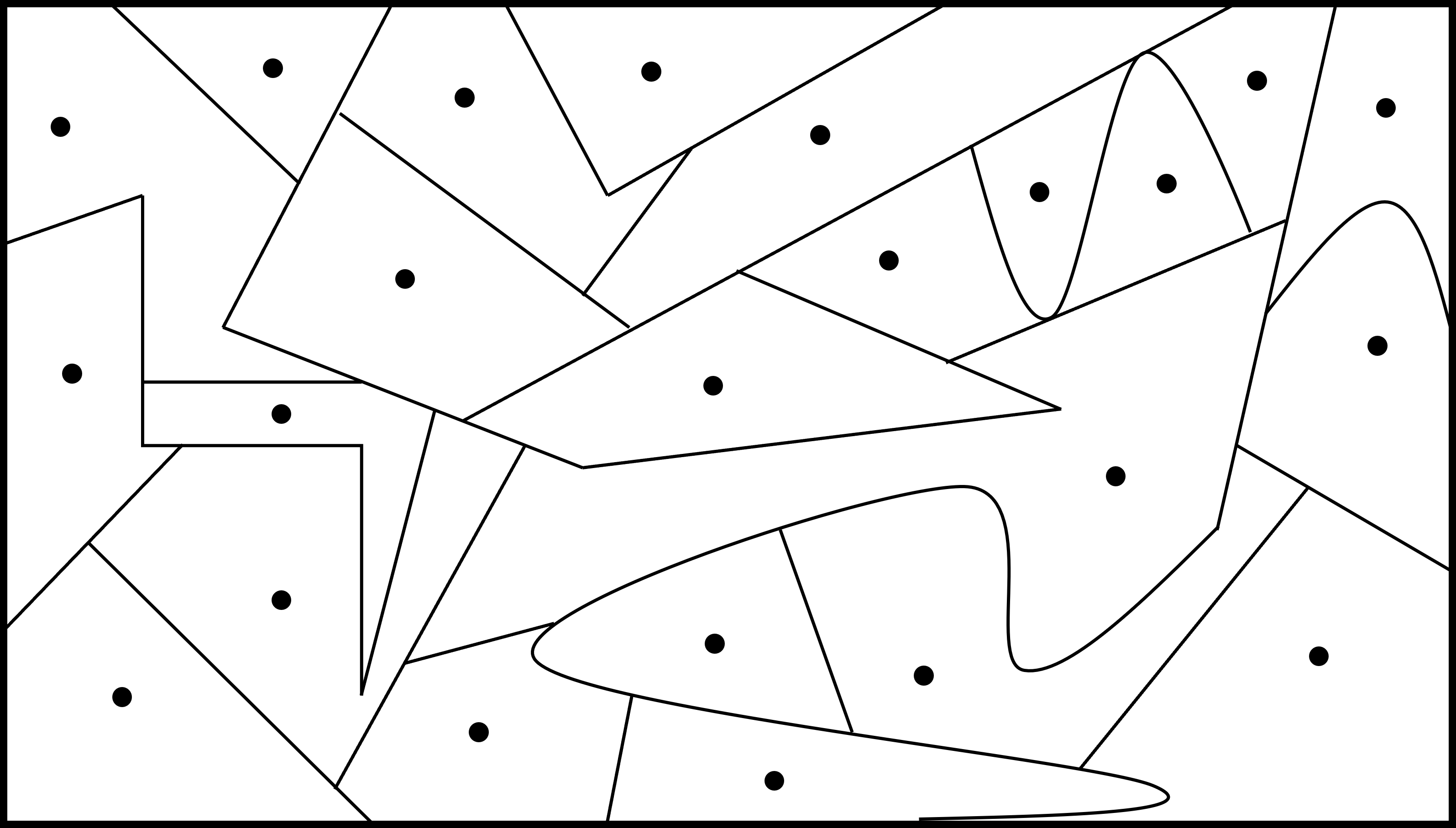}
    ~~
    \includegraphics[scale=0.22]{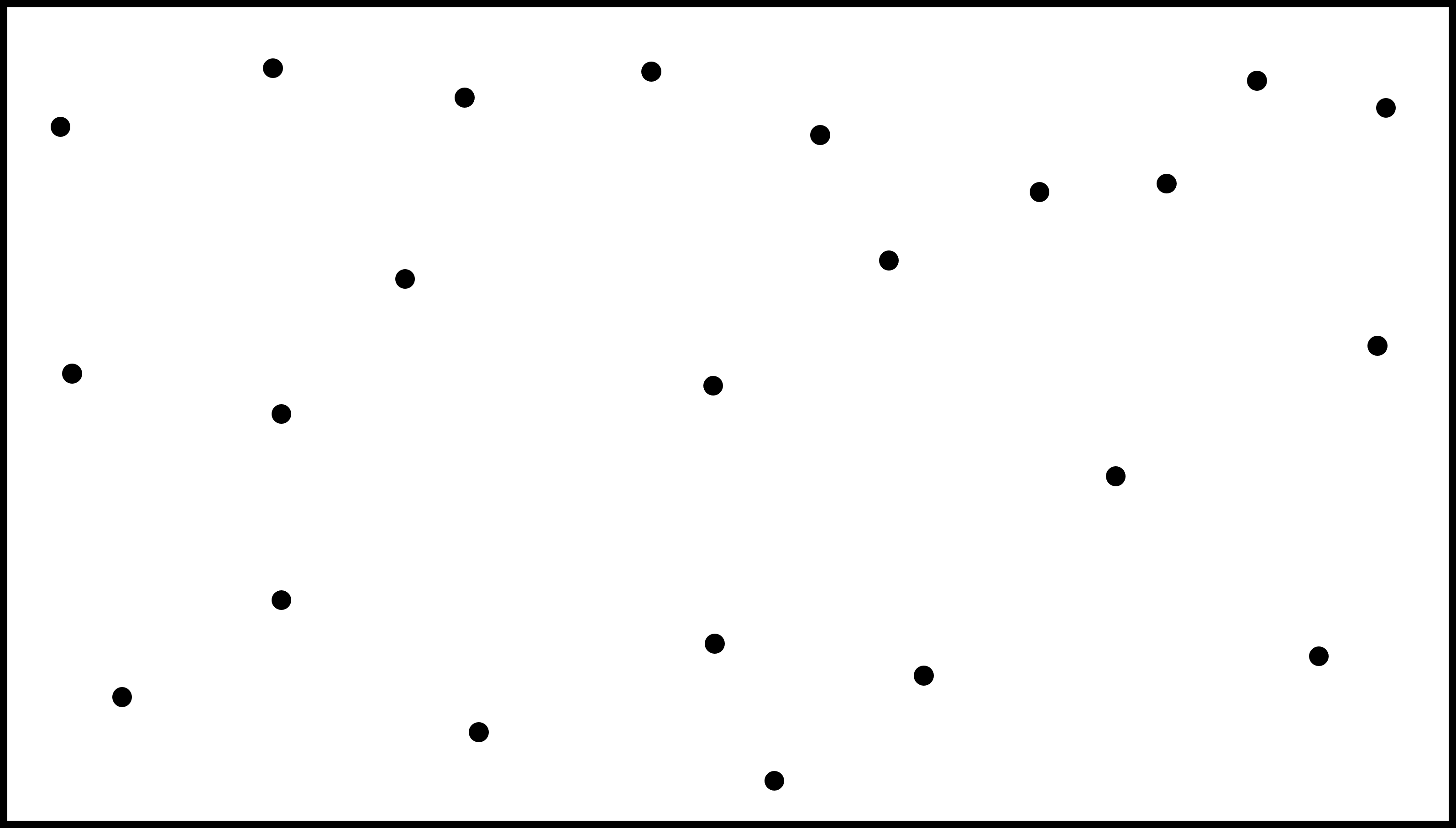}
    \caption{A tiling with a corresponding Delone set.}
    \label{fig:Tiling_to_Delone}
\end{figure}

Conversely, given a Delone set $\Lambda\subset \R^d$, one may construct the \emph{Voronoi tiling}. Each point $x\in \Lambda$ is assigned the cell $T_x$ of all points $y\in\R^d$ closer to $x$ than to any other point in $\Lambda$, namely $T_x=\{y\in\R^d \mid \forall x'\in \Lambda, \norm{x-y}\le \norm{x'-y} \}$.
\begin{figure}[h!]
    \centering
    \includegraphics[scale=0.22]{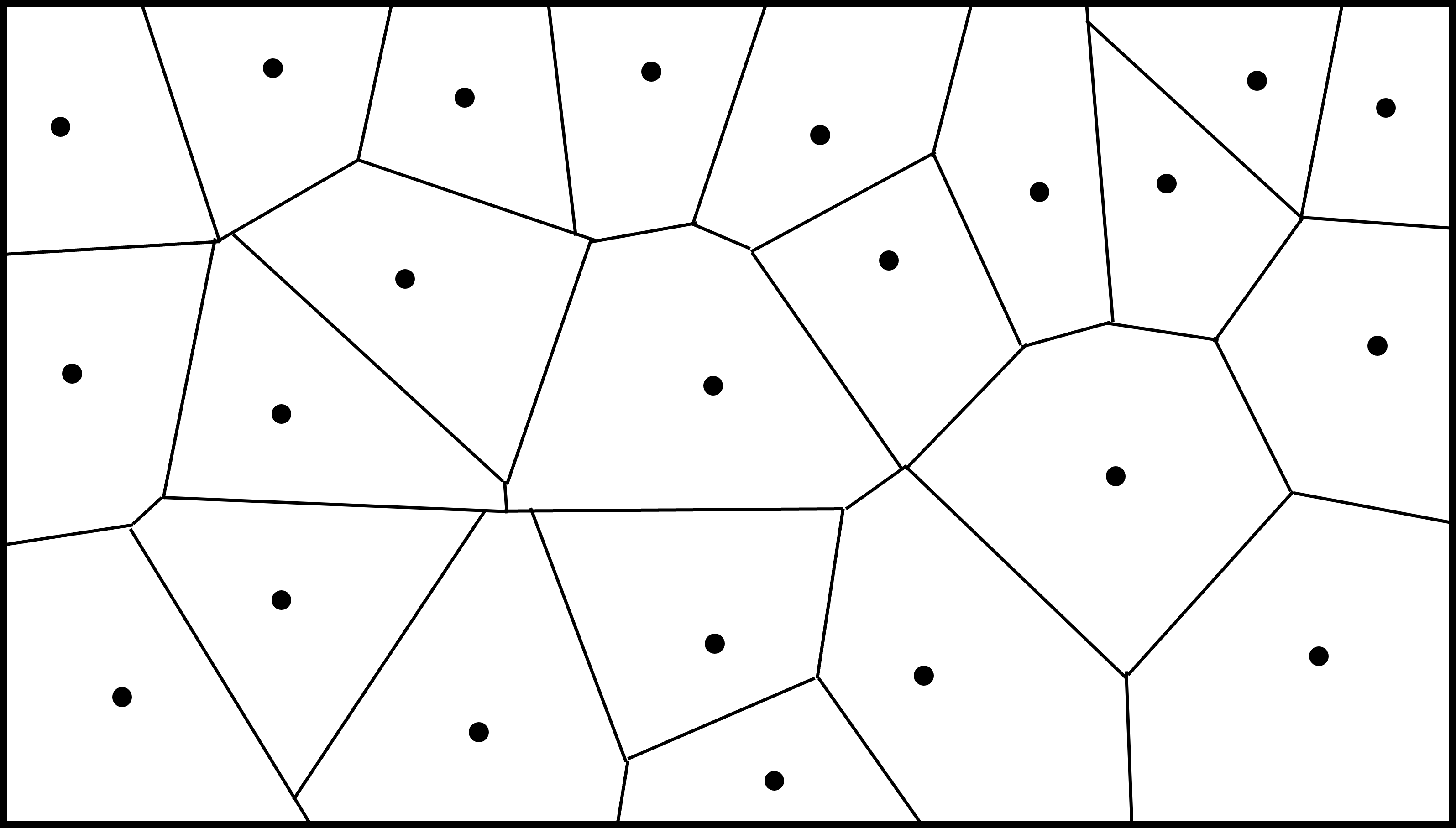}
    \caption{The Voronoi tiling of the above Delone set. 
    }
    \label{fig:Delone_to_tiling}
\end{figure}
The Voronoi tiling produces a decomposition of $\R^d$ into bounded convex polytopes, the \emph{Voronoi cells}, providing a natural geometric tiling of $\R^d$ associated with the Delone set $\Lambda$, as illustrated in Figure \ref{fig:Delone_to_tiling}.

Consider now the special case of tilings with tiles of equal volume. In the plane, examples include the square or the hexagonal tilings, Conway and Radin's pinwheel tiling \cite{Radin-1994}, the substitution tiling known as the chair tiling, as well as the hat and spectre aperiodic monotile tilings introduced by Smith et al. \cite{SmithMyersKaplanGoodman-Strauss1-2024}, \cite{SmithMyersKaplanGoodman-Strauss2-2024}. 

\begin{thm}\label{thm:same_volume_tiles}
    Let $\tau = \{T_i\}_{i \in I}$ be a tiling of $\R^d$ by tiles of equal volume $\nu>0$, and uniformly bounded diameter. Let $\Lambda_\tau = \{x_i\}_{i\in I}$ be a Delone set obtained from $\tau$ by placing one point $x_i$ in each tile $T_i$ of $\tau$. Then $\Lambda_\tau$ is BD equivalent to $\sqrt[d]{\nu}\Z^d$. 
\end{thm}

Before the proof, let us state the following infinite graph version of a theorem by Hall, which first appeared in the context of BD equivalence in the work of Laczkovich \cite{Laczkovich-1992} and is still a central tool in its study as described in \S\ref{sec:BD}. Recall that a \emph{perfect matching} in a graph $G=(V,E)$ is a set of edges $E'\subseteq E$ such that every $v\in V$ is adjacent to exactly one edge in $E'$. If $G=(V,E)$ is a \emph{bipartite graph}, that is, $V=A\uplus B$ and every $e\in E$ has a vertex in $A$ and a vertex in $B$, a perfect matching can naturally be viewed as a bijection between the two sides $A$ and $B$. For a set of vertices $U\subseteq V$ we let $N(U) = \{v\in V \mid \{v,u\}\in E\}$ denote the set of neighbors of $U$. For a proof of the following theorem, see, e.g., \cite{Diestel-2025}, \cite{Halpern-1966}, or \cite{Rado-1949}.
\begin{thm}[Hall's marriage theorem]\label{thm:Hall}
    Let $G=(A\uplus B,E)$ be a locally finite\footnote{The degree at every vertex is finite.}, infinite bipartite graph with sides $A$ and $B$. Then $G$ contains a perfect matching if and only if the following two conditions hold:
    \begin{enumerate}
        \item
        For every finite $A'\subset A$ we have $\#A'\le \#N(A')$.
        \item
        For every finite $B'\subset B$  we have $\#B'\le \#N(B')$.
    \end{enumerate}
\end{thm}

We can now prove Theorem \ref{thm:same_volume_tiles} about Delone sets associated with tilings with tiles of equal volume.

\begin{proof}[Proof of Theorem \ref{thm:same_volume_tiles}] 
    (Compare with \cite[\S 4]{BuragoKleiner-2002}, \cite[\S 4]{McMullen-1998}.)
    By scaling the elements of $\tau$ by $1/\sqrt[d]{\nu}$ we may assume that $\nu=1$. Let $\varsigma =\{S_m\}_{m\in\Z^d}$ be the tiling by unit volume cubes centered at integer lattice points, that is, the tiling with tiles $S_m = m + \left[ -\frac12, \frac12 \right]^d$ for every $m\in\Z^d$. Consider the bipartite graph $G=(\tau\uplus \varsigma,E)$, where 
    \[
    E=\{\{T_i,S_m\} \mid T_i\cap S_m\neq \emptyset\}.
    \]
    We claim that $G$ contains a perfect matching. Given a set $F$ of $k$ tiles of $\tau$, if $|N(F)|<k$, then the set $\bigcup_{T\in F}T$ is covered by less than $k$ tiles of $\varsigma$, which yields a contradiction since 
    $\vol\left(\bigcup_{T\in F}T\right) = k$. Thus $|F|\le |N(F)|$. Similarly, for every set $F$ of $k$ tiles of $\varsigma$ we have $|F|\le |N(F)|$, so by Hall's marriage theorem (Theorem \ref{thm:Hall}), $G$ contains a perfect matching $E'$. Define a function $\phi:\Lambda_\tau\to \Z^d$ by 
    \[
    \phi(x_i) = m \quad\text{ if and only if}\quad \left\{T_i,m+\left[ -\frac12, \frac12 \right]^d\right\} \in E'. 
    \]
    Since $E'$ is a perfect matching, $\phi$ is a bijection. Now denote by 
    $M=\sup_{i\in I}\{\diam(T_i)\}$. Note that if $\phi(x_i)=m$ then the sets $T_i$ and $m+\left[ -\frac12, \frac12 \right]^d$ intersect. Thus, by the triangle inequality, 
    \[
    \norm{\phi(x_i)-x_i} = \norm{m-x_i} \le M+\sqrt{d}, 
    \]
    and the proof is complete. 
\end{proof}

The following corollary was also obtained in \cite{DuneauOguey-1990} with a direct proof that does not rely on Hall's marriage theorem. Recall that the \emph{covolume} of a lattice in $\R^d$ is the volume of any of its fundamental domains. 

\begin{cor}\label{cor:lattices_BD}
     All lattices of equal covolume in $\R^d$ are BD equivalent. Furthermore, every $d$-periodic Delone set in $\R^d$ is BD equivalent to a lattice. 
\end{cor}

\begin{proof}
    By scaling, it suffices to show that every lattice of covolume 1 in $\R^d$ is BD equivalent to $\Z^d$, but this follows directly from Theorem \ref{thm:same_volume_tiles}.  
    
    Consider now a $d$-periodic Delone set $X\subset\R^d$ and let $(v_1,\ldots,v_d)$ denote $d$ linearly independent periods of $X$. By translating, we may assume that $0\in X$. Let $\Lambda = \spn_\Z(v_1,\ldots,v_d)$ be the lattice spanned by $(v_1,\ldots,v_d)$ and let $D$ be a parallelepiped fundamental domain for $\Lambda$ (i.e. every $y\in\R^d$ can be written uniquely as $y=\lambda+a$ with $\lambda\in \Lambda$ and $a\in D$). Since $X$ is $\Lambda$-periodic, the sets 
    \[
    \{X\cap(\lambda+D) \mid \lambda\in\Lambda\}
    \]
    differ by translations. Setting $\#(X\cap D)=s$, one easily shows that $X$ is BD equivalent to $\Lambda_s:=\spn_\Z\{\frac{v_1}{s},v_2,\ldots,v_d\}$, as required. 
\end{proof}

\section{Geometric and dynamical properties of Delone sets and tilings}\label{sec:FLC_topology_etc}

The investigation of Delone sets lies at the intersection of geometry and dynamics. On the one hand, their geometric structure, which is manifested in the arrangement of points, distances, and local configurations, captures local order and regularity. On the other hand, the dynamical viewpoint, obtained by considering the translation action on the space of Delone sets, provides a natural framework for studying how geometric organization relates to long-range order, see e.g. 
\cite{LagariasPleasants-2003}, \cite{LeeMoodySolomyak-2002}. Understanding the interplay between these two perspectives reveals deep connections between local geometric constraints and global dynamical behavior. We will now introduce several fundamental properties of Delone sets that will play an important role in subsequent sections.

\subsection{Distribution of points and patches}
A natural property to examine in a point set is the existence of a well-defined density. This can be defined with respect to different types of averaging sets. We focus here on two notions of density. 

\begin{definition}\label{def:central_density}
    A point set $\Lambda\subset\R^d$ is said to have \emph{central density} $\alpha$ if the limit 
    \[
    \lim_{R\to\infty} \frac{\#\left(\Lambda\cap \left([-R/2,R/2]^d\right)\right)}{R^d} 
    \]
    exists and is equal to $\alpha$.
\end{definition}

\begin{exe}
    Prove that if $\Lambda$ has central density $\alpha$, then for every $x\in\R^d$ we have 
    \[
    \lim_{R\to\infty} \frac{\#(\Lambda \cap (x+[-R/2,R/2]^d))}{R^d} = \alpha.
    \]
\end{exe}

A more general notion of density requires the following definition of averaging sets. 

\begin{definition}
    A \emph{van Hove} sequence in $\R^d$ is a sequence of bounded measurable subsets $(F_n)_{n\in\N}$ of $\R^d$ so that 
    \[
\lim_{n\rightarrow\infty}\frac{\vol((\partial F_n)^{+\varepsilon})}{\vol(F_n)}=0
    \]
    for all $\varepsilon>0$, where \begin{equation}\label{eq:eps_neighborhood_of_sets}
    F^{+\varepsilon}=\{y\in\R^d \mid \dist(y,F)< \varepsilon\}=\bigcup_{x\in F}B(x,\varepsilon)
    \end{equation}
    is the $\varepsilon$-neighborhood of a set $F\subset\R^d$.
\end{definition}

\begin{definition}
    A point set $\Lambda\subset\R^d$ is said to have \emph{asymptotic density} $\alpha$ if the limit 
    \[
    \lim_{n\to\infty} \frac{\#\left(\Lambda\cap F_n\right)}{\vol(F_n)}
    \]
    exists and is equal to $\alpha$ for any van Hove sequence $(F_n)_{n\in\N}$ in $\R^d$.
\end{definition}

Clearly, the existence of asymptotic density is stronger than having central density. For example, the set of primitive lattice points in $\Z^d$ is well-known to have central density $1/\zeta(d)$, where $\zeta$ is the Riemann zeta function, but does not have an asymptotic density as the set is not relatively dense. 

\begin{exe}
    Compute the central density of the set $\Lambda=(-\N)\cup 2\N$ from Exercise \ref{ex:-N_cup_2N} and show that it does not have an asymptotic density.
\end{exe}

Given a Delone set $\Lambda\subset\R^d$, any finite subset of $\Lambda$ is called a \emph{patch}. Every patch in $\Lambda$ is of the form $\Lambda\cap K$ for some compact set $K$. Similarly, if $\tau$ is a tiling of $\R^d$, we also use the term patch to refer to any set of tiles of $\tau$ that intersect some compact set $K$. The \emph{support} of a patch $P$ is the subset of $\R^d$ that is the union of all tiles in $P$, and we denote it by $\supp(P)$. We use the term \emph{$r$-patch} to refer to patches for which $K$ is a ball of radius $r>0$ centered around some point in $\R^d$.
Two patches $P$ and $Q$, either in a Delone set or in a tiling, are \emph{translation equivalent} if there exists some $v\in\R^d$ such that $P=Q+v$.  
 
\begin{definition}
    A Delone set $\Lambda \subset \R^d$ has \emph{finite local complexity (FLC)} if for every $r>0$ the collection of $r$-patches is finite, up to translation equivalence. 
\end{definition}

The notion of FLC for a tiling $\tau$ is defined similarly. Furthermore, a tiling $\tau$ has FLC if and only if it has finitely many translation-equivalence classes of tiles, called \emph{prototiles}, and every two prototiles can be positioned next to each other in only finitely many ways. Since our main focus is on Delone sets, we define the notions below in detail for Delone sets and remark here that parallel definitions can be given for tilings. 

\begin{example}
Any lattice in $\R^d$ has FLC, and so does any translated lattice. In fact, any subset of a translated lattice has FLC, even if it fails to be Delone, see for example the set of primitive lattice vectors mentioned in \S\ref{sec:introduction} and the set $\Lambda= (-\N)\cup2\N$ from Exercise \ref{ex:-N_cup_2N}.  
    
On the other hand, the following examples do not have FLC:
\begin{itemize}
    \item 
    $\Lambda = \{n+\frac{1}{n} \mid n\in\Z\minus\{0\}\}$.
    \item 
    $\Lambda= \left\{\binom{x}{y} \mid x,y\in \Z, x<0\right\} \cup \left\{\binom{x}{\alpha y} \mid x,y\in \Z, x\ge 0 \right\}$ for $\alpha\in \R\minus\Q$.   
\end{itemize}
\end{example}

\begin{exe}
    Show that any Delone set $\Lambda\subset\R^d$ is BD equivalent to an FLC Delone set.
\end{exe}
\begin{exe}
    Show that any tiling with finitely many interval prototiles in $\R$ has FLC. Find a counterexample for a polygonal tiling in $\R^2$.
\end{exe}

In view of the discussion about tilings and Delone sets in \S\ref{sec:introduction}, one can deduce that when studying BD equivalence, it suffices to consider Delone sets that arise from tilings. The following exercise demonstrates that it is even sufficient to consider only Delone sets that arise from FLC tilings. A particular method to construct tilings by finitely many tiles is substitution tilings, which is discussed in \S \ref{sec:substitution}. 

\begin{exe}\label{ex:finitely_many_tiles}
     Refine the Voronoi construction and show that every Delone set in $\R^d$ arises from an FLC tiling, which in particular has finitely many tiles, up to translation (Hint: equation \eqref{eq:Voronoi-type}). 
\end{exe}

Having established the notion of finite local complexity, which ensures that a Delone set possesses only finitely many distinct local configurations up to translation equivalence, we now turn to a stronger property that concerns the spatial distribution of these patterns and their recurrence. 

\begin{definition}\label{def:repetitivity}
    A Delone set $\Lambda\subset\R^d$ is \emph{repetitive} if for every $r>0$ there exists an $R=R(r)>0$ such that every $R$-patch contains a translated copy of every $r$-patch of $\Lambda$. The function $R(r):(0,\infty)\to(0,\infty)$ is called the \emph{repetitivity function} of $\Lambda$.  
\end{definition}

\begin{exe}
    Show that repetitivity of a Delone set implies finite local complexity.
\end{exe}

\begin{exe}\label{ex:repetitive_Delone_sets}
Are the following Delone sets $\Lambda \subset \Z^2$ repetitive?
    \begin{enumerate}
        \item 
        $\Lambda = \{(10m+r_m, 10n+r_n) ~\mid~ m,n\in\Z\}$, where $r_m\in\{0,1,\ldots,9\}$ denotes the remainder upon dividing $m$ by $10$.
        \item 
        For each $(m,n)\in 3\Z\times 3\Z$, independently pick a point in $(m,n)+\{0,1,2\}^2$, uniformly at random, and let $\Lambda$ be the set of points that were chosen. Is it true that $\Lambda$ is repetitive with probability $1$?
    \end{enumerate}
\end{exe}

Notice that the repetitivity function clearly satisfies $R(r)\ge r$. A particular focus will be on Delone sets whose repetitivity function $R(r)$ grows slowly, and in particular at most linearly in $r$, since this behavior reflects a high degree of uniformity and structural regularity. Such Delone sets are called \emph{linearly repetitive} and are studied in detail in \cite{Aliste-PrietoCoronelCortezDurandPetite-2015}. They will be examined in more detail in the following sections. 

Repetitivity ensures that every local configuration in a Delone set reappears infinitely often, but it does not quantify how frequently these patterns occur. To capture this notion, we introduce the stronger concept of uniform patch frequency, which measures the asymptotic density of appearances of each finite patch.

\begin{definition}
    Let $\Lambda\subset\R^d$ be a Delone set. We say that $\Lambda$ has \emph{uniform patch frequency (UPF)} if every patch in $\Lambda$ occurs with asymptotic density. 
\end{definition}

\subsection{Spaces of Delone sets}\label{subsec:spaces_of_Delone_sets}
As in many areas of mathematics, valuable insight can be gained by studying not only individual objects but also the space they form. In the case of Delone sets, considering the space of all such sets equipped with an appropriate topology provides a natural framework for understanding and relating their geometric and dynamical properties. Since Delone sets are closed subsets of $\R^d$, we endow the space of all Delone sets with the \emph{Chabauty--Fell topology}. This topology was introduced by Chabauty in \cite{Chabauty-1950} and further studied by Fell in \cite{Fell-1962}. See also \cite{LenzStollmann-2003} for its relation to the Hausdorff metric, which is recalled in the proposition below.

The Chabauty--Fell topology is induced by the metric given below. Roughly speaking, two sets are close if, after intersecting each with a large ball around the origin, the resulting compact sets are close in the Hausdorff metric. For spaces of FLC Delone sets, this topology coincides with the \emph{local rubber topology}, which is more commonly used in the context of aperiodic order (see, for example, \cite{Robinson-2004}, \cite{Rudolph-1988}, and \cite{Solomyak-1997}).

We denote by $\CC(\R^d)$ the collection of all closed subsets of $\R^d$. 
We define a metric on $\CC(\R^d)$ as follows. For $F_1,F_2\in\CC(\R^d)$: 
\begin{equation}\label{eq:metric}
    D(F_1,F_2) = 
\inf\left(\left\{ \varepsilon>0 \: \Bigg| \: \begin{matrix}F_2 \cap B(0,1/\varepsilon)\subset F_1^{+\varepsilon}\\
F_1 \cap B(0,1/\varepsilon)\subset F_2^{+\varepsilon}
\end{matrix}\right\} \cup\{1\}\right),
\end{equation}
where $F^{+\varepsilon}$ is the $\varepsilon$-neighborhood of the set $F$, as in \eqref{eq:eps_neighborhood_of_sets}. 
\begin{prop}
    The function $D(\cdot,\cdot)$ is a metric on $\CC(\R^d)$ and the metric space $(\CC(\R^d),D)$ is compact.
\end{prop}

\begin{proof}[Sketch of proof]
    A proof of the fact that $D$ is indeed a metric on $\CC(\R^d)$ can be found in \cite[Appendix~A]{SmilanskySolomon1-2022}. The compactness of the space follows from a standard diagonalization argument. 
    First, notice that for every compact set $K\subset\R^d$, the topology induced by $D$ on the space $\CC(K)$ of closed subsets of $K$ coincides with the topology induced by the Hausdorff metric, and therefore, $(\CC(K),D|_{\CC(K)})$ is compact. 
    Next, given a sequence $(F_m)_{m\in\N}\subset \CC(\R^d)$, we construct a convergent subsequence. Let $K_n = \overline{B(0,n)}$ be the closed ball of radius $n$ around the origin of $\R^d$. \begin{itemize}
        \item 
        The sequence $(F_m\cap K_1)$ belongs to the compact space $\CC(K_1)$ and thus admits a subsequence $(F_{m,1})$ converging to a limit $L_1\in \CC(K_1)$.  
        \item 
        The sequence $(F_{m,1}\cap K_2)$ belongs to the compact space $\CC(K_2)$, thus admits a 
        subsequence $(F_{m,2})$ converging to a limit $L_2\in \CC(K_2)$.
        \\ \vdots \\
        We continue in that manner to obtain the following sequence of sequences:
        \[
        \begin{matrix}
            F_{1,1} & F_{2,1} &F_{3,1}& F_{4,1} &F_{5,1}& F_{6,1} &\cdots\\
            F_{1,2} & F_{2,2} &F_{3,2}& F_{4,2} &F_{5,2}& F_{6,2} &\cdots\\
            F_{1,3} & F_{2,3} &F_{3,3}& F_{4,3} &F_{5,3}& F_{6,3} &\cdots\\
            F_{1,4} & F_{2,4} &F_{3,4}& F_{4,4} &F_{5,4}& F_{6,4} &\cdots\\
            F_{1,5} & F_{2,5} &F_{3,5}& F_{4,5} &F_{5,5}& F_{6,5} &\cdots\\
            F_{1,6} & F_{2,6} &F_{3,6}& F_{4,6} &F_{5,6}& F_{6,6} &\cdots\\
            &&\vdots&&\vdots&&
        \end{matrix}
        \]
    \end{itemize}  
    Clearly, $L_n\subseteq L_{n+1}$ for every $n\ge 1$. It remains to verify that the sequence $(F_{n,n})$ converges in $(\CC(\R^d),D)$ to the set $L = \bigcup_{n=1}^\infty L_n$. 
\end{proof}

\begin{exe}
    Verify that ineed $L \in \CC(\R^d)$ and $\lim_{n\to\infty}F_{n,n} = L$.  
\end{exe}

\begin{exe}
    Construct a Delone set $\Z^2\neq \Lambda \subset \R^2$ such that for every sequence $v_n\in\Z^2$ satisfying $\|v_n\|\rightarrow \infty$ as $n\to \infty$, we have 
    \[
    \lim_{n\to\infty} \Lambda + v_n = \Z^2.
    \]
    Find examples that, in addition, satisfy either
    \begin{enumerate}
        \item 
        $\Lambda$ contains a translated copy of every patch in $\Z^2$, or 
        \item 
        $\Lambda$ contains no translated copy of any patch with at least $2$ points of $\Z^2$.
    \end{enumerate}
    Furthermore, show that there is no closed set $\Lambda\subset\R^2$ and
        no $0\neq v\in\R^2$ such that
        \[
        \lim_{t\to+\infty}(\Lambda+tv)=\Z^2.
        \]
\end{exe}

Having introduced the metric, we now study the dynamics of Delone sets under the action of the group $\R^d$ by translations, noting that translating a Delone set always produces another Delone set.

\begin{definition}\label{def:hull}
    Let $\Lambda\subset\R^d$ be a Delone set. The \emph{hull} of $\Lambda$ is the orbit closure 
    \[
    X_\Lambda = \overline{\{\Lambda+t~\mid~ t\in\R^d\}},
    \]
    with respect to the action of $\R^d$ by translations and the above topology. 
\end{definition}

The hull $X_\Lambda$ of a Delone set 
$\Lambda$ provides a natural compact space on which $\R^d$ acts, allowing the study of $\Lambda$ from a dynamical perspective. In this framework, geometric properties of $\Lambda$ are reflected in the dynamical behavior of the hull. This connection is exemplified in the theorem below.

Recall that a topological dynamical system is a pair $(X,G)$ that consists of a Hausdorff compact space $X$ and a topological group $G$ that acts on it. A subsystem is a set $\emptyset\neq X'\subset X$, which is closed and invariant under the action of $G$. The system, or $X$, is called \emph{minimal} if it contains no proper subsystems. Equivalently, $X$ is minimal if, for every $x\in X$, the orbit closure $\overline{G.x}$ is equal to $X$. The reader is referred to \cite{Furstenberg-1981} for a more comprehensive treatment of this topic. Furthermore, $X$ is called \emph{uniquely ergodic} if there exists a unique $G$-invariant Borel probability measure on $X$. 

\begin{thm}\label{thm:minimal_uniquely-ergodic}
    Let $\Lambda\subset\R^d$ be an FLC Delone set. Then 
    \begin{enumerate}
        \item\label{thm_item:minimal+uni.erg.1}
        $\Lambda$ is repetitive if and only if $X_\Lambda$ is minimal.         \item\label{thm_item:minimal+uni.erg.2}
        $\Lambda$ has uniform patch frequency if and only if $X_\Lambda$ is uniquely ergodic.        
    \end{enumerate}
\end{thm}
Assertion \ref{thm_item:minimal+uni.erg.1} follows from general topological dynamics, see for example \cite[\S 1.4]{Furstenberg-1981}, and was proved in the context of Delone sets in \cite[Theorem~3.2]{LagariasPleasants-2003}. A proof of assertion \ref{thm_item:minimal+uni.erg.2} can be found in  \cite[Theorem~2.7]{LeeMoodySolomyak-2002}. For related notions and results in the non-FLC setup, see \cite{FrettlohRichard-2014}.

\begin{exe}
    Use Zorn's lemma to prove that every compact topological
    dynamical system contains a minimal subsystem. Deduce that, for
    every \emph{FLC} Delone set $\Lambda\subset\R^d$, its hull
    $X_\Lambda$ contains a repetitive Delone set $\Lambda'$.
\end{exe}

Without the FLC assumption, the same topological argument produces a minimal subsystem, but its elements need only be almost repetitive (see definitions in \cite{FrettlohRichard-2014} and \cite{SmilanskySolomon2-2022}) rather than repetitive in the exact-patch sense of Definition
\ref{def:repetitivity}.

\section{BL equivalence}\label{sec:BL}
The study of biLipschitz equivalence of Delone sets is rooted in several distinct areas of mathematics. The question of the existence of Delone sets that are not biLipschitz equivalent to a lattice was posed independently by Furstenberg in the 1960s, as mentioned in \cite{BuragoKleiner-1998}, and by Gromov in \cite{Gromov-1993}. Gromov's motivation was in the context of geometric group theory and quasi-isometries: since quasi-isometric metric spaces contain biLipschitz-equivalent Delone sets, he asked whether every Delone set in $\R^d$ is biLipschitz equivalent to the standard lattice $\Z^d$. An affirmative answer would imply that all large-scale geometries of $\R^d$ induced by Delone sets are essentially the same. Furstenberg, by contrast, was motivated by ergodic theory, specifically in relation to Kakutani equivalence for $\R^d$-actions. In this context, return times to a transversal section form a Delone set, and Furstenberg wondered whether such sets could always have a biLipschitz matching to $\Z^d$, enabling a representation of the action in terms of a discrete system. 

The question was resolved negatively by Burago and Kleiner \cite{BuragoKleiner-1998}, and independently by McMullen \cite{McMullen-1998}, who constructed explicit examples of Delone sets in $\R^2$ that are not biLipschitz equivalent to $\Z^2$. Their method was rooted in geometric analysis: they associated with a Delone set a Voronoi-type tiling and a piecewise constant function $u$ reflecting local density, then showed that in some cases $u$ cannot be realized (almost everywhere) as the Jacobian of a biLipschitz homeomorphism $\Phi:\R^2 \to \R^2$. This analysis reduced the biLipschitz equivalence problem to the prescribed Jacobian problem, linking the rectifiability of a Delone set to whether its associated density function meets strong analytic criteria. While these first counterexamples were constructed by locally breaking every biLipschitz constant and had no global structure, later work by Cortez and Navas \cite{CortezNavas-2016} introduced repetitive Delone sets with uniform patch frequency that are non-rectifiable, refining the Burago--Kleiner construction to preserve global regularity. These properties add a strong form of statistical regularity to the set, yet the Cortez-Navas examples still fail to be biLipschitz equivalent to a lattice, underscoring the subtlety of the equivalence relation.

A parallel and ongoing line of research has examined Delone sets arising from constructions in aperiodic order, such as substitution tilings, cut-and-project sets, and others, aiming to find a non-rectifiable Delone set defined using finite data. In these settings, one typically starts from an aperiodic but rule-based set of tiles or points and then analyzes whether the associated Delone sets are rectifiable. For instance, it has been shown that Delone sets arising from primitive substitution tilings (including Penrose tilings) are always rectifiable \cite{Solomon-2011}. This result, as well as related results on other families of constructions, will be discussed further in \S \ref{sec:substitution}, \S \ref{sec:multiscale}, and \S\ref{sec:C&P}.

\subsection{The existence of non-rectifiable Delone sets}\label{subsec:non-rectifiable sets}
We begin this section by presenting the connection between the rectifiability of Delone sets and the problem of realizing a given function as the Jacobian determinant of a biLipschitz homeomorphism of $\R^d$, a perspective developed in the foundational works of McMullen \cite{McMullen-1998} and of Burago and Kleiner \cite{BuragoKleiner-1998}.

\begin{thm}\label{thm:Jac_implies_BL} 
The following are equivalent:
\begin{enumerate}
    \item 
    Every positive measurable function on $\R^d$, with $f$ and $1/f$ bounded, can be realized a.e. as the Jacobian of a biLipschitz homeomorphism  $\Phi:\R^d\to\R^d$.
    \item 
    Every Delone set $\Lambda\subset\R^d$ is biLipschitz equivalent to $\Z^d$. 
\end{enumerate}
\end{thm}

This theorem appears in \cite[Theorem~4.1]{McMullen-1998}. We include here the proof of the first implication, which in particular shows that the rectifiability of a Delone set $\Lambda$ would follow if a certain density function that depends on $\Lambda$ can be realized a.e. as the Jacobian of a biLipschitz homeomorphism of $\R^d$.

\begin{proof}[Proof of (1) implies (2):]
     Let $\Lambda\subset\R^d$ be a Delone set with separation and covering constants $0< r\le R$. Consider the Voronoi tiling of $\Lambda$ (see Figure \ref{fig:Delone_to_tiling}), defined by $\tau_\Lambda = \{T_x\}_{x\in \Lambda}$ where 
    \begin{equation}\label{eq:Voronoi_tiling}
        T_x=\{y\in\R^d \mid \forall x'\in \Lambda, \norm{x-y} \le \norm{x'-y}\}.
    \end{equation}

Define a function $f_\Lambda:\R^d\to\R$ by 
\begin{equation}\label{eq:f_Lambda}
f_\Lambda(y)=\begin{cases}
    \frac{1}{\vol(T_x)},& \exists x\in \Lambda, y\in\inter(T_x). \\
    1,& \text{otherwise}.
\end{cases}  
\end{equation}
Each tile $T_x$ is a compact, convex set that satisfies $B(x,r/2)\subset T_x\subset B(x,R)$, thus the hypothesis on $f$ in (1) is satisfied. Then by assumption, there exists a biLipschitz homeomorphism $\Phi:\R^d\to\R^d$ such that for a.e. $y\in \R^d$ we have 
\[
 f_\Lambda(y)= \Jac(\Phi)(y),
\]
where $\Jac(\Phi) := \det(D\Phi)$ denotes the Jacobian determinant of $\Phi$. Then, by the theory of multivariate integration, for every $x\in \Lambda$ we have 
\[
\vol(\Phi(T_x)) = \int_{\Phi(T_x)}1=\int_{T_x}(1\circ\Phi)\cdot \Jac(\Phi)=\vol(T_x)\frac{1}{\vol(T_x)} = 1.
\]
That is, $\Phi$ maps the tiling $\tau_\Lambda$ to a new tiling $\mathcal{S}$ of $\R^d$ by tiles of volume 1. Since $\Phi$ is a biLipschitz map, the set $\Phi(\Lambda)$ is a Delone set, and it consists of a single point in each tile of $\mathcal{S}$, thus, by Theorem \ref{thm:same_volume_tiles}, $\Phi(\Lambda)$ is BD equivalent to $\Z^d$, and hence BL equivalent to $\Z^d$ by Proposition \ref{prop:BD_implies_BL}. Clearly, $\Lambda$ is BL equivalent to $\Phi(\Lambda)$ and therefore $\Lambda$ is BL equivalent to $\Z^d$.  
\end{proof}

Notice that this proof establishes a stronger assertion. Given a Delone set $\Lambda$, there is a specific function $f_\Lambda$, so that $\Lambda$ is BL equivalent to $\Z^d$ if the function $f_\Lambda$ can be realized a.e. as the Jacobian of a biLipschitz homeomorphism of $\R^d$. 

\begin{cor}\label{cor:Jac_implies_BL}
    Let $\Lambda$ be a Delone set. If $f_\Lambda$ from \eqref{eq:f_Lambda} can be realized a.e. as the Jacobian of a biLipschitz homeomorphism of $\R^d$, then $\Lambda$ is BL equivalent to $\Z^d$.
\end{cor}

The proof of (2) implies (1) is given in more detail in \cite[Lemma~2.1]{BuragoKleiner-1998}. The idea is as follows. 
\begin{proof}[Sketch of proof of (2) implies (1):]
    For simplicity of presentation, we assume $d=2$. Suppose that $\rho:[0,1]^2 \to \R_{>0}$ cannot be realized a.e. as the Jacobian of a biLipschitz homeomorphism. Choose a sequence of increasingly large axis‑aligned squares $S_k$. In each square $S_k$, encode the forbidden density $\rho$ by rescaling it to that square and subdividing $S_k$ into many tiny subsquares whose number in each cell closely matches the local average of $\rho^{-1}$. Placing exactly one point at the center of each tiny cell, and placing points of $\Z^2$ outside of $\bigcup_kS_k$, produces a Delone set $\Lambda \subset \mathbb{R}^2$. If one assumes a biLipschitz map $g:\Lambda \to \Z^2$ exists, then after rescaling and passing to a subsequence along the sequence $(S_k)$, these maps converge to a biLipschitz map $\Phi:[0,1]^2 \to \R^2$ whose Jacobian must be exactly $\rho$, contradicting the hypothesis that no such $\Phi$ exists. 
\end{proof}

In view of Theorem \ref{thm:Jac_implies_BL}, the existence of non-rectifiable Delone sets in $\R^2$ follows from the theorem below (compare \cite[\S 3]{McMullen-1998},\cite[Theorem~1.2]{BuragoKleiner-1998}).

\begin{thm}\label{thm:non-Jac_function}
    Fix $\delta>0$. There exists a measurable function $f:[0,1]^2\to[1-\delta,1+\delta]$, such that for any biLipschitz homeomorphism $\Phi:[0,1]^2\to\R^2
    $ the equation $f= \Jac(\Phi)$ a.e. does not hold.
\end{thm}

We sketch the idea of McMullen's construction from \cite{McMullen-1998}. A similar construction of an unattainable density is described in \cite{BuragoKleiner-1998}, see also the related discussion in \cite{DymondKaluzaKopecka-2018}.

\begin{proof}[Sketch of proof of Theorem \ref{thm:non-Jac_function}]
    Set $I=[0,1]^2$ and let $J$ be the middle third square in $I$. Let $f_0:I\to\R_{\ge 0}$ be a function that is constant on both $J$ and $I\minus J$ and has $\int_J f_0=1-\delta$ and $\int_I f_0 = 1$. Cover the inside neighborhood of $\partial I$ by smaller squares $I_k$, with affine maps $h_k:I_k\to I$, and define 
    \[
    f_1(x)=\begin{cases}
        f_0\circ h_k(x),& x\in I_k.\\
        f_0(x), & \text{otherwise}.
    \end{cases}
    \]
    Repeat this construction, covering the edges of the small squares by even smaller squares, and iterate to define a sequence of functions $(f_n)_{n\in\N}$ on $I$. 

    Let  
    $f=\lim_{n\rightarrow\infty}f_n$. Then $f$ exists a.e. and is bounded from above and below. Assume, for the sake of contradiction, that $f=\Jac(\Phi)$ a.e. for some biLipschitz homeomorphism $\Phi:I\rightarrow\R^2$. Let $L=\sup\frac{\absolute{\Phi(b)-\Phi(a)}}{\absolute{b-a}}$, where the supremum is taken over all edges $[a,b]$ of squares of all levels that appear in the process. For simplicity, assume that the supremum is achieved on some edge $[a,b]$ of a square $I'$ of level $n$, and assume $[a,b]$ is a horizontal edge. Let $I'_k$ denote the small squares of level $n+1$ that cover the edge $[a,b]$, and denote their union by $R$. Since the area of $R$ is equal to the area of $\Phi(R)$, $\Phi$ approximately maps $R$ to a rectangle with a long edge stretched by $L$ and a short edge compressed by $1/L$. This implies that the perimeter of some small square $I'_k$ is stretched by at least $L/2$. By our construction of $f$ and our assumption that it is the Jacobian of $\Phi$, if $J'_k$ is the middle third square of $I'_k$, then the area of $\Phi(J'_k)$ occupies most of the area of $\Phi(I'_k)$. Since the perimeter of $J'_k$ is a third of the perimeter of $I'_k$, we deduce that it must be stretched by approximately $3L/2$, contradicting the definition of $L$.  
\end{proof}

A similar argument would work for any dimension $d\geq2$. As a consequence, we obtain the following theorem, which resolves the aforementioned open question of Furstenberg and Gromov. 

\begin{thm}
    There exist Delone sets in $\R^d$, $d\ge 2$, which are non-rectifiable. 
\end{thm}

Building on the methods above, Magazinov \cite{Magazinov-2011} showed that for every $d\ge 2$, the set of biLipschitz equivalence classes of Delone sets in $\R^d$ has the cardinality of the continuum. 
Cortez and Navas \cite{CortezNavas-2016} provided a new layer of structure by explicitly constructing non-rectifiable Delone subsets of $\Z^2$ that are repetitive and have uniform patch frequency (see also \cite{BhatDymond-2025} for more about repetitivity rates in non-rectifiable sets). Equivalently, the orbit closures of these Delone sets are minimal and uniquely ergodic (see Theorem \ref{thm:minimal_uniquely-ergodic}). Their method adapts the Burago--Kleiner construction of an unattainable density to the discrete setting, yielding combinatorially explicit, uniquely ergodic examples that are nonetheless not BL equivalent to a lattice. The Cortez-Navas example has the additional property that every Delone set in its hull is non-rectifiable, see \cite[Remark~14]{CortezNavas-2016}. In particular, it is not true that the hull of every Delone set contains a rectifiable Delone set.

\subsection{The Burago--Kleiner sufficient condition for rectifiability}\label{subsec:BK_condition}
In a later work \cite{BuragoKleiner-2002}, Burago and Kleiner provided 
a sufficient condition under which a Delone set in $\R^2$ is BL equivalent to a lattice. This condition, which is given in terms of the rate of convergence of the point-counting discrepancies along expanding Euclidean squares, was later extended in \cite{Aliste-PrietoCoronelGambaudo-2013} to arbitrary dimension $d\ge 2$. 

\begin{definition}\label{def:discrepancy}
    Given a discrete set $\Lambda \subset \R^d$, a measurable set $B \subset \R^d$, and a parameter $\alpha > 0$, the \emph{discrepancy} of $\Lambda$ with respect to $B$ and $\alpha$ is the quantity 
\[
\disc{\alpha}{\Lambda}{B} = \absolute{ \#(\Lambda \cap B) - \alpha \cdot \vol(B) }.
\]
\end{definition}

We let $Q(x,R)\subset\R^d$ denote the closed cube $x+[0,R]^d$ of edge-length $R$, and note that it has volume $R^d$. For a Delone set $\Lambda \subset \R^d$ and for $\alpha, R>0$ we set 
    \begin{equation}\label{eq:BK_discrepancy}
    \Delta_\Lambda(\alpha,R) = \sup_{x\in R\Z^d}\left\{\frac{\disc{\alpha}{\Lambda}{Q(x,R)}}{\alpha\cdot\vol(Q(x,R))}\right\}= \sup_{x\in R\Z^d}\left\{\absolute{\frac{\#(\Lambda\cap Q(x,R))}{\alpha\cdot R^d}-1}\right\}.
    \end{equation}

The following result is known as the Burago--Kleiner sufficient condition for rectifiability (\cite[Theorem~1.3]{BuragoKleiner-2002}, \cite[Theorem~3.1]{Aliste-PrietoCoronelGambaudo-2013}).
\begin{thm}\label{thm:BK_sufficient_condition} 
Let $\Lambda\subset\R^d$ be a Delone set. Suppose there exists an $\alpha>0$ so that  
\begin{equation}\label{eq:BK_series}
\sum_{k=1}^\infty \Delta_\Lambda(\alpha,2^k) < \infty.
\end{equation}
Then $\Lambda$ is BL equivalent to a lattice. 
\end{thm}

Note that this condition is only sufficient and not necessary, as can be seen by the following exercise. 

\begin{exe}\label{ex:BL for which BK fails}
    Consider the Delone set 
    \[
    \Lambda = \left\{(x,y)\in\Z^2 ~\Big| ~ y\le 0\right\} \cup 
    \left\{(x,y)\in2\Z^2 ~\Big| ~ y> 0\right\}.
    \]
    Prove that Burago and Kleiner's condition fails for $\Lambda$, but 
        $\Lambda$ is rectifiable.
        
        \medskip
        \noindent\textbf{Hint:} the map $h:\Lambda\to\Z^2$, $(x,y)\mapsto\begin{cases}
            (x,y), & y\le 0, \\
            (x/2,y/2), & y>0
        \end{cases}$
        is not biLipschitz.
\end{exe}

\begin{exe}
    Prove that if a Delone set $\Lambda\subset\R^d$ does not have a central density, namely, the limit in Definition \ref{def:central_density} does not exist, then condition \eqref{eq:BK_series} fails. 
\end{exe}

\begin{exe}
    Prove that if $\Lambda$ satisfies \eqref{eq:BK_series} then every $\Lambda'\in X_\Lambda$ is rectifiable.
\end{exe}

We note that Burago and Kleiner originally stated their condition differently. Indeed, setting
\[
E_\Lambda(\alpha,R) = \sup_{x\in R\Z^d}\left\{\max\left(
\frac{\#(\Lambda\cap Q(x,R))}{\alpha\cdot R^d} , 
\frac{\alpha\cdot R^d}{\#(\Lambda\cap Q(x,R))}
\right)\right\},
\]
they proved in \cite[Theorem 1.3]{BuragoKleiner-2002} that the existence of $\alpha>0$ for which the product 
$\prod_{k=1}^\infty E_\Lambda(\alpha,2^k)$ converges is a sufficient condition for rectifiability. However, it is sometimes more convenient to work with the additive condition \eqref{eq:BK_series}, see \cite{HaynesKellyWeiss-2014} for example, and so we include the following simple computation.

We note that Burago and Kleiner originally stated their condition
multiplicatively. For $R>0$, set
\[
E_\Lambda(\alpha,R)
=
\sup_{x\in R\Z^d}
\left\{
\max\left(
\frac{\#(\Lambda\cap Q(x,R))}{\alpha R^d},
\frac{\alpha R^d}{\#(\Lambda\cap Q(x,R))}
\right)
\right\},
\]
with the convention that the second quotient, and hence
$E_\Lambda(\alpha,R)$, is infinite if
$\#(\Lambda\cap Q(x,R))=0$ for some $x\in R\Z^d$.

Burago and Kleiner proved in
\cite[Theorem~1.3]{BuragoKleiner-2002} that convergence of the
corresponding product is a sufficient condition for rectifiability.
Since finitely many initial scales play no role in their construction,
the condition may equivalently be imposed on a sufficiently far
tail of the product. The additive formulation is related to this
tail condition as follows.

\begin{prop}\label{prop:convergenct_prod_vs_sum}
    If
    $
    \sum_{k=1}^\infty\Delta_\Lambda(\alpha,2^k)<\infty,
    $
    then there exists $k_0\in\N$ such that
    $
    \prod_{k=k_0}^\infty E_\Lambda(\alpha,2^k)<\infty.
    $
\end{prop}

\begin{proof}
    Set $\delta_k=\Delta_\Lambda(\alpha,2^k)$. Since $\sum_k\delta_k<\infty$, we have $\delta_k\to0$. Choose
    $k_0$ so that $\delta_k\le\frac12$ for every $k\ge k_0$.
    For $x\in2^k\Z^d$, write
    \[
    a_{k,x}
    =
    \frac{\#(\Lambda\cap Q(x,2^k))}
         {\alpha2^{kd}}.
    \]
    By the definition of $\delta_k$,
    $\absolute{a_{k,x}-1}\le\delta_k$.
    Hence, for $k\ge k_0$ we have 
    $\frac12\le a_{k,x}\le\frac32$.
    In particular, all relevant cubes contain at least one point of
    $\Lambda$, and
    \[
    \max\{a_{k,x},a_{k,x}^{-1}\}-1
    \le
    2\delta_k.
    \]
    Taking the supremum over $x$ gives
    \[
    E_\Lambda(\alpha,2^k)-1\le2\delta_k.
    \]
    Since $\log y\le y-1$ for $y\ge1$,
    \[
    \sum_{k=k_0}^\infty
    \log E_\Lambda(\alpha,2^k)
    \le
    2\sum_{k=k_0}^\infty\delta_k
    <
    \infty,
    \]
    and the assertion follows.
\end{proof}

To illustrate how the convergence of the above series relates to the statement of the theorem, we present below a brief outline of the proof of Theorem \ref{thm:BK_sufficient_condition} in the two-dimensional case. The reader is referred to \cite{BuragoKleiner-2002} for the complete proof. A main part of Burago and Kleiner's proof consists of the following technical lemma \cite[Proposition~3.2]{BuragoKleiner-2002}, in which they solve a certain prescribed Jacobian problem in a box. The higher-dimensional analogue, which was used in the proof of \cite{Aliste-PrietoCoronelGambaudo-2013}, relies on a similar lemma, taken from \cite{RiviereYe-1996}.

\begin{lem}\label{lem:BK_technical_lemma}
    Let $T=[0,2]^2$. There exists $C>0$ such that for every function $f:T\to\R_{\ge 0}$, which is constant on each of the four open unit lattice subsquares of $T$, there exists a biLipschitz homeomorphism $\Phi:T\to T$ that satisfies 
    \begin{enumerate}
        \item 
        $\Phi$ fixes $\partial T$ pointwise. \label{item_thm:BK_lemma_1}
        \item 
        $\Jac(\Phi) = \frac{\vol(T)}{\int_T f}f$ a.e.
        \item 
        The biLipschitz constant of $\Phi$ satisfies $\bL(\Phi)\le \left(\frac{\max f}{\min f}\right)^C$.
    \end{enumerate}
\end{lem}

Relying on the above lemma, we sketch the proof of Burago and Kleiner's sufficient condition for rectifiability in $d=2$. 
\begin{proof}[Sketch of proof of Theorem \ref{thm:BK_sufficient_condition}]
    {\bf Step 1:} Let $\tau_0$ denote the tiling of $\R^2$ by unit squares with vertices at $\Z^2$. By rescaling $\Lambda$, we may assume that every tile of $\tau_0$ contains at most one point of $\Lambda$. Define a Voronoi-type tiling $\tau_\Lambda=\{T_x\}_{x\in \Lambda}$ by 
    \begin{equation}\label{eq:Voronoi-type}
    T_x = \bigcup\{Q\in\tau_0 \mid \forall x'\in \Lambda, \dist(x,Q)\le \dist(x',Q)\},
    \end{equation}    
    (compare this with \eqref{eq:Voronoi_tiling}), where if a cube $Q$ is of equal distance from multiple points in $\Lambda$, it is arbitrarily assigned to exactly one of them. 
    Define a function $f:\R^2\to\R_{\ge 0}$, which is constant on unit lattice squares, by 
    \[
    f(y) = \begin{cases}
        \frac{1}{\vol(T_x)},& \exists x\in \Lambda, y\in\inter(T_x). \\
    1,& \text{otherwise}. 
    \end{cases}
    \]
    Then, as in the proof of (1) implies (2) in Theorem \ref{thm:Jac_implies_BL}, it suffices to show that $f$ can be realized a.e. as the Jacobian of a biLipschitz homeomorphism of $\R^2$. Furthermore, note that for a cube $Q(x,R)$ with $R$ large, the difference between $\int_{Q(x,R)}f$ and $\#(\Lambda\cap Q(x,R))$ is bounded by a linear function of $R$. Since the denominator in \eqref{eq:BK_discrepancy} is quadratic in $R$ (since $d=2$), replacing $\#(\Lambda\cap Q(x,R))$ by $\int_{Q(x,R)}f$ will not affect the convergence in \eqref{eq:BK_series}.  

    {\bf Step 2:} We show that $f$, as above, can be realized a.e. as a Jacobian. More generally, we show that if $g:\R^2\to\R_{\ge 0}$ is constant on open lattice unit squares, and satisfies 
    \[
\sum_{k=1}^\infty\sup_{x\in 2^k\Z^d}\left|\frac{\int_{Q(x,2^k)}g}{\alpha\cdot 2^{kd}}-1\right| < \infty,
    \]
    for some suitable $\alpha>0$, then $g$ can be realized a.e. as a Jacobian. Furthermore, as explained in the proof of Proposition \ref{prop:convergenct_prod_vs_sum}, it suffice to make the weaker assumption, namely that $\prod_{k=1}^\infty E_g(\alpha,2^k)<\infty$, where 
    \[
    E_g(\alpha,R)=\sup_{x\in R\Z^d}\left\{\max\left(
\frac{\int_{Q(x,R)}g}{\alpha\cdot R^d} , \frac{\alpha\cdot R^d}{\int_{Q(x,R)}g}
\right)\right\}.
    \]

    {\bf Step 3:} Observe that if we can solve the Jacobian problem for $g/\alpha$ then we can also solve it for $g$ by postcomposing a scaling map. Thus, we may assume that $\alpha=1$. 
    Denote by $\tau_n$ the tiling of $\R^2$ by $2^n$-squares with vertices at $2^n\Z^2$ and let $T^{(n)}_{i,j}$ denote the tile of $\tau_n$ with lower left corner at $(i,j)\in 2^n\Z^2$. Set $g_0=g$, and for each $n\ge 1$, denote by $g_n$ the function whose value at each $x\in\inter(T^{(n)}_{i,j})$ is the average of $g$ on $T^{(n)}_{i,j}$. 
    
    We now apply Lemma \ref{lem:BK_technical_lemma} on $g_{n-1}/g_n$ in $T^{(n)}_{i,j}$ for each $i,j\in 2^n\Z$ separately. This yields a biLipschitz homeomorphism $\psi_n:\R^2\to\R^2$, for every $n\in\N$, with 
    \[
    \Jac(\psi_n) = \frac{g_{n-1}}{g_n}=\frac{g_{n-1}}{g_n\circ\psi_n} \text{ a.e.},
    \]
    where the second equality holds because $\psi_n$ maps every
    $T^{(n)}_{i,j}$ to itself and $g_n$ is constant on that square. Furthermore, 
    \[
    \bL(\psi_n)\le \left(\frac{\max (g_{n-1}/g_n)}{\min (g_{n-1}/g_n)}\right)^C.
    \]
    Note that    
    \begin{align*}
       \frac{\max (g_{n-1}/g_n)}{\min (g_{n-1}/g_n)} &\le \frac{\max g_{n-1} \cdot \max g_n }{\min g_n \cdot \min g_{n-1}} \\
    & \le 
    \frac{\sup_{x\in 2^{n}\Z^d} 
    \left(
    \frac{\int_{Q(x,2^n)}g}{2^{nd}}
    \right) \cdot 
    \sup_{x\in 2^{n-1}\Z^d} 
    \left(
    \frac{\int_{Q(x,2^{n-1})}g}{2^{(n-1)d}}
    \right)}
    {\inf_{x\in 2^{n}\Z^d} 
    \left(
    \frac{\int_{Q(x,2^n)}g}{2^{nd}}
    \right) \cdot 
    \inf_{x\in 2^{n-1}\Z^d} 
    \left(
    \frac{\int_{Q(x,2^{n-1})}g}{2^{(n-1)d}}
    \right)}
    \\
    & \le \left(E_g(1,2^n)\right)^2 \cdot \left(E_g(1,2^{n-1})\right)^2. 
    \end{align*}
Thus 
    \[
    \bL(\psi_n)\le E_g(1,2^n)^{2C} \cdot E_g(1,2^{n-1})^{2C}.
    \]
    Setting $\phi_n=\psi_n\circ\cdots\circ\psi_1:\R^2\to\R^2$ yields 
    \[
    \bL(\phi_n)\le E_g(1,1)^{2C}\left(\prod_{k=1}^n E_g(1,2^k)\right)^{4C} \le \left(\prod_{k=0}^\infty E_g(1,2^k)\right)^{4C} <\infty,
    \]
    and by the chain rule, the Jacobian factors telescope:
    \[
    \Jac(\phi_n)
    =
    \frac{g_0}{g_n\circ\phi_n}.
    \]
    Every $\psi_j$ with $j\le n$ preserves each square of
    $\tau_j$, and therefore preserves the larger squares on which
    $g_n$ is constant. Thus
    $g_n\circ\phi_n=g_n$, and consequently
    \[
    \Jac(\phi_n)=\frac{g}{g_n}.
    \]
    
    Finally, recall that $g_n$ has average $1$ and is constant on $2^n\times 2^n$ squares with vertices in $2^n\Z^2$. Then, since $\bL(\phi_n)$ is uniformly bounded, we may apply Arzel\`a--Ascoli to obtain a partial limit $\Phi:\R^2\to\R^2$, which is a biLipschitz homeomorphism with Jacobian $g$ a.e., as required. 
\end{proof}

To conclude this part, we note that although the condition
\eqref{eq:BK_series} in Theorem
\ref{thm:BK_sufficient_condition} may be difficult to verify,
it remains an important general condition for rectifiability. It has
been used to establish rectifiability for several major families of
Delone sets, including primitive substitution Delone sets and
certain cut-and-project sets. It was also used by Aliste-Prieto,
Coronel, and Gambaudo to prove the following theorem.

\begin{thm}
    Let $\Lambda\subset\R^d$ be a linearly repetitive Delone set.
    Then $\Lambda$ is rectifiable.
\end{thm}

A similar result holds in the non-FLC setting, see
\cite{SmilanskySolomon2-2022}.

\subsection{An explicit encoding obstruction to rectifiability.}
Following \cite{BhatDymond-2025}, a measurable function
$\rho:[0,1]^d\longrightarrow\R_{>0}$
is called \emph{non-realisable} if there is no biLipschitz homeomorphism
$\Phi:[0,1]^d\longrightarrow\R^d$ satisfying
\[
\Jac(\Phi)=\rho
\qquad\text{a.e. on }[0,1]^d.
\]

We say that a Delone set $\Lambda\subset\R^d$
\emph{encodes the density $\rho$} if there exist numbers
$r_n\to\infty$, closed cubes $Q_n\subset\R^d$ of edge length $r_n$,
and bijective affine maps
$\phi_n:[0,1]^d\longrightarrow Q_n$ such that the measures
\[
\mu_n
=
\frac1{r_n^d}
\sum_{x\in\Lambda\cap Q_n}
\delta_{\phi_n^{-1}(x)}
\]
converge weakly to $\rho(u)du$ on $[0,1]^d$. The following proposition appears in \cite{BhatDymond-2025}.

\begin{prop}
    Let $\Lambda\subset\R^d$ be a Delone set. If $\Lambda$ encodes
    a non-realisable density $\rho$, then $\Lambda$ is
    non-rectifiable.
\end{prop}

\begin{proof}[Sketch of proof]
    Suppose that there exists a biLipschitz bijection
    \[
    f:\Lambda\longrightarrow\Z^d.
    \]
    Restrict $f$ to $\Lambda\cap Q_n$, rescale both its domain and
    range by the factor $r_n^{-1}$, and normalize by a translation.
    The resulting maps are defined on subsets of $[0,1]^d$ that
    become dense as $n\to\infty$, and they have uniformly bounded
    biLipschitz constants.

    After extending the maps and passing to a subsequence,
    Arzel\`a--Ascoli gives a biLipschitz map
    \[
    \Phi:[0,1]^d\longrightarrow\R^d.
    \]
    Since $f$ is a bijection onto $\Z^d$, the pushforwards of the
    rescaled empirical measures converge to Lebesgue measure on
    $\Phi([0,1]^d)$. On the other hand, the encoding assumption
    implies that these pushforwards converge to
    $\Phi_\#\bigl(\rho(u)\,du\bigr)$.
    It follows that
    \[
    \Phi_\#\bigl(\rho(u)\,du\bigr)
    =
    \mathbf 1_{\Phi([0,1]^d)}(y)\,dy.
    \]
    The change of variables formula therefore gives
    \[
    \absolute{\det D\Phi(u)}=\rho(u)
    \qquad\text{for a.e. }u\in[0,1]^d.
    \]
    After composing with a reflection if necessary, this contradicts
    the non-realisability of $\rho$.
\end{proof}

Thus, encoding a non-realisable density provides a checkable
sufficient condition for non-rectifiability. The following converse
question remains open.

\begin{open}
    If a Delone set $\Lambda\subset\R^d$ is non-rectifiable, must
    $\Lambda$ encode a non-realisable density?
\end{open}

\section{BD equivalence}\label{sec:BD}
Bounded displacement equivalence arises naturally in the study of large-scale geometry of discrete sets and random point processes (see also \S \ref{sec:beyond BL and BD}) in the context of lattice perturbations. Indeed, it follows immediately from Definition \ref{def:BD} that for a Delone set in $\R^d$ to be BD equivalent to a lattice, namely to be uniformly spread, simply means to be a uniformly bounded perturbation of one. It is therefore natural to attempt to classify which sets are equivalent to a lattice, and to explore properties of those that are not. Moreover, since BD equivalence implies BL equivalence (see Proposition \ref{prop:BD_implies_BL}), BD equivalence also arises in the study of biLipschitz equivalence of Delone sets. One example is the appearance of BD equivalence in the arguments of Burago--Kleiner and McMullen, where it appears as a crucial intermediate step (see the proof of Theorem \ref{thm:Jac_implies_BL}, (1) implies (2)). Another occurs in the search for non-rectifiable constructions, where it is a natural first step to explore sets that are not uniformly spread.

\subsection{The Laczkovich criterion for uniform spreadness} 
A key insight into the structure of uniformly spread Delone sets arises from the work of Laczkovich in relation to his solution of Tarski's circle-squaring problem, where he showed that being uniformly spread is equivalent to having uniformly strong discrepancy bounds. 
To state Laczkovich's result in full detail, denote by $\HHH(r)$ the set of all finite unions of axis-parallel cubes in $\R^d$ of side length $r$ with vertices in $r\Z^d$, for $r>0$, and set $\HHH=\HHH(1)$. Recall that $F^{+\varepsilon}$ denotes the $\varepsilon$-neighborhood of a set $F$ as in \eqref{eq:eps_neighborhood_of_sets}, $\disc{\alpha}{\Lambda}{B}$ is the discrepancy of $\Lambda$ with respect to $B$ and $\alpha$ is as in Definition \ref{def:discrepancy}. The following result is known as Laczkovich's criterion and appeared in \cite[Theorem~1.1]{Laczkovich-1992}. 

\begin{thm}\label{thm:Laczkovich}
The following are equivalent for a discrete set $\Lambda \subset \R^d$ and $\alpha > 0$:
\begin{enumerate}
    \item\label{item_thm:Laczkovich_1}
    There exists a constant $C > 0$ such that for every bounded measurable set $B \subset \R^d$ we have 
    \[
    \disc{\alpha}{\Lambda}{B}\le C \cdot \vol\left( (\partial B)^{+1} \right).
    \]
    \item\label{item_thm:Laczkovich_2}
    There exists a constant  $C > 0$ such that for every $U\in\HHH$ we have 
    \[
    \disc{\alpha}{\Lambda}{U}\le C \cdot \vol_{d-1}(\partial U).
    \]
    \item\label{item_thm:Laczkovich_3}
    There exists a bijection $\phi : \Lambda \to \alpha^{-1/d} \Z^d$ such that
    \[
    \sup_{x \in \Lambda} \norm{\phi(x) - x} < \infty.
    \]
\end{enumerate}
\end{thm}
Note that the equivalence between conditions \eqref{item_thm:Laczkovich_1} and \eqref{item_thm:Laczkovich_3}, in dimension $2$, was proved earlier in \cite[Theorem~3.1]{KuipersNiederreiter-1974}. Laczkovich's criterion is the primary tool used to prove and disprove BD equivalence between a Delone set $\Lambda$ and a lattice. We elaborate further on applications of this theorem in \S \ref{sec:substitution}, \S \ref{sec:multiscale} and \S\ref{sec:C&P}.

Before we discuss the proof, we note that the next corollary of Theorem \ref{thm:Laczkovich} is often used in applications.

\begin{cor}\label{cor:Lacz_with_other_cubes}
    For every $r>0$, the condition below is equivalent to the conditions in the theorem.
    \begin{enumerate}
        \item[(2')] There exists a constant  $C > 0$ such that for every $U\in\HHH(r)$ we have 
    \[
    \disc{\alpha}{\Lambda}{U}\le C \cdot \vol_{d-1}(\partial U).
    \]   
    \end{enumerate} 
\end{cor}

\begin{exe}
    Deduce Corollary \ref{cor:Lacz_with_other_cubes} from the equivalence \eqref{item_thm:Laczkovich_2} $\Leftrightarrow$ \eqref{item_thm:Laczkovich_3} in Theorem \ref{thm:Laczkovich}.
\end{exe}

The proof of the Laczkovich criterion, namely Theorem \ref{thm:Laczkovich}, relies on Hall's marriage theorem (appearing here as Theorem \ref{thm:Hall}) and on the following technical lemma. 

\begin{lem}\label{lem:Lacz_tech_lemma}
There exists a constant $C$, depending only on $d$, such that 
    \begin{enumerate}
        \item \label{item_lem:Laczkovich_tech_1}
        For every $b\ge 1$ and every bounded set $B \subset \R^d$ we have 
        \[
        \vol\left((\partial B)^{+b}\right) \le C\cdot b^d\cdot \vol\left((\partial B)^{+1}\right) .
        \]
    \item\label{item_lem:Laczkovich_tech_2}
        For every $r\ge 1$ and $U \in \HHH(r)$ we have 
        \[
         \vol\left((\partial U)^{+1}\right)  \le C\cdot \vol_{d-1}(\partial U).
        \]
    \end{enumerate}
\end{lem}

The reader is referred to \cite[Lemma~2.1]{Laczkovich-1992} for the proof of \eqref{item_lem:Laczkovich_tech_1}.

\begin{exe}
    Deduce \eqref{item_lem:Laczkovich_tech_2} from the proof of \cite[Lemma~2.2]{Laczkovich-1992}.
\end{exe}

Indeed, the implication \eqref{item_thm:Laczkovich_1} $\Rightarrow$ \eqref{item_thm:Laczkovich_2} in Theorem \ref{thm:Laczkovich} follows directly from \eqref{item_lem:Laczkovich_tech_2} of Lemma \ref{lem:Lacz_tech_lemma}.  
Our main focus here is the equivalence between \eqref{item_thm:Laczkovich_2} and \eqref{item_thm:Laczkovich_3}, which we prove in  \S \ref{subsec:extending_Lacz} in a more general setting, namely for two non-lattice Delone sets.
For the implication \eqref{item_thm:Laczkovich_3} $\Rightarrow$ \eqref{item_thm:Laczkovich_1}, which is crucial in refuting BD equivalence to a lattice (see examples in \S \ref{sec:substitution} and \S \ref{sec:multiscale}), Laczkovich uses a sophisticated argument to directly bound the quantities
$\pm\left(\#(\Lambda\cap B)-\alpha\vol(B)\right)$ for a measurable set $B$. The reader is referred to \cite[p.42]{Laczkovich-1992} for the full proof.  
However, when considering non-BD equivalence of Delone sets, it is often sufficient to choose a suitably regular set $B$, for example, a finite union of convex compact cells that are tiles of the corresponding Voronoi tiling. These can in turn be approximated by elements of $\HHH(r)$ for small $r$, with an error of order $\vol_{d-1}(\partial B)$, and thus condition \eqref{item_thm:Laczkovich_2} suffices. 

Several immediate applications of Laczkovich’s theorem include the following.

\begin{exe}
    Reprove Proposition \ref{prop:BD_implies_BL} for uniformly spread Delone sets by verifying that Laczkovich's condition implies the Burago--Kleiner condition in Theorem \ref{thm:BK_sufficient_condition}.
\end{exe}

\begin{exe}\label{exe:BD_of_the_entire_hull}
    Prove that if $\Lambda$ is BD equivalent to a lattice, then every $\Lambda'\in X_\Lambda$ is BD equivalent to that lattice.
\end{exe}

\begin{exe}
    Prove that if $\Lambda_1,\ldots,\Lambda_n$ are pairwise disjoint Delone sets in $\R^d$, each is BD equivalent to some lattice, then their union $\bigcup_{i=1}^n \Lambda_i$ is also BD to a lattice.      
\end{exe}

\begin{remark}
Even when the existence of a BD-map from a Delone set to a lattice can be established using Laczkovich's theorem, an explicit description of the bijection is hard to find since the argument relies on Hall's marriage theorem. However, there are examples where this can be achieved, as shown in \cite{AkiyamaHamadaIto-2025}, where explicit BD-maps are used to construct aperiodic tile sets.    
\end{remark}

\subsection{A general criterion for BD equivalence}\label{subsec:extending_Lacz}
Laczkovich’s criterion provides a quantitative method for determining whether a given Delone set is BD equivalent to a lattice, viewed as a canonical reference set. The same underlying geometric idea can be extended to compare two arbitrary Delone sets, neither of which needs to be a lattice. This leads to a generalized criterion, established in \cite{FrettlohSmilanskySolomon-2021} and \cite{SmilanskySolomon1-2022}, for BD equivalence between arbitrary pairs of Delone sets, formulated in terms of their discrepancy, with respect to sets $U\in \HHH$.

\begin{thm}\label{thm:non_BD_criterion}
	Let $\Lambda_1, \Lambda_2\subset\R^d$ be two Delone sets. Then $\Lambda_1 \stackrel{\mathrm{BD}}{\nsim} \Lambda_2$ 
    if and only if there exists a sequence $(U_m)_{m\in\N}$ of sets in $\HHH$ for which
	\begin{equation}\label{eq:non_BD_condition}
	\lim_{m\to\infty} \frac{| \#(\Lambda_1 \cap U_m) - \#(\Lambda_2 \cap U_m) |}{\vol_{d-1}(\partial U_m)} = \infty.
	\end{equation}   
\end{thm}

The following exercise provides a nice application of Theorem~\ref{thm:non_BD_criterion}.
\begin{exe}\label{cor:different_densities_implies_not_BD}
	Let $\Lambda_1,\Lambda_2\subset\R^d$ be Delone sets such that $\Lambda_1$ has central density  $\alpha$ (see Definition \ref{def:central_density}), while $\Lambda_2$ does not have central density $\alpha$. Prove that $\Lambda_1 \stackrel{\mathrm{BD}}{\nsim} \Lambda_2$. 
\end{exe}

The proof of Theorem \ref{thm:non_BD_criterion} relies on Proposition \ref{prop:BD_criterion} (see \cite[Proposition~3.1]{FrettlohSmilanskySolomon-2021}), which extends Laczkovich’s idea by relating BD equivalence to Hall’s marriage theorem.

\begin{prop}\label{prop:BD_criterion}
	Given two Delone sets $\Lambda_1, \Lambda_2\subset\R^d$, there  exists a BD-map $\phi:\Lambda_1 \to \Lambda_2$ if and only if there exists $s>0$ such that for every $U\in\HHH$ we have    \begin{equation}\label{eq:Hall_with_cubes}
	 \#(\Lambda_1 \cap U) \le \#(\Lambda_2 \cap U^{+s}) \qquad
	 \text{ and }  \qquad
	 \#(\Lambda_2 \cap U) \le \#(\Lambda_1 \cap  U^{+s}).
	\end{equation}
\end{prop}

\begin{proof}
	The existence of a BD-map from $\Lambda_1$ to $\Lambda_2$ is equivalent to the existence of some $r>0$ for which the locally finite bipartite graph
	\begin{equation*}
	  G_r = \big(V_1\uplus V_2,E_r\big) = \big( \Lambda_1 \uplus \Lambda_2, \big\{ \{x,y\}
	\mid x \in \Lambda_1, y \in \Lambda_2, \norm{x-y} \le r \big\} \big)
    \end{equation*}
    admits a perfect matching. 
	Thus, by Hall’s marriage theorem, it remains to show that there exists $s>0$ satisfying \eqref{eq:Hall_with_cubes} for all $U \in \HHH$ if and only if there exists $r > 0$ such that
    \begin{equation}\label{eq:Hall's_condition}
    \#F_1 \le \#N_{G_r}(F_1) \qquad \text{and} \qquad \#F_2 \le \#N_{G_r}(F_2)
    \end{equation}
    hold for all finite subsets $F_i \subset \Lambda_i$. 
    
     For a set of vertices $V$ in $G_r$, denote by $U(V)\in\HHH$ the union of all the unit cubes that intersect $V$. Since the diameter of a unit cube is $\sqrt{d}$, we clearly have $U(V)^{+s}\subset V^{+(s+\sqrt{d})}$. Thus, assuming \eqref{eq:Hall_with_cubes}, we obtain \eqref{eq:Hall's_condition} with $r=s+\sqrt{d}$ by 
     \[
     \#F_1\le \#(\Lambda_1\cap U(F_1))\le \#(\Lambda_2\cap U(F_1)^{+s}) 
     \le \#(\Lambda_2\cap F_1^{+(s+\sqrt{d})})= \# N_{G_{s+\sqrt{d}}}(F_1),
     \]
     and similarly for $F_2\subset\Lambda_2$. 
     Conversely, given $U\in \HHH$, let $F_1=\Lambda_1\cap U$.  Then 
     \[
     \#(\Lambda_1 \cap U) = \# F_1 \le \# N_{G_r}(F_1)
     \le \#(\Lambda_2 \cap U^{+r}).
     \]
     By the symmetry between the sides of $G_r$, the proof is complete. 
\end{proof}

The following lemma connects the quantities $\#(\Lambda \cap (U'\minus U))$ and $\vol_{d-1}(\partial U)$, for $U\in\HHH(r)$ and certain sets $U'$ containing $U$. The lower bound appears in \cite{Laczkovich-1992}, at the end of the proof of Lemma 2.3. We include this elegant argument here for completeness. 

\begin{lem}\label{lem:U'-U}
    Let $\Lambda\subset\R^d$ be a Delone set. There exist constants
    $C_1,C_2>0$ and $r_0>0$, depending only on $\Lambda$ and $d$,
    such that:
    \begin{enumerate}
        \item\label{item_lem:U'-U_1}
        For every $U\in\HHH$ and every $s>1$,
        \[
        \#\bigl(\Lambda\cap(U^{+s}\minus U)\bigr)
        \le
        C_1s^d\vol_{d-1}(\partial U).
        \]

        \item
        Let $r\ge r_0$ and $U\in\HHH(r)$. Write
        $U=\bigcup_{i=1}^kQ_i$, where the $Q_i$ are cubes with vertices in $r\Z^d$.
        For every $i$, let $Q_i'$ be the cube of side length $3r$ concentric with $Q_i$, and set $U'=\bigcup_{i=1}^kQ_i'$. Then
        \begin{equation*}
        C_2r\vol_{d-1}(\partial U)
        \le
        \#\bigl(\Lambda\cap(U'\minus U)\bigr).
        \end{equation*}
    \end{enumerate}
\end{lem}

\begin{proof}
    Let $r_\Lambda>0$ be the separation constant of $\Lambda$ and
    set $\rho=\frac{r_\Lambda}{3}$. So the balls $\{B(x,\rho)\}_{x\in\Lambda}$, are pairwise disjoint. Since $U^{+s}\minus U\subset(\partial U)^{+s}$, we obtain
    \begin{align*}
    &\#\bigl(\Lambda\cap(U^{+s}\minus U)\bigr)
      \vol(B(0,\rho))\le
    \vol\left((\partial U)^{+(s+\rho)}\right).
    \end{align*}
    Applying both parts of Lemma \ref{lem:Lacz_tech_lemma}
    and using $s>1$ gives
    \[
    \vol\left((\partial U)^{+(s+\rho)}\right)
    \le
    C s^d\vol_{d-1}(\partial U),
    \]
    which proves the first assertion.

    For the second assertion, relative denseness of $\Lambda$
    implies that there exist constants $c_\Lambda>0$ and
    $\ell_0>0$ such that every open cube $P$ of edge length
    $\ell\ge\ell_0$ satisfies
    \begin{equation}\label{eq:Delone_cube_lower_bound}
    \#(\Lambda\cap P)\ge c_\Lambda\ell^d.
    \end{equation}

    Write $\partial U=\bigcup_{j=1}^tF_j$,
    where the $F_j$ are the exposed $(d-1)$-dimensional
    faces of the cubes forming $U$. 
    For every $j$, let $P_j$ be the reflection across $F_j$ of the
    cube of $U$ adjacent to $F_j$. Then $P_j$ is an exterior
    $r$-cube contained in $U'$, apart from its common boundary face
    with $U$, and thus $\inter P_j\subset U'\minus U$. 
    In the sequence $P_1, \ldots, P_t$, the same cube can appear at most $2d$ times, which is the number of possible reflections of a cube. Then by
    \eqref{eq:Delone_cube_lower_bound}, for every
    $r\ge \ell_0$, we have
    \[
    \#(\Lambda\cap P_j)
    \ge
    c_\Lambda r^d.
    \]
    Therefore,
    \begin{align*}
    2d\,
    \#\bigl(\Lambda\cap(U'\minus U)\bigr)
    &\ge
    \sum_{j=1}^t
    \#(\Lambda\cap P_j)
    \ge
    c_\Lambda tr^d
    =
    c_\Lambda r\vol_{d-1}(\partial U).
    \end{align*}
    Absorbing the factor $2d$ into the constant proves the second
    assertion.
\end{proof}

Having established the necessary preliminaries, we turn to the proof of Theorem~\ref{thm:non_BD_criterion}.

\begin{proof}[Proof of Theorem \ref{thm:non_BD_criterion}]

    First, assume that
    $\Lambda_1\stackrel{\mathrm{BD}}{\nsim}\Lambda_2$.
    By Proposition \ref{prop:BD_criterion}, for every
    $m\in\N$ there exists $V_m\in\HHH$ such that
    \begin{equation*}\label{eq:not_Hall_with_cubes}
	 \#(\Lambda_1 \cap V_m) > \#(\Lambda_2 \cap V_m^{+2\sqrt d m}) \qquad
	 \text{ or }  \qquad
	 \#(\Lambda_2 \cap V_m) > \#(\Lambda_1 \cap  V_m^{+2\sqrt d m}).
	\end{equation*}
    Passing to a subsequence and interchanging the roles of the two
    sets if necessary, we may assume that for every $m$ we have
    \[
    \#(\Lambda_1\cap V_m)
    >
    \#\left(\Lambda_2\cap V_m^{+2\sqrt d\,m}\right)
    \]
    
    Let $U_m\in\HHH(m)$ be the union of the $m$-cubes that
    intersect $V_m$, and let $U_m'$ be the union of the concentric
    $3m$-cubes, as in Lemma \ref{lem:U'-U}. Then
    \[
    V_m\subset U_m\subset U_m'
    \subset V_m^{+2\sqrt d\,m}.
    \]
    Hence
    \[
    \#(\Lambda_1\cap U_m)
    >
    \#(\Lambda_2\cap U_m')
    =
    \#(\Lambda_2\cap U_m)
    +
    \#\bigl(\Lambda_2\cap(U_m'\minus U_m)\bigr).
    \]
    Using the lower bound in Lemma
    \ref{lem:U'-U}, for every sufficiently large $m$ we obtain that 
    \[
    \frac{
        \#(\Lambda_1\cap U_m)-\#(\Lambda_2\cap U_m)
    }{
        \vol_{d-1}(\partial U_m)
    }
    >
    C_2m
    \xrightarrow{m\to\infty}
    \infty.
    \]

    Conversely, let $U_m\in\HHH$ be a sequence satisfying \eqref{eq:non_BD_condition}. For contradiction, assume that $\Lambda_1\bd\Lambda_2$. Proposition \ref{prop:BD_criterion} then implies that there exists a constant $s>0$ so that \eqref{eq:Hall_with_cubes} holds for every $U\in \HHH$.
	Let $m$ be large enough so that 
	\[\absolute{ \#(\Lambda_1 \cap U_m) - \#(\Lambda_2 \cap U_m) } > C_1\cdot s^d\cdot \vol_{d-1}(
	\partial U_m), \]
	where $C_1$ is as in \eqref{item_lem:U'-U_1} of Lemma \ref{lem:U'-U}. By passing to a subsequence, we may assume without loss of generality that
	\[\#\left(\Lambda_1 \cap U_m\right) > \#\left(\Lambda_2 \cap U_m\right) + C_1\cdot s^d\cdot\vol_{d-1}(
	\partial U_m) \]
	holds for every $m\in\N$. Invoking \eqref{item_lem:U'-U_1} of Lemma \ref{lem:U'-U} thus yields that  
	\[\#\left(\Lambda_1 \cap U_m\right) > \#\left(\Lambda_2 \cap U_m\right) + \#\left(\Lambda_2 \cap (U_m^{+s}\minus U_m)\right) = \#\left(\Lambda_2 \cap U_m^{+s}\right),\] 
	contradicting \eqref{eq:Hall_with_cubes}, and the proof is complete.
\end{proof}

\subsection{A dichotomy for BD equivalence}\label{subsec:dichotomy}

Having established the framework and tools for detecting non-BD equivalence between Delone sets, we now turn to a broader perspective and consider the possible range of distinct BD classes that may be represented within the hull of a single Delone set (see Definition \ref{def:hull}). 
Clearly, a Delone set is BD equivalent to each of its global translations. However, non-equivalent Delone sets may still arise as limits of these translations, taken with respect to the Chabauty--Fell topology, see \S\ref{sec:FLC_topology_etc}.
While some hulls of Delone sets exhibit only finitely many BD equivalence classes, others display far more intricate behavior. For a set $X \subset \CC(\R^d)$ of Delone sets, we denote by $\BD(X)$ the number of distinct BD classes represented in $X$. Before discussing a general result on $\BD(X)$, we first present a concrete example that illustrates how multiple BD classes can appear in a simple, non-repetitive setting. 

\begin{example}\label{ex:3_BD_classes}
    Consider the Delone set $\Lambda = (-\N)\cup2\N\subset \R$. The hull $X_\Lambda$ contains three types of sets: translations of $\Lambda$, $\Z$, and $2\Z$. Translations of $\Lambda$ are all equivalent to $\Lambda$, which by Exercise \ref{ex:-N_cup_2N} is non-equivalent to both $\Z$ and $2\Z$. By a direct computation, or by Exercise \ref{cor:different_densities_implies_not_BD}, $\Z$ and $2\Z$ are non-equivalent as well. It follows that $\BD(X_\Lambda)=3$.
\end{example}

The following result was established in \cite{SmilanskySolomon1-2022}. It is a dichotomy on possible values of $\BD(X)$ for minimal spaces of Delone sets, and shows in particular that a phenomenon such as that of Example \ref{ex:3_BD_classes} cannot occur if $\Lambda$ is repetitive.  

\begin{thm}\label{thm:BD_dichotomy}
    Let $X\subset\CC(\R^d)$ be a minimal space of Delone sets. Then the following dichotomy holds: either
    \begin{enumerate}
        \item\label{item_thm1:dichotomy} 
        $\BD(X) = 1$, in which case every $\Lambda\in X$ is BD equivalent to a lattice, or
        \item\label{item_thm2:dichotomy} 
        $\BD(X) = 2^{\aleph_0}$, in which case no $\Lambda\in X$ is BD equivalent to a lattice.
    \end{enumerate}
\end{thm}

The main idea of constructing continuously many BD classes that underlies the proof of Theorem \ref{thm:BD_dichotomy} first appeared in \cite{FrettlohSmilanskySolomon-2021}, in the specific context of mixed substitution tilings. It was later further developed in \cite{Solomon-2020}, which presented a primitive substitution tiling whose hull intersects continuously many BD classes, before being applied in full generality. We also note that a similar result was obtained independently in \cite{FrettlohGarberSadun-2022}, where the dichotomy between $1$ and $2^{\aleph_0}$ BD-class representatives in a minimal hull is proved for FLC Delone sets that admit an asymptotic density. 

\begin{remark}\label{rem:one_BD_class}
While it is clear that if $\Lambda$ is BD equivalent to a lattice then every $\Lambda'$ in its hull is BD to a lattice as well (see Exercise~\ref{exe:BD_of_the_entire_hull}), deducing that $\Lambda$ is BD to a lattice from the assumption that $\BD(X)=1$ is nontrivial, particularly without assuming the existence of an asymptotic density. A proof of this fact can be found in \cite[\S 4]{SmilanskySolomon1-2022}.
\end{remark}

We give a sketch of the main ideas of the proof of Theorem \ref{thm:BD_dichotomy}. First, we establish the following simple result. 

\begin{lem}\label{lem:U_m_contains_balls}
    Let $\Lambda_1,\Lambda_2\subset\R^d$ be two Delone sets and suppose that $U_m\in \HHH$ is a sequence satisfying \eqref{eq:non_BD_condition}. Then for every $R>0$, there exists an $M$ so that every $U_m$ with $m\ge M$ contains a ball of radius $R$.    
\end{lem}

\begin{proof}
    For contradiction, suppose that there is some $R>0$ so that $U_m$ does not contain any $R$-ball for infinitely many values of $m$. This implies that for each $m$ we have $U_m\subset (\partial U_m)^{+R}$. Thus, Lemma \ref{lem:Lacz_tech_lemma} yields that 
    \[
    \vol(U_m)\le C\cdot R^d\cdot \vol_{d-1}(\partial U_m), 
    \]
    which clearly contradicts \eqref{eq:non_BD_condition}.
\end{proof}

To obtain continuously many BD-class representatives, we introduce a ``comparison set'' $\Omega$ which has the cardinality of the continuum. The construction will proceed by matching a distinct BD class representative to every element $\omega\in\Omega$. The set is described in the following exercise.

\begin{exe}\label{exe:Omega=cont.}
    Consider the equivalence relation on $\{0,1\}^\N$ in which $\sigma_1,\sigma_2\in\{0,1\}^\N$ are equivalent if they differ by finitely many entries. Let $\Omega\subset \{0,1\}^\N$ be a set of equivalence class representatives. Then $\absolute{\Omega} = 2^{\aleph_0}$.
\end{exe}

We now outline the main ideas behind the proof of Theorem \ref{thm:BD_dichotomy}.

\begin{proof}[Sketch of the proof of Theorem \ref{thm:BD_dichotomy}]
    In view of Remark \ref{rem:one_BD_class}, we will only explain how to construct continuously many BD classes from the assumption that $\BD(X)>1$. The proof proceeds through the following steps:
\begin{enumerate}
    \item
    Choose a pair $\Lambda_1,\Lambda_2\in X$ of Delone sets so that $\Lambda_1 \stackrel{\mathrm{BD}}{\nsim} \Lambda_2$. By Theorem \ref{thm:non_BD_criterion}, there exists a sequence of sets $(U_m)$ in $\HHH$ that satisfies \eqref{eq:non_BD_condition}.  
    \item 
    Since $X$ is minimal, it is equal to the orbit closure of $\Lambda_1$, and in particular, $\Lambda_2$ can be obtained as a limit of translations of $\Lambda_1$. This implies that the patches $\Lambda_2\cap U_m$ can essentially be found in $\Lambda_1$, up to a small perturbation. More explicitly, for every $m\in\N$, there is a translation vector $v_m \in\R^d$ such that 
    $\#(\Lambda_2 \cap U_m) = \#((\Lambda_1+v_m) \cap U_m)$, and so
    \begin{equation}\label{eq:non_BD_dichotomy}
	\lim_{m\to\infty} \frac{| \#(\Lambda_1 \cap U_m) - \#((\Lambda_1+v_m) \cap U_m) |}{\vol_{d-1}(\partial U_m)} = \infty.
	\end{equation}   
    \item 
    Applying Lemma~\ref{lem:U_m_contains_balls}, we deduce that for every $R>0$, there exists $M\in\N$ such that every $U_m$ with $m\ge M$ contains an $R$-ball. Using the minimality of $X$, and after passing to a subsequence if necessary, we may assume that the sets $U_m$ grow sufficiently fast so that every $U_m$-patch of $\Lambda_1$ contains, up to a $\frac{1}{m}$-perturbation, any $U_{m-1}$-patch of $\Lambda_1$ (this relies on the fact that minimality is equivalent to a weaker form of repetitivity in the non-FLC setting, see, e.g., \cite[Theorem~3.11]{FrettlohRichard-2014} and \cite[Theorem~6.5]{SmilanskySolomon-2021} for a precise definition and a proof of this equivalence).
    \item 
    By \eqref{eq:non_BD_dichotomy}, for every $m\in\N$, one of the two sets $\Lambda_1\cap U_m$ and $(\Lambda_1+v_m)\cap U_m$ contains significantly more points than the other. We refer to the one with more points as \textbf{dense} (D), and to the one with fewer points as \textbf{sparse} (S). Accordingly, for each $m$, we denote by $D_m$ and $S_m$ the patches $\Lambda_1\cap U_m$ and $(\Lambda_1+v_m)\cap U_m$, where $D_m$ is the denser one and $S_m$ the sparser.
    \item 
    We now consider the set $\Omega$ from Exercise~\ref{exe:Omega=cont.}. For every element $\omega \in \Omega$ we construct a Delone set in $X$ as follows.  
    For each $\omega \in \Omega$, define a sequence of patches by  
    \[
    P^\omega_m =
    \begin{cases}
    D_m, & \omega(m) = 0, \\
    S_m, & \omega(m) = 1~.
    \end{cases}
    \]
    Next, choose translation vectors $t^\omega_m$ such that each patch $P^\omega_m + t^\omega_m$ contains the origin, and the sequence $(P^\omega_m + t^\omega_m)$ is “essentially nested” (see \cite{SmilanskySolomon1-2022} for technical details).  
    Set $\Lambda_\omega = \lim_{m \to \infty} (P^\omega_m + t^\omega_m)$.  
    Then, for every pair of distinct elements $\sigma, \omega \in \Omega$ and every $m \in \N$ with $\sigma(m) \neq \omega(m)$, there exists a translate $\tilde U_m$ of $U_m$ such that, essentially,
    \[
    \frac{\left|\#(\Lambda_\sigma \cap \tilde U_m) - \#(\Lambda_\omega \cap \tilde U_m)\right|}
    {\vol_{d-1}(\partial \tilde U_m)}
    = \Theta\!\left(
    \frac{\left|\#(\Lambda_1 \cap U_m) - \#((\Lambda_1 + v_m) \cap U_m)\right|}
    {\vol_{d-1}(\partial U_m)}
    \right),
    \]
    which, by \eqref{eq:non_BD_dichotomy}, tends to infinity as $m \to \infty$.  
    Since $\sigma(m)$ and $\omega(m)$ differ for infinitely many $m$, we obtain a sequence of sets in $\HHH$ satisfying \eqref{eq:non_BD_condition} for the Delone sets $\Lambda_\sigma$ and $\Lambda_\omega$, and the assertion follows.\qedhere
\end{enumerate}
\end{proof}

The dichotomy in Theorem \ref{thm:BD_dichotomy} applies to many constructions in aperiodic order. These include the hulls of Delone sets associated with primitive substitution tilings, incommensurable multiscale substitution tilings and cut-and-project sets, described in \S\ref{sec:substitution}, \S\ref{sec:multiscale} and \S\ref{sec:C&P}, respectively, all of which are minimal, see \cite{Solomyak-2006} and \cite[\S6]{SmilanskySolomon-2021} (note that for both tiling constructions, to define $\Lambda=\Lambda_\tau$ with a minimal hull, points must be chosen by first marking a fixed point in each prototile). In fact, the combination of Theorems \ref{thm:BL_substitution} and \ref{thm:BD_substitution} gives rise to minimal systems of Delone sets that contain representatives from continuously many BD classes, all of which are rectifiable. An even stronger result, which further clarifies the relationship between BD and BL equivalence, was established by Dymond and Kalu{\v z}a in \cite[Theorem~5.1]{DymondKaluza-2024}.

\begin{thm}
    Every BL equivalence class of Delone sets in $\R^d$, $d\geq 1$, contains representatives from continuously many BD classes.
\end{thm}

\section{Substitution tilings}\label{sec:substitution}
Every Delone set in $\R^d$ can be obtained by choosing a point in each tile of a tiling of finite local complexity, see Exercise \ref{ex:finitely_many_tiles}. A simple but fascinating method for constructing such tilings is via substitution rules.
Recall that in symbolic dynamics, substitution systems provide a systematic way to generate sequences with rich combinatorial and dynamical structure, often combining long-range order with aperiodicity. The analogous construction in the geometric setting of tilings replaces the symbolic alphabet by a finite collection of prototiles, and the substitution rule acts on tiles rather than on symbols. A well-known and fundamental example is given by the Penrose tilings \cite{Penrose-1979}, see also Figures \ref{fig:Penroserule} and \ref{fig:Penrosepatches} below. Additional examples are explored in e.g. \cite[\S6]{BaakeGrimm-2013}, \cite{Robinson-2004}. This geometric version, generally known as a substitution tiling, plays a similar role in the study of aperiodic order as substitution sequences do in one-dimensional symbolic dynamics. Note that the techniques and computations developed in this section carry over, with only minor modifications, to the setting of substitution sequences over a finite alphabet.

\subsection{Primitive substitution rules and tilings}
Recall our definitions of a tiling $\tau$ and tiles in $\R^d$, see Definition \ref{def:tiling}. We consider tiles to be of the same \emph{type} if they differ by a translation, and as in \S\ref{sec:FLC_topology_etc}, we use the term \emph{prototiles} to refer to a set of representatives of the resulting equivalence classes. In some contexts, including in tile-counting estimates as in \S\ref{subsec:tile_counting}, it may be beneficial to allow a larger group of isometries, see Remark \ref{rem:prototiles}.

Similarly to the symbolic dynamics setup, given a set of tiles $\AA$, we let $\AA^*$ denote the collection of all finite patches whose tiles are translation equivalent to tiles in $\AA$. Every element of $\AA^*$ tiles a bounded subsets of $\R^d$.

\begin{definition}\label{def:SubRule}
Let $\xi>1$ be fixed and let $\AA=\{T_1,\ldots,T_n\}$ be a set of $d$-dimensional tiles. A \emph{substitution rule} is a map $\rho:\AA\to\xi^{-1}\AA^*$ satisfying ${\supp(T_i) = \supp(\rho(T_i))}$ for every $i$. That is, $\rho$ provides a set of dissection rules defined on the tiles in $\AA$, which partition each tile in $\AA$ into rescaled copies of tiles from $\AA$. The constant $\xi$ is called the \emph{inflation constant} of $\rho$.
\end{definition}

\begin{example}\label{ex:Penrose}
    The Penrose-Robinson substitution rule $\rho$ is illustrated in Figure \ref{fig:Penroserule}. Here, the substitution rule is applied to thin and wide triangles. A thin triangle is mapped to a patch consisting of two rescaled thin triangles and a single rescaled wide triangle. A wide triangle is mapped to a patch consisting of a single rescaled thin triangle and a single rescaled wide triangle. The scaling constant $\xi$ is computed in Exercise \ref{ex:Penrose_subs}.      
\begin{figure}[h!]
    \centering
    \includegraphics[scale=0.5]{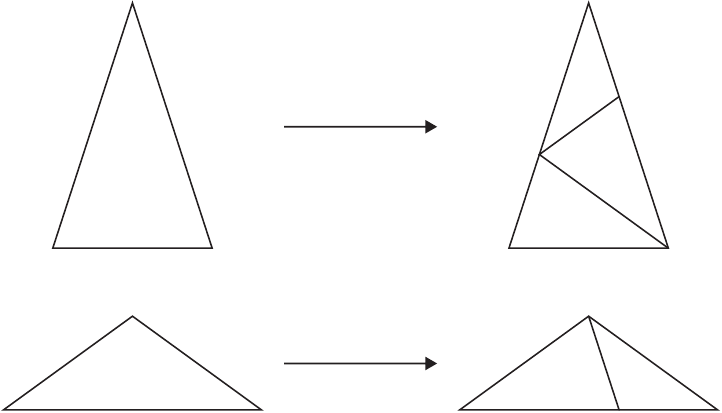}
    \caption{The Penrose-Robinson substitution rule in $\R^2$.}
    \label{fig:Penroserule}
\end{figure}

Strictly speaking, since each of the thin and wide triangles can appear in multiple orientations, the associated set $\AA$ is rather large (but finite). Nevertheless, as discussed in Remark \ref{rem:prototiles}, for the results of this section, it can be regarded as a rule on two tiles only. See \cite[\S6.2]{BaakeGrimm-2013} and the reference therein for further details about this construction. 
\end{example}

The mapping $\rho$ is naturally extended to elements of $\AA^*$, as well as to infinite tilings with tiles in $\AA$, by applying $\rho$ to each tile separately. 

\begin{definition}\label{Def:subsTiling}
Let $\rho$ be a substitution rule defined on a finite set $\AA$ of tiles in $\R^d$. The collection of \emph{legal patches} is defined by 
\begin{equation*}
    \mathcal{P}_\rho=\left\{(\xi \rho)^m(T)\mid m\in\N\:,\: T\in\AA\right\}.
\end{equation*}
The substitution \emph{tiling space} $X_\rho$ is the set of all tilings of $\R^d$ such that every patch of the tiling is a subpatch of some element of $\mathcal{P}_\rho$. The tilings $\tau\in X_\rho$ are called \emph{substitution tilings} associated with $\rho$, with prototiles from the set $\AA$. 
\end{definition}

\begin{figure}[h!]
    \centering
    \includegraphics[scale=0.4]{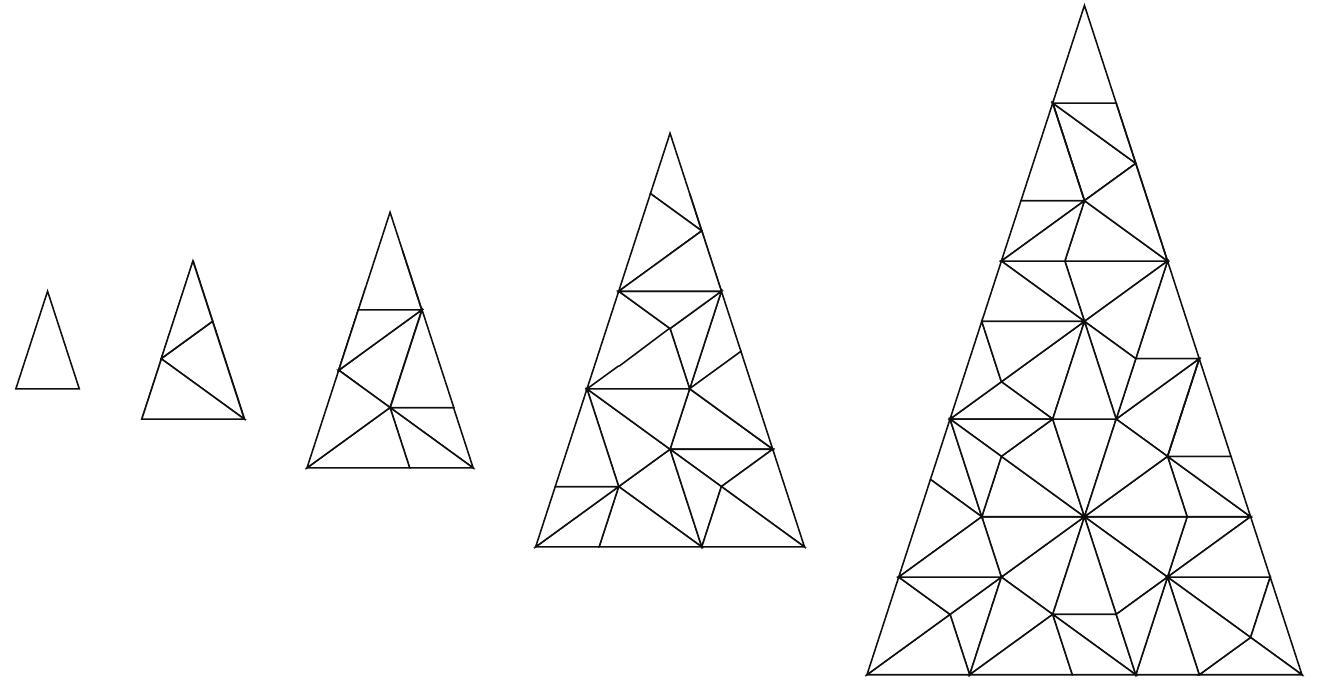}
    \caption{Legal patches for the Penrose-Robinson substitution rule of Example \ref{ex:Penrose}.}
    \label{fig:Penrosepatches}
\end{figure}

The definitions given above clearly describe how to tile bounded regions using repeated substitutions, but they are somewhat implicit regarding how to obtain a tiling of $\R^d$ from these finite tilings. In particular, although it is true, it is not immediately clear from the definition that $X_\rho \neq \emptyset$. Several explicit constructions exist, relying on different compactness arguments, to extend the finite tilings generated by the substitution to a full-space tiling. One approach uses the topology defined in \S \ref{subsec:spaces_of_Delone_sets}, while another is described in the following exercise.  

\begin{exe}\label{ex:Konig} 
    Define a locally finite tree whose vertices correspond to legal patches, and apply K\H{o}nig’s Lemma (see, e.g., \cite{Diestel-2025}  or \cite{Levy-2002}) to construct tilings of $\R^d$, thereby showing that indeed $X_\rho \neq \emptyset$.
\end{exe}

\begin{definition}\label{Def:SubMatrix+Primitive}
Let $\rho$ be a substitution rule on a set of tiles $\AA=\{T_1, \ldots,T_n\}$ in $\R^d$. Define the \emph{substitution matrix} $M_\rho=(a_{ij})$, an $n\times n$ matrix, where $a_{ij}$ is the number of copies of $\xi^{-1}T_i$ in $\rho(T_j)$. 
\end{definition}   

A substitution rule $\rho$ is called \emph{primitive} if the matrix $M_\rho$ is a primitive matrix, that is, if there exists an $m\in \N$ such that all entries of $M_\rho^m$ are strictly positive (see e.g. \cite{Queffelec-2010}). In such a case, the set of prototiles of any tiling in $X_\rho$ is precisely $\AA$, and so $\AA$ is referred to as the set of prototiles of $\rho$. Primitivity is often assumed in the study of substitution systems and has many implications. In particular, it allows the application of the powerful Perron--Frobenius theorem, as we will see shortly. Primitivity of the Penrose-Robinson substitution rule is illustrated in Exercise \ref{ex:Penrose_subs}.

We record the following lemma, which is an important feature of primitive substitution tilings that appears as an exercise in \cite[Exercise~5]{Robinson-2004}. We present a short sketch of the proof, which follows directly from the definition of the substitution tiling space. Informally, it states that every tiling $\tau \in X_\rho$ admits a $\rho$-preimage.

\begin{lem}\label{lem:tau_m_Def}
Let $\rho$ be a primitive substitution rule. Then for every $\tau\in X_\rho$ and for every $m\in\N$, there exists a tiling $\tau_m\in X_\rho$ that satisfies $(\xi \rho)^m(\tau_m)=\tau$.  
\end{lem} 

\begin{proof}[Sketch of proof]
Let $\tau \in X_\rho$. It is clearly sufficient to find a 
tiling $\tau_1 \in X_\rho$ such that $(\xi \rho)(\tau_1) = \tau$. 
For every $n \in \N$, consider the patch $Q_n$ consisting of all tiles in $\tau$ that intersect the ball $B(0,n)$. 
By Definition~\ref{Def:subsTiling}, for each $n$, there exists a legal patch $P_n \in \mathcal{P}_\rho$ that contains $Q_n$ as a subpatch. 
Recall that $P_n = (\xi \rho)^{j_n}(T_{i_n})$ for some integer $j_n$ and $T_{i_n} \in \AA$. 
Now, consider the sequence of patches $P_n' = (\xi \rho)^{j_n - 1}(T_{i_n})$. 
Using either the topology induced by the metric in~\eqref{eq:metric} or K\H{o}nig’s Lemma (as in Exercise~\ref{ex:Konig}), one can extract a convergent subsequence whose limit defines a tiling $\tau_1$ of $\R^d$, which implies the assertion.
\end{proof}

 \subsection{Tile-counting estimates} \label{subsec:tile_counting}
 In order to verify either the Burago--Kleiner condition or Laczkovich's criterion, one needs good estimates on the quantity $\#(\Lambda \cap U)$ for certain target sets $U$. In the context of Delone sets associated with substitution tilings, such counting problems can be approached conveniently and efficiently using linear algebra. The ideas presented below are drawn from \cite{Solomon-2011} and \cite{Solomon-2014}, which provide additional details.

We use vectors with positive integer entries to record the number of tiles of each type appearing in a given patch. Specifically, if there are $n$ types of tiles, each patch is naturally associated with a vector in $\mathbb{N}^n$, whose $i$-th coordinate counts the number of tiles of type $i$ within the patch. For instance, if $e_i$ denotes the $i$-th vector of the standard basis of $\mathbb{R}^n$, then $e_i$ represents a patch containing a single tile of type $i$. Note that the patch cannot be reconstructed from this vector, which serves as an abelianized representation of the patch, retaining only the number of tiles of each type and forgetting their spatial arrangement.

\begin{exe}\label{ex:vector_representation}
    Let $\rho$ be a substitution rule on a set of tiles $\AA=\{T_1, \ldots,T_n\}$ in $\R^d$ and let $P$ be a patch in $\tau\in X_\rho$ represented by the vector $w_P\in\N^n$. Show that for every integer $k\ge 1$, the vector that represents the patch $(\xi\rho)^k(P)$ is $M_\rho^k(w_P)$. 
\end{exe}

Let us now fix our notation and conventions. From here on, we let $\rho$ denote a primitive substitution rule with an inflation factor $\xi > 1$, defined on a set of prototiles $\AA=\{T_1, \ldots,T_n\}$ in $\R^d$. Assume the prototiles have $d$-dimensional volumes $s_1, \ldots, s_n$, and set $u_{\vol} = (s_1,\ldots,s_n)^t$. For a tiling $\tau \in X_\rho$, an associated Delone set $\Lambda=\Lambda_\tau\subset \R^d$ is defined by picking a point in each tile of $\tau$.

Let $M_\rho$ be the substitution matrix, and denote its eigenvalues by $\lambda_1, \ldots, \lambda_n$, ordered so that $|\lambda_1| \ge \cdots \ge |\lambda_n|$. By the Perron--Frobenius theorem, $\lambda_1$ is a simple eigenvalue and has a positive right eigenvector $v_1$. We work over $\C$ and fix a Jordan basis of $M_\rho$. I.e., a Jordan chain associated with an eigenvalue $\lambda$
has the form $v_{i_1},v_{i_2},\ldots,v_{i_r}$, 
where
\[
(M_\rho-\lambda I)v_{i_1}=0
\quad\text{and}\quad
(M_\rho-\lambda I)v_{i_q}=v_{i_{q-1}}
\qquad (2\le q\le r).
\]
We say that every vector in this chain corresponds to $\lambda$.
For a vector $v_{i_q}$ in this chain, set
\begin{equation}\label{eq:k_{i_q}}
    k_{i_q}=r-q+1.
\end{equation}

Thus, $k_{i_q}$ is the number of vectors in the Jordan chain from
$v_{i_q}$ to its last vector, including $v_{i_q}$ itself. In
particular, the eigenvector $v_{i_1}$ has $k_{i_1}=r$, while the last
generalized eigenvector has $k_{i_r}=1$. As before, $v_i(j)$ denotes
the $j$-th coordinate of $v_i$.

\begin{lem}\label{lem:u_1}
    The leading eigenvalue  $\lambda_1$ of $M_\rho$ satisfies $\lambda_1 = \xi^d > 1$, and $u_{\vol}$ is the left eigenvector of $M_\rho$ that corresponds to $\lambda_1$.
\end{lem}

\begin{proof}
    The assertion follows from a direct computation. For every $1\le j\le n$, the $j$-th entry of $M_\rho^t u_{\vol}$ is given by 
    \begin{equation}\label{eq:left_eigenvector} 
        (M_\rho^t u_{\vol})(j) = 
        \sum_{i=1}^n a_{ij}s_i = 
        \sum_{i=1}^n \#\{\text{tiles of type } i \text{ in } \rho(T_j) \}\cdot s_i.
    \end{equation}
    By the definition of $\rho$, the expression on the right-hand side of \eqref{eq:left_eigenvector} is equal to $\vol(\xi T_j)=\xi^ds_j$. This holds for all $1\leq j\leq n$ and so $M_\rho^t u_{\vol} = \xi^du_{\vol}$. Since $M_\rho$ and $M_\rho^t$ have the same eigenvalues, and since the Perron--Frobenius theorem asserts that the only eigenvalue of $M_\rho^t$ that admits a strictly positive eigenvector is $\lambda_1$, we deduce that $\lambda_1=\xi^d>1$, as required. 
\end{proof}

\begin{remark}\label{rem:prototiles}
    The Penrose-Robinson substitution rule described in Example \ref{ex:Penrose} determines a partition of thin and wide tiles regardless of their orientation. While there are more than two (but still finitely many) translation equivalence classes of tiles in patches and tilings generated by this substitution rule, for the tile-counting results described here, there is no difference if translation equivalence is extended to isometries. In other words, in this context, we may consider the Penrose-Robinson substitution rule to have two prototiles only, and a substitution matrix of order $2$.
\end{remark}

\begin{exe}\label{ex:Penrose_subs}
    Find a $2\times 2$ substitution matrix for the Penrose-Robinson substitution rule illustrated in Figure \ref{fig:Penroserule} and verify that the rule is primitive. Compute $\lambda_1$ and $u_{\vol}$ and deduce that the inflation constant is the golden ratio. Using Figure \ref{fig:Penrosepatches}, verify the assertion of Exercise \ref{ex:vector_representation} for $T$, the thin triangle, and $k=1,\ldots,4$.
\end{exe}

The Perron--Frobenius theorem implies something even stronger, namely that the span of $\{v_2,\ldots,v_n\}$ intersects the positive cone  
$\R^n_{\ge 0} = \{v\in\R^d ~\mid~ \forall i,~ v(i)\ge 0\}$ only at the origin (see \cite[\S26]{Hogben-2006} or \cite{SchneiderBarker-1989}).  
In particular, this implies that for every $v\in\R^n_{\ge 0}$, if  
$v=\sum_{i=1}^n c_i v_i$ is expressed as a linear combination of the Jordan basis of $M_\rho$, then $c_1 \neq 0$.  
This observation plays an important role in the following lemma, which relates $v_1$ to the leading term in tile-counting asymptotics. 

\begin{lem}\label{lem:v_1_main_term}
    For every $i,j\in\{1,\ldots,n\}$ and every integer $k\ge 1$, the proportional number of tiles of type $j$ contained in a legal patch of the form $(\xi\rho)^k(T_i)$, approaches $\frac{v_1(j)}{\inpro{v_1}{\mathbf{1}}}$ at an exponential rate. More precisely, we have 
    \[
    \absolute{\frac{\#\{\text{tiles of type }j \text{ in }(\xi\rho)^k(T_i)\}}{\#\{\text{tiles in }(\xi\rho)^k(T_i)\}}
    - \frac{v_1(j)}{\inpro{v_1}{\mathbf{1}}}} = \OO\left( \frac{\lambda_2^k \cdot k^{r-1}}{\lambda_1^k} \right),
    \]
    where $r$ is the maximal size of a Jordan block of $\lambda_2$, and $\mathbf{1}=(1,1,\ldots,1)^t$.
\end{lem}

\begin{proof}
    Let $P=(\xi\rho)^k(T_i)$, for some $i\in\{1,\ldots,n\}$ and $k\ge 1$. As in Exercise \ref{ex:vector_representation}, the number of tiles of each type in $P$ is represented by the vector $M_\rho^k(e_i)$, that is, 
    \[
    \#\{\text{tiles of type }j \text{ in } P\} = [M_\rho^k(e_i)](j).
    \]
    Writing $e_i = \sum_{\ell=1}^n c_\ell v_\ell$, a linear combination of the Jordan basis of $M_\rho$, we obtain
    \[
    M_\rho^k(e_i) = M_\rho^k\left(\sum_{\ell=1}^n c_\ell v_\ell \right) = \sum_{\ell=1}^nc_\ell M_\rho^k v_\ell = 
    c_1\lambda_1^kv_1 + \sum_{\ell=2}^nc_\ell M_\rho^k v_\ell.
    \]
    As mentioned above, we have $c_1\neq 0$, and so 
    \[
    \#\{\text{tiles of type }j \text{ in }P\} = c_1\lambda_1^kv_1(j) + \sum_{\ell=2}^n c_\ell [M_\rho^k v_\ell](j) = c_1\lambda_1^kv_1(j) + \OO(\lambda_2^k\cdot k^{r-1}). 
    \]
    It follows that 
    \[
    \#\{\text{tiles in }P\} = \inpro{c_1\lambda_1^kv_1 + \sum_{\ell=2}^n c_\ell M_\rho^k v_\ell}{\mathbf{1}} = c_1\lambda_1^k\inpro{v_1}{\mathbf{1}} + \OO(\lambda_2^k\cdot k^{r-1}), 
    \]
    and the proof is complete.
\end{proof}

In view of Lemma \ref{lem:v_1_main_term}, the asymptotic density of $\Lambda$ must be equal to 
\begin{equation}\label{eq:alpha}
\alpha = 
\frac{\sum_{i=1}^n v_1(i)}{\sum_{i=1}^n v_1(i)\, s_i}
= \frac{\langle \mathbf{1}, v_1 \rangle}{\langle u_{\vol}, v_1 \rangle}.
\end{equation}
Using the same formalism, we can express the discrepancy of a general patch in linear-algebraic terms. 
Let $\mathscr{P}$ denote the set of all patches of $\tau$. For every patch $P \in \mathscr{P}$, let $w_P$ be the vector representing the number of tiles of each type in $P$. Then 
\begin{equation*}
\#(\Lambda \cap P) = \sum_{i=1}^n w_P(i) = \inpro{\mathbf{1}}{w_P}, 
\quad \text{and} \quad 
\vol(P) = \sum_{i=1}^n w_P(i)\, s_i = \inpro{u_{\vol}}{w_P}.
\end{equation*}
The discrepancy of $P$ therefore satisfies 
\[
\disc{\alpha}{\Lambda}{P} = \absolute{F_\rho(w_P)},
\]
where $F_\rho(\cdot)$ is the following linear functional:
\begin{equation*}
F_\rho(w_P) = \inpro{\mathbf{1}}{w_P} - \frac{\inpro{\mathbf{1}}{v_1}}{\inpro{u_{\vol}}{ v_1}}\,
  \inpro{u_{\vol}}{w_P}.
\end{equation*}

The next consequence, concerning the Jordan basis $(v_1,\ldots,v_n)$ of $M_\rho$, plays an important role in discrepancy estimates. 

\begin{lem}\label{lem:lambda_t}
For every Jordan-basis vector $v_i$ corresponding to an eigenvalue
$\lambda_i\neq\lambda_1$, we have
\[
\inpro{u_{\vol}}{v_i}=0
\quad\text{and consequently}\quad
F_\rho(v_i)=\inpro{\mathbf{1}}{v_i}.
\]
In particular, if $v_i\in\mathbf{1}^\perp$, then $F_\rho(v_i)=0$.
Furthermore, $F_\rho(v_1)=0$.
\end{lem}

\begin{proof}
Clearly,
\[
F_\rho(v_1)
=
\inpro{\mathbf{1}}{v_1}
-
\frac{\inpro{\mathbf{1}}{v_1}}
     {\inpro{u_{\vol}}{v_1}}
\inpro{u_{\vol}}{v_1}
=0.
\]

Let $v_i$ be a generalized eigenvector corresponding to
$\lambda_i\neq\lambda_1$. Choose $q\ge 1$ such that
\[
(M_\rho-\lambda_i I)^qv_i=0.
\]
Since $u_{\vol}^tM_\rho=\lambda_1u_{\vol}^t$, we have
\[
0
=
u_{\vol}^t(M_\rho-\lambda_i I)^qv_i
=
(\lambda_1-\lambda_i)^q u_{\vol}^tv_i.
\]
Since $\lambda_i\neq\lambda_1$, it follows that $\inpro{u_{\vol}}{v_i}=0$. Therefore,
$F_\rho(v_i) = \inpro{\mathbf{1}}{v_i}$.
In particular, this vanishes whenever $v_i\in\mathbf{1}^\perp$.
\end{proof}

Having established the necessary preliminaries, we can now state the
main estimate for the discrepancy of inflated tiles. For an eigenvalue
$\lambda$ of $M_\rho$, let
\[
W_\lambda:=\ker(M_\rho-\lambda I)^n
\]
denote the generalized eigenspace corresponding to $\lambda$.
Let $t\ge 2$ be the minimal index such that
\[
W_{\lambda_t}\nsubseteq\mathbf{1}^\perp.
\]
Using the numbers $k_i$ associated with the Jordan-basis vectors as
defined in \eqref{eq:k_{i_q}}, set
\begin{equation}\label{eq:visible_Jordan_length}
k
=
\max\left\{
k_i
\,\middle|\,
v_i\in W_{\lambda_t}
\text{ and }
v_i\notin\mathbf{1}^\perp
\right\}.
\end{equation}
Thus, $k-1$ is the maximal polynomial degree produced by a
Jordan-basis vector corresponding to $\lambda_t$ that is detected by
the discrepancy functional $F_\rho$.

As before, the Delone set $\Lambda=\Lambda_\tau$ corresponds to a
fixed tiling $\tau\in X_\rho$. For $m\in\N$, we denote by $\tau_m$ the
$m$-th preimage of $\tau$ under $\rho$, as in Lemma
\ref{lem:tau_m_Def}, and by $\mathscr{T}^{(m)}$ the set of all
tiles of $\tau_m$.

\begin{thm}\label{thm:Discrepancy_Estimate_substitution}
Let $t\ge 2$ be the minimal index such that $W_{\lambda_t}\nsubseteq\mathbf{1}^\perp$and let $k$ be as in
\eqref{eq:visible_Jordan_length}. Then there exist constants
$c_1,c_2>0$, depending only on the parameters of the substitution
rule, with the following properties:
\begin{enumerate}
\item\label{item_thm:lower_bound}
There exist $j\in\{1,\ldots,n\}$ and $m_0\in\N$ such that for every
$m\ge m_0$ and every $T\in\mathscr{T}^{(m)}$ of type $j$,
\begin{equation}\label{eq:disc_lower_bound}
c_1\,m^{k-1}\absolute{\lambda_t}^{m}
\le
\disc{\alpha}{\Lambda}{T}.
\end{equation}

\item\label{item_thm:upper_bound}
For every $m\in\N$ and every $T\in\mathscr{T}^{(m)}$,
\begin{equation}\label{eq:disc_upper_bound}
\disc{\alpha}{\Lambda}{T}
\le
c_2\,m^{k-1}\absolute{\lambda_t}^{m}.
\end{equation}
\end{enumerate}
\end{thm}

\begin{proof}
Let $T\in\mathscr{T}^{(m)}$ be of type $j$. Then
$w_T=M_\rho^me_j$. 
Write
\[
M_\rho^me_j
=
\sum_{i=1}^n c_i^{(j)}(m)v_i
\]
in the fixed Jordan basis. 
Suppose that
$
v_{i_1},\ldots,v_{i_r}
$
is a Jordan chain associated with an eigenvalue $\lambda$. For
$1\le q\le r$, the standard formula for powers of a Jordan block gives
\[
M_\rho^m v_{i_q}
=
\sum_{p=0}^{q-1}
\binom{m}{p}\lambda^{m-p}v_{i_{q-p}}.
\]
Consequently, the coefficient of a given vector $v_i$ in
$M_\rho^me_j$ is bounded by a constant times
\[
m^{k_i-1}\absolute{\lambda_i}^{m}.
\]

By Lemma \ref{lem:lambda_t},
\[
F_\rho(v_i)=\inpro{\mathbf{1}}{v_i}
\]
for every non-Perron Jordan-basis vector. Moreover, by the definition
of $t$, every Jordan-basis vector corresponding to an eigenvalue that
precedes $\lambda_t$ belongs to $\mathbf{1}^\perp$ and hence is
annihilated by $F_\rho$. By the definition of $k$, the contribution
from the generalized eigenspace $W_{\lambda_t}$ is therefore bounded
by
\[
C\,m^{k-1}\absolute{\lambda_t}^{m}.
\]
The contributions from eigenvalues of smaller modulus can be absorbed
into the same bound. Hence
\[
\disc{\alpha}{\Lambda}{T}
=
\absolute{F_\rho(M_\rho^me_j)}
\le
c_2\,m^{k-1}\absolute{\lambda_t}^{m},
\]
which proves \eqref{item_thm:upper_bound}.

For the lower bound, choose an index $\ell$ satisfying
\[
v_\ell\in W_{\lambda_t}\setminus\mathbf{1}^\perp
\quad\text{and}\quad
k_\ell=k,
\]
where $k$ is as in \eqref{eq:visible_Jordan_length}. The coefficient of $m^{k-1}\lambda_t^m$ in
$F_\rho(M_\rho^m x)$, viewed as a linear functional of $x\in \C^n$,
is nonzero: indeed, a vector situated $k-1$ places after $v_\ell$ in
its Jordan chain produces a nonzero multiple of
\[
m^{k-1}\lambda_t^m F_\rho(v_\ell),
\]
and $F_\rho(v_\ell)=\inpro{\mathbf{1}}{v_\ell}\neq0$.
Since $e_1,\ldots,e_n$ form a basis, this leading-coefficient
functional is nonzero on at least one standard basis vector $e_j$.
For that index $j$,
\[
\absolute{F_\rho(M_\rho^me_j)}
\ge
c_1\,m^{k-1}\absolute{\lambda_t}^{m}
\]
for all sufficiently large $m$, proving
\eqref{item_thm:lower_bound}.
\end{proof}

\subsection{Rectifiability and conditions for uniform spreadness}

We now apply the estimates in Theorem \ref{thm:Discrepancy_Estimate_substitution} to study rectifiability and uniform spreadness of Delone sets that correspond to primitive substitution tilings. As we will see, while $\absolute{\lambda_2}<\lambda_1$ is enough to imply rectifiability, additional algebraic constraints on the eigenvalues of $M_\rho$ are required to establish uniform spreadness.


\begin{thm}\label{thm:BL_substitution}
    Every Delone set that corresponds to a primitive substitution tiling is rectifiable.
\end{thm}

\begin{proof}
    Throughout the proof, the symbol $c$ denotes a constant depending only on the initial parameters of the tiling, whose value may change from equation to equation.
    
    Let $\rho$ be a primitive substitution rule on the $d$-dimensional tiles $\{T_1,\ldots,T_n\}$, let $\tau\in X_\rho$ and let $\Lambda=\Lambda_\tau$ be a corresponding Delone set. We claim that $\Lambda$ satisfies the Burago--Kleiner condition from Theorem \ref{thm:BK_sufficient_condition}. That is, we show that 
    \begin{equation}
    \sum_{\ell=1}^\infty \Delta_\Lambda(\alpha,2^\ell) < \infty,
    \end{equation}
    where 
    \begin{equation}
    \Delta_\Lambda(\alpha,R) = \sup_{x\in R\Z^d}\left\{\frac{\disc{\alpha}{\Lambda}{Q(x,R)}}{\alpha\cdot\vol(Q(x,R))}\right\},
    \end{equation} 
    and $Q(x,R)\subset\R^d$ is a closed cube $x+[0,R]^d$ of edge-length $R$.

     Fix $\varepsilon>0$ such that 
     \[
    \absolute{\lambda_2}+\varepsilon <\lambda_1.
     \]
     For $R>0$, let $m$ be the minimal integer for which $R\le \xi^{2m}$, and let $\tau_m$ be an $m$-th preimage under $\rho$, as in Lemma \ref{lem:tau_m_Def}. Note that for the choice $R=2^\ell$, we can write 
    $m=\Theta(\ell)$. Given an $R$-cube $Q(R)=Q(x,R)$, let $P$ be the set of all tiles of $\tau_m$ that are contained in $Q(R)$, and let $U\subseteq Q(R)$ denote their union (see Figure \ref{fig:substitution_BL_proof}). 
    
\begin{figure}[h!]
    \centering
    \includegraphics[scale=0.3]{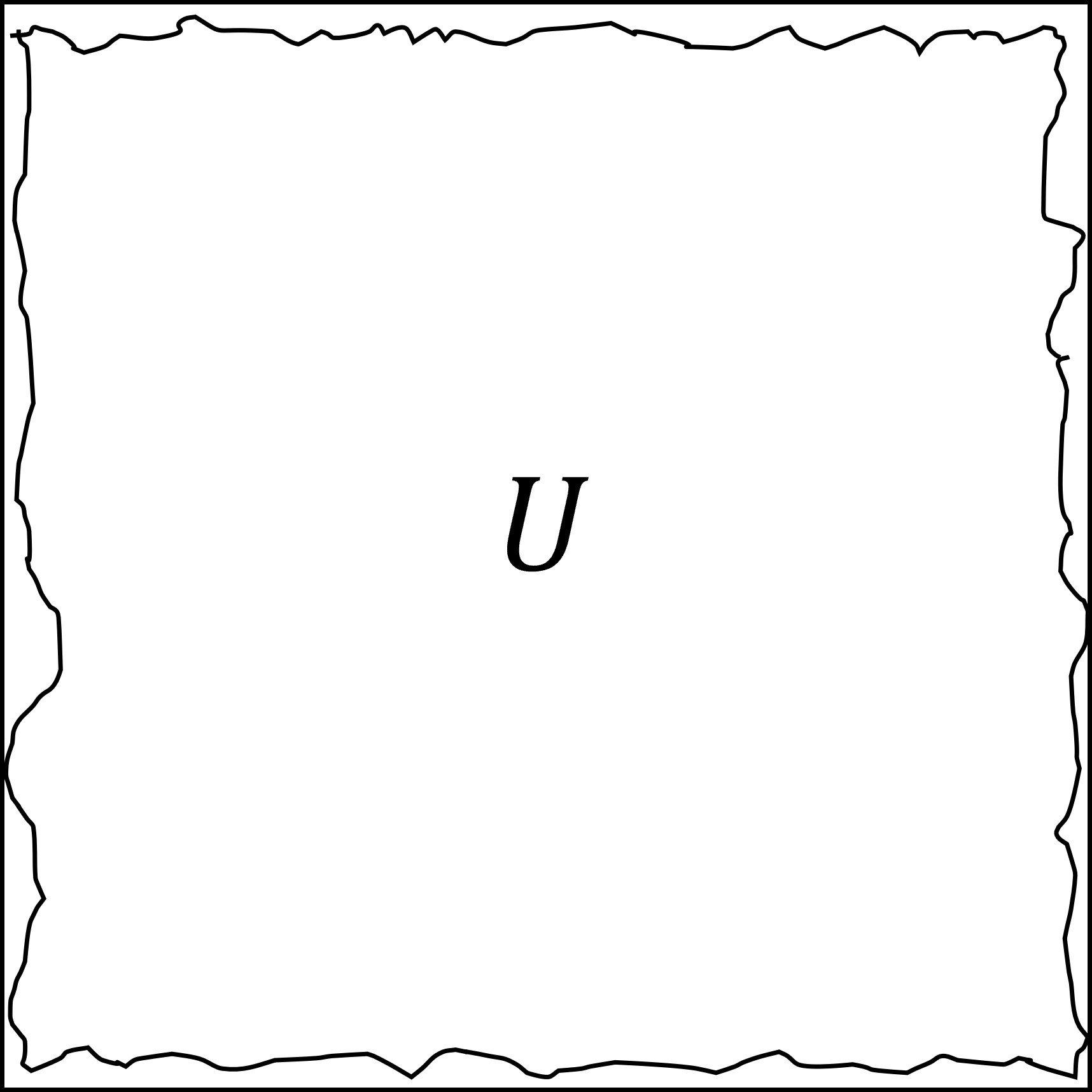}
    \caption{The region $U$ is the union of tiles of $\tau_m$ that are contained in $Q(R)$.}
    \label{fig:substitution_BL_proof}
\end{figure}

    Note that the volume of the tiles of $\tau_m$ is $\Theta((\xi^d)^m) = \Theta(\lambda_1^m)$, thus $P$ contains $\Theta\left(\vol(Q(R))/\lambda_1^m\right) = \Theta(\lambda_1^{2m}/\lambda_1^m) = \Theta(\lambda_1^m)$ tiles of $\mathscr{T}^{(m)}$. For each such tile $T$, the estimates of Theorem \ref{thm:Discrepancy_Estimate_substitution} hold, and by \eqref{eq:disc_upper_bound}, for every $\varepsilon>0$ we have 
    \[
    \disc{\alpha}{\Lambda}{T} = \OO((\absolute{\lambda_2}+\varepsilon)^m).
    \]
    Using the fact that $\disc{\alpha}{\Lambda}{A\cup B} \le \disc{\alpha}{\Lambda}{A}+\disc{\alpha}{\Lambda}{B}$, for disjoint sets $A$ and $B$, we deduce that 
    \begin{equation}\label{eq:disc_U}
    \disc{\alpha}{\Lambda}{U} \le \sum_{T\in P}\disc{\alpha}{\Lambda}{T} = \OO(\lambda_1^m\cdot (\absolute{\lambda_2}+\varepsilon)^m).
    \end{equation}
   On the other hand, the maximal diameter of a tile in $\tau_m$ is $\OO(\xi^m) = \OO(\sqrt{R})$. Hence, $\vol(Q(R)\minus U) = \OO(R^{d-\frac12})$ and $\#(\Lambda\cap (Q(R)\minus U)) = \OO(R^{d-\frac12})$, and so
    \begin{equation}\label{eq:disc_U^c}
    \disc{\alpha}{\Lambda}{Q(R)\minus U} = \OO(R^{d-\frac12}) = \OO(\vol(Q(R))/\xi^m).
    \end{equation}
    
    Combining \eqref{eq:disc_U} and \eqref{eq:disc_U^c} implies that  
    \[
     \frac{\disc{\alpha}{\Lambda}{Q(R)}}{\alpha\cdot\vol (Q(R))}
    \le 
    \frac{\disc{\alpha}{\Lambda}{U} + \disc{\alpha}{\Lambda}{Q(R)\minus U}}{\alpha\cdot\vol (Q(R))}
    = \OO\left(\left(\frac{\absolute{\lambda_2}+\varepsilon}{\lambda_1}\right)^m + \frac{1}{\xi^m}\right).   
    \] 
    As $\varepsilon<\lambda_1 - \absolute{\lambda_2}$, it follows that  
    \[
    \sum_{\ell=1}^\infty \Delta_\Lambda(\alpha,2^\ell) 
    =\sum_{\ell=1}^\infty \OO\left(\left(\frac{\absolute{\lambda_2}+\varepsilon}{\lambda_1}\right)^{\ell} + \frac{1}{\xi^{\ell}}\right) < \infty,
    \]
    and so the Burago--Kleiner rectifiability condition is satisfied, as required.
\end{proof}

\begin{open}
    In view of Corollary \ref{cor:Jac_implies_BL}, McMullen asked\footnote{Personal communication.} whether the function $f_\Lambda$ in \eqref{eq:f_Lambda} that arises from the Penrose tiling can be realized a.e. as a Jacobian. This was confirmed in \cite{Solomon-2013}, where it was shown that $f_\Lambda$ can be realized a.e. as a Jacobian for every primitive substitution tiling by star-shaped tiles. Can this result be generalized to any substitution tiling? 
\end{open}

To establish uniform spreadness, the coarse estimates employed above will not suffice. The following result was established in \cite{Solomon-2014}.  
For the positive result in \ref{item_thm:is_uniformly_spread} below, the proof relies on a stronger regularity assumption on the tiles, namely that they are \emph{biLipschitz homeomorphic} to closed balls. This is required for the technical step in Lemma \ref{lem:economic_packing} below, which was originally proved by Laczkovich in \cite{Laczkovich-1992} for sets $U\in \HHH$, providing an economical packing of $U$ by dyadic cubes. It was later extended to the setting of substitution tilings in \cite{Solomon-2014}, which is the version stated here. 

\begin{thm}\label{thm:BD_substitution}
Suppose that $\Lambda$ is a Delone set arising from a primitive
substitution rule on a set of tiles that are biLipschitz homeomorphic
to closed balls in $\R^d$. Let $t\ge2$ be the minimal index such that $W_{\lambda_t}\nsubseteq\mathbf{1}^\perp$.
\begin{enumerate}
\item\label{item_thm:is_uniformly_spread}
If
$
\absolute{\lambda_t}<\lambda_1^{\frac{d-1}{d}},
$
then $\Lambda$ is uniformly spread.

\item\label{item_thm:is_not_uniformly_spread}
If
$
\absolute{\lambda_t}>\lambda_1^{\frac{d-1}{d}},
$
then $\Lambda$ is not uniformly spread.
\end{enumerate}
\end{thm}

Let $\rho$ denote a primitive substitution with prototiles $\{T_1,\ldots,T_n\}$, now assumed to be biLipschitz homeomorphic to closed balls in $\R^d$. Fix a substitution tiling $\tau\in X_\rho$ and an associated Delone set $\Lambda=\Lambda_\tau$, and as before let $\mathscr{P}$ denote the collection of all patches in $\tau$. We denote by $\mathscr{T} = \bigcup_{m=0}^\infty\mathscr{T}^{(m)}$ the set of all inflated tiles of all the tilings $\tau_m$, as described by Lemma \ref{lem:tau_m_Def}. For a set of tiles $\mathcal{D}\subset \mathscr{T}$, we denote by $S(\mathcal{D})$ the closure of $\mathcal{D}$ under the operations of disjoint union and proper difference (that is, the operation $A\minus B$ only for sets $A,B$ that satisfy $B\subseteq A$), where every element of $\mathcal{D}$ can be used only once (compare \cite[pp.~40--41]{Laczkovich-1992} for a constructive definition).

\begin{lem}\label{lem:economic_packing}
    Given any patch $P\in \mathscr{P}$ with $U=\supp(P)$, and an inflated tile $T\in \mathscr{T}$ that satisfies $U\subset T$ and $\vol(U)\le \frac12\vol(T)$, there exist $\tilde T_1,\ldots,\tilde T_s\in \mathscr{T}$ such that 
    \begin{enumerate}
        \item 
        $\tilde T_i\subset T$ for every $1\le i\le s$.
        \item\label{item_lem:eco_packing_2}
        $U\in S(\{\tilde T_1,\ldots,\tilde T_s\})$.
        \item\label{item_lem:eco_packing_3}
        There is a constant $C$, depending only on the parameters of the tiling, such that for every $m\in\N$ we have 
        \[
        \#\{i~\mid~ \tilde T_i\in \mathscr{T}^{(m)}\}\le C\cdot\frac{\vol_{d-1}(\partial U\cap \inter(T))}{\xi^{m(d-1)}}.
        \]
    \end{enumerate}
\end{lem}

The reader is referred to \cite[Proposition~3.5]{Solomon-2014} for the proof of Lemma \ref{lem:economic_packing}, which we do not include here. Relying on the lemma, we deduce Theorem \ref{thm:BD_substitution}.

\begin{proof}[Sketch of proof of Theorem \ref{thm:BD_substitution}]
First notice that for any tile $T\in\mathscr{T}^{(m)}$ we have 
\[
\vol(T) = \Theta((\xi^d)^m)=\Theta(\lambda_1^m) ~ \text{ and }~ 
\vol_{d-1}(\partial T) = \Theta((\xi^{d-1})^m)=\Theta\left(\left(\lambda_1^{\frac{d-1}{d}}\right)^m\right).
\]
Thus, \eqref{item_thm:is_not_uniformly_spread} is a direct consequence of \eqref{item_thm:lower_bound} of Theorem \ref{thm:Discrepancy_Estimate_substitution} together with condition \eqref{item_thm:Laczkovich_1} of Laczkovich's criterion in Theorem \ref{thm:Laczkovich}. We therefore turn to the proof of \eqref{item_thm:is_uniformly_spread}.

We prove \eqref{item_thm:Laczkovich_2} of Theorem \ref{thm:Laczkovich} for sets of the form $U = \supp(P)$ with some $P\in\mathscr{P}$. Given such a set $U$, we fix $T$ as in Lemma \ref{lem:economic_packing}, to obtain the inflated tiles  $\tilde T_1,\ldots,\tilde T_s \in \mathscr{T}$. Observe that whenever either $B\subseteq A$ or $A\cap B=\emptyset$ holds, the discrepancy is subadditive, in the sense that 
    \begin{align*}
    & \disc{\alpha}{\Lambda}{A\minus B} \le \disc{\alpha}{\Lambda}{A} + \disc{\alpha}{\Lambda}{B}, ~\text{ or }
    \\
    & \disc{\alpha}{\Lambda}{A\cup B} \le \disc{\alpha}{\Lambda}{A} + \disc{\alpha}{\Lambda}{B},
    \end{align*}   
respectively. By \eqref{item_lem:eco_packing_2} of Lemma \ref{lem:economic_packing} and the definition of $S(\mathcal{D})$, we therefore obtain that
\[
\disc{\alpha}{\Lambda}{U}\le \sum_{i=1}^s \disc{\alpha}{\Lambda}{\tilde T_i}.
\]
Each $\tilde T_i$ belongs to some $\mathscr{T}^{(m)}$. Thus, combining \eqref{item_thm:upper_bound} of Theorem \ref{thm:Discrepancy_Estimate_substitution} and \eqref{item_lem:eco_packing_3} of Lemma \ref{lem:economic_packing} yields 

\begin{align*}
& \sum_{i=1}^s
\disc{\alpha}{\Lambda}{\tilde T_i}
\\
&\le
\sum_{m=0}^\infty
\#\{i \mid \tilde T_i\in\mathscr{T}^{(m)}\}
\cdot c_2\,m^{k-1}\absolute{\lambda_t}^{m}
\\
&\le
\sum_{m=0}^\infty
C c_2\,
\vol_{d-1}(\partial U\cap\inter(T))
\frac{m^{k-1}\absolute{\lambda_t}^{m}}
     {\xi^{m(d-1)}}.
\end{align*}
The series converges because $\absolute{\lambda_t}
<
\lambda_1^{\frac{d-1}{d}}
=
\xi^{d-1}
$.
Thus,
$
\disc{\alpha}{\Lambda}{U}
=
\OO\bigl(\vol_{d-1}(\partial U)\bigr),
$
proving \eqref{item_thm:Laczkovich_2} of Theorem
\ref{thm:Laczkovich} and completes the argument.
\end{proof}

\begin{exe}
    Show that it is indeed sufficient to prove \eqref{item_thm:Laczkovich_2} of Theorem \ref{thm:Laczkovich} for sets of the form $U = \supp(P)$, for some $P\in\mathscr{P}$, instead of sets $U\in\HHH$.
\end{exe}

\begin{open}
    Generalize \eqref{item_thm:is_uniformly_spread} of Theorem \ref{thm:BD_substitution} to a broader setting of primitive substitution tilings, namely, to cases with weaker or no regularity assumptions on the tiles.
\end{open}

\begin{remark}
In the borderline case
$
\absolute{\lambda_t}
=
\lambda_1^{\frac{d-1}{d}},
$
there exist examples of both uniformly spread Delone sets and Delone
sets that are not uniformly spread.
With $k$ defined by
\eqref{eq:visible_Jordan_length}, Theorem
\ref{thm:Discrepancy_Estimate_substitution} implies that
$\Lambda$ is not uniformly spread whenever $k\ge2$. Indeed, in this
case the discrepancy of the inflated tiles supplied by
\eqref{item_thm:lower_bound} grows at least as
$
m^{k-1}
\left(\lambda_1^{\frac{d-1}{d}}\right)^m,
$
whereas their boundary measure grows only as
$
\left(\lambda_1^{\frac{d-1}{d}}\right)^m.
$
In particular, $k\ge2$ if $\lambda_t$ has a nontrivial Jordan block
containing at least two generalized eigenvectors that do not belong
to $\mathbf{1}^\perp$. More generally, it is enough that some
Jordan-basis vector $v_i\notin\mathbf{1}^\perp$ corresponding to
$\lambda_t$ satisfies $k_i\ge2$.

The mere existence of a nontrivial Jordan block for $\lambda_t$ is
not sufficient. For example, if the eigenvector at the beginning of
a block lies in $\mathbf{1}^\perp$ and only the final generalized
eigenvector lies outside $\mathbf{1}^\perp$, then the corrected
parameter may be $k=1$, and the above argument gives no conclusion.

Furthermore, in \cite{FrettlohSmilanskySolomon-2021}, an example is
given of two primitive substitution rules $\rho_1$ and $\rho_2$ on
the same set of tiles such that
$
M_{\rho_1}=M_{\rho_2},
$
but the Delone sets corresponding to $\rho_1$ are uniformly spread,
whereas those corresponding to $\rho_2$ are not.
\end{remark}

\section{Multiscale substitution tilings}\label{sec:multiscale}
As in the case of the well-studied primitive substitution tilings of \S\ref{sec:substitution} and cut-and-project sets of \S\ref{sec:C&P}, the framework of \emph{multiscale substitution tilings} offers another general method for constructing aperiodic Delone sets with strong regularity properties that depend only on finite initial data. An early example is given by Sadun's generalized pinwheel tiling studied in \cite{Sadun-1998}, which extends the construction of the pinwheel tiling \cite{Radin-1994} by allowing the triangle on which the substitution rule is defined to appear in two distinct scales. The general framework was introduced and explored in \cite{SmilanskySolomon-2021}, where abstract substitution rules with multiple distinct scales on a set of finite tiles in any dimension were considered. Under a simple irrationality assumption on the participating scales, a rich family of tilings with hierarchical and self-similar structure that is disjoint from the classical ``fixed-scale'' construction is defined. Many favorable properties, such as minimality and unique ergodicity, carry over to the multiple scale construction, while new and distinct geometric and combinatorial behaviors emerge. This flexibility makes multiscale substitution tilings a particularly interesting object of study in the context of equivalence relations on Delone sets.

\subsection{Multiscale substitution rules and tilings}
Once again, recall our definitions of tiles and tilings of $\R^d$ as given in Definition \ref{def:tiling}.

\begin{definition}\label{def:substitution_scheme}
A \emph{multiscale substitution rule} $\sigma$ is a map on a set $\AA=\{T_1,\ldots,T_n\}$ of unit-volume tiles in $\R^d$, so that $\sigma (T_i)$ is a patch of tiles with $\supp(T_i)=\supp(\sigma (T_i))$ for every $T_i\in \AA$. The patch $\sigma (T_i)$ contains at least two tiles, each a translation of a rescaled copy of a tile in $\AA$. We let $\omega_\sigma(T_i)$ denote the list of rescaled tiles whose translations appear in the patch $\sigma (T_i)$, presented as  
\begin{equation*}
\omega_\sigma(T_i)=\left(\alpha_{ij}^k T_j \mid \,j=1,\ldots,n,\,\,k=1,\ldots,k_{ij}\right).
\end{equation*} 
 
Here the number $k_{ij}$ of rescaled copies of $T_j$ in $\sigma (T_i)$ is assumed to be finite and $\alpha_{ij}^k\in(0,1)$ for all $i,j,k$.
\end{definition}

The main difference between Definition  \ref{def:substitution_scheme} and the setup of a substitution rule given in Definition \ref{def:SubRule} in \S~\ref{sec:substitution} is in the scales at which tiles appear in the patches $\sigma(T_i)$ and $\rho(T_i)$. While in $\rho(T_i)$ all tiles are rescaled copies of elements of $\AA$ with respect to a single scaling constant $\xi^{-1}$, in the multiscale setup, multiple scaling constants are allowed. 
We note that multiscale substitution rules in $\R$ were explored by Kamae in his work on numeration systems \cite{Kamae-2005}.

\begin{example}\label{ex:square_scheme}
    Figure \ref{fig:square_scheme} describes a multiscale substitution scheme in $\R^2$ with a single square tile, decomposed into 17 squares of two distinct scales. 
\end{example}

\begin{figure}[h!]
\centering
	\includegraphics[scale=0.5]{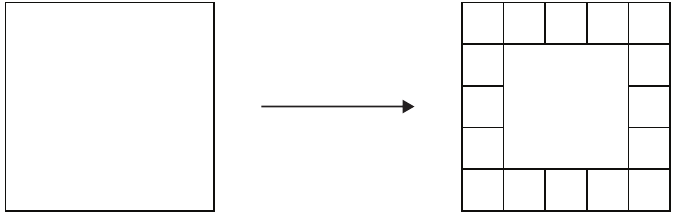}\caption{A multiscale substitution rule on a square in $\R^2$, with scaling constants $1/5$ and $3/5$.
	}	\label{fig:square_scheme}
\end{figure}

\begin{remark}
    The assumption that tiles in $\AA$ are of unit-volume is not crucial. However, since every substitution rule is essentially equivalent to one with unit-volume tiles (see \cite[\S3.1]{Smilansky-2020}), we can assume this throughout our discussion with no loss in generality, fixing the values of $\alpha_{ij}^k$ in a unique and consistent way.
\end{remark}

The major difference between Definition \ref{def:substitution_scheme} above and Definition \ref{def:SubRule} of the substitution rule considered in \S\ref{sec:substitution} is that tiles may appear in any number of distinct scales. 
When multiple scaling factors are allowed, the standard substitution procedure, in which all tiles are substituted simultaneously, would cause the ratios between tile volumes to diverge, eventually producing tiles of arbitrarily large or small scales. Tilings defined this way would therefore not be suitable for the study of Delone sets, and so a new procedure is required to generate appropriate tilings in the multiscale setup.

\begin{definition}\label{def: substitution flow}
	Let $\sigma$ be a multiscale substitution rule on $\AA$. For each $T_i\in \AA$, the \emph{substitution semi-flow} defines a family of patches in the following way. At $t=0$, set $F_0(T_i)=T_i$, which is a patch consisting of a single tile. As $t$ increases, inflate the patch by a factor of $e^{t}$ and substitute tiles of volume larger than $1$ according to the substitution rule  $\sigma$ (here we extend the substitution rule to act on rescaled copies of tiles in the natural way).  Equivalently, $F_t(T_i)$ is the patch supported on $e^tT_i$, which is the result of repeated substitution until all tiles in the patch are of unit volume or less. Tiles that emerge as rescaled copies of $T_i\in \AA$ are called \emph{tiles of type $i$}. Note that in contrast with \S\ref{sec:substitution}, tiles of the same type need not be translation equivalent, but are rescaled copies of each other.
	
	We assume that the tiles $T_i$ are positioned so that the origin of $\R^d$ lies in their interior. In this case, their inflations exhaust $\R^d$. Denote
	\begin{equation*}
	\PP_{i}=\left\{ F_t(T_{i}) \mid \,t\in\R_{\ge 0} \right\}.
	\end{equation*}
	Then the patches $\PP_\sigma=\bigcup_{i=1}^{n}\PP_{i}$ are called the \emph{generating patches}, and play the role of the legal patches described in Definition \ref{Def:subsTiling}. 
\end{definition}

\begin{figure}[h!]
    \centering
	\includegraphics[scale=0.115]{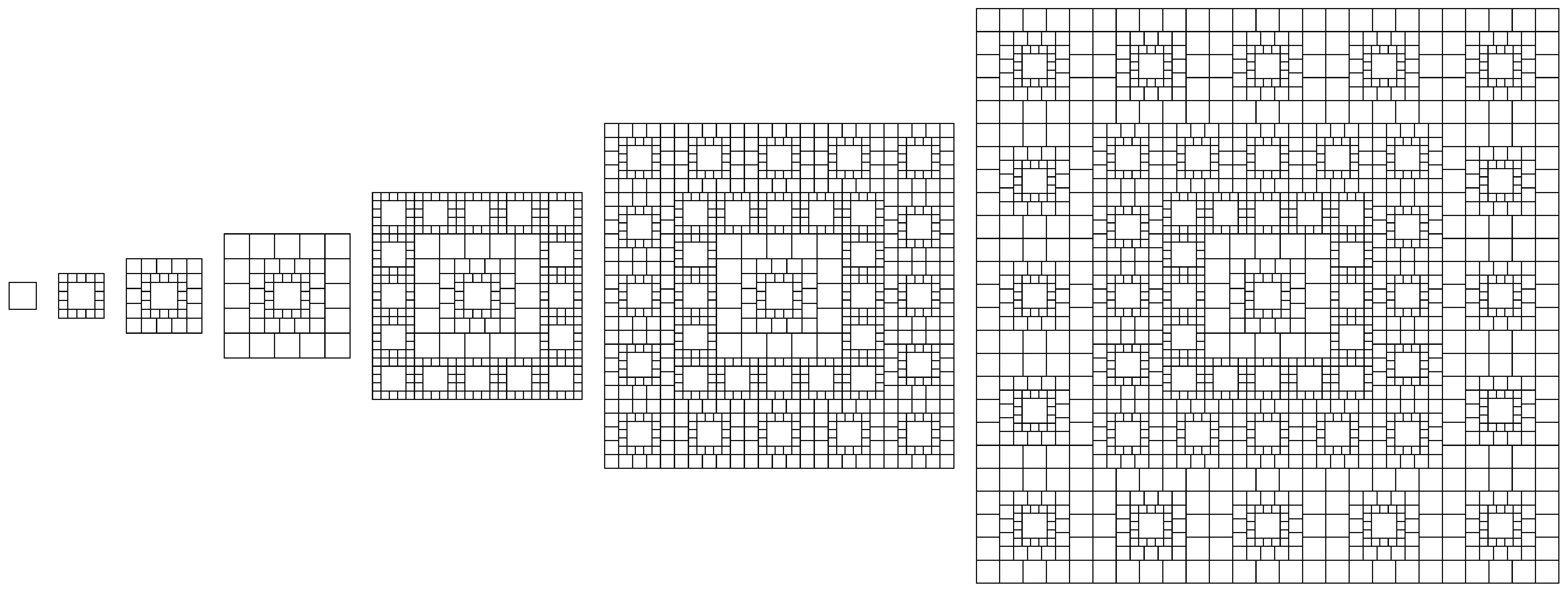}\caption{Generating patches for the square multiscale substitution rule of Example \ref{ex:square_scheme}.\label{fig:square_generating_patches}}	
\end{figure}

\begin{remark}
    Let $(t_m)_{m\in \N}$ be the strictly increasing sequence of times for which $F_{t_m}(T_i)$ contains a tile of unit volume, and consider the rescaled patches
    $$
    \pi_m:=\frac{1}{\vol(e^{t_m}T_i)}F_{t_m}(T_i).
    $$
    Then $(\pi_m)_{m\in \N}$ is the \emph{Kakutani sequence of partitions} of $T_i$ generated by $\sigma$. These sequences, produced by substituting all tiles of maximal volume at step $m$ to define the subsequent partition $\pi_{m+1}$, can be viewed as a multi-tile and higher-dimensional generalization of Kakutani's construction of sequences of partitions of the unit interval \cite{Kakutani-1976} and studied in \cite{Smilansky-2020}. 
\end{remark}

In analogy to Definition \ref{Def:subsTiling} for the substitution tiling framework, we define the \emph{tiling space} $X_\sigma$ as the set of all tilings of $\R^d$ obtained as limits of translations of generating patches $\mathcal{P}_\sigma$ under the metric given in \eqref{eq:metric}. 
For each $j \in \{1, \ldots, n\}$, let
\[
\alpha_j = \min\{\alpha_{ij}^k \mid \exists\, i \text{ such that } \alpha_{ij}^k T_j \in \omega_\sigma(T_i)\},
\]
that is, the minimal scale $\alpha$ for which a tile of the form $\alpha T_j$ appears in the subdivision $\sigma(T_i)$ of some $T_i\in \AA$. It follows directly from the construction that for every $j$, every scale of a tile of type $j$ in a tiling $\tau \in X_\sigma$ belongs to $[\alpha_j, 1]$. In particular, selecting a point in the interior of each tile in $\AA$ according to a consistent rule gives rise to a Delone set $\Lambda=\Lambda_\tau$.

\subsection{Incommensurability and the associated graph}
The following properties are fundamental in the study of multiscale substitution rules and the tilings they generate. 
\begin{definition}
A multiscale substitution rule $\sigma$ is called \emph{irreducible} if for every pair $i,j \in \{1,\ldots,n\}$, there exists $t > 0$ such that $F_t(T_i)$ contains a tile of type $j$. It is said to be \emph{incommensurable} if there exist tiles $T_i, T_j \in A$ and tiles of types $i$ and $j$ appearing in patches from $\PP_i$ and $\PP_j$, respectively, which have unit volume at times $t_1$ and $t_2$ with $t_1 \notin \Q t_2$. Otherwise, the scheme is called \emph{commensurable}.
\end{definition}

As we will see shortly in Exercise \ref{exe:incom_in_graphs}, both properties can be described in an equivalent and easy-to-verify way using the graphs associated with the substitution rule. 
Incommensurability forces a real distinction between the tilings in $X_\sigma$ and primitive substitution tilings as discussed in \S\ref{sec:substitution}. The following results were established in \cite{SmilanskySolomon-2021}. 

\begin{thm}\label{thm:multiscaledynamics}
    Let $\sigma$ be an irreducible incommensurable substitution rule on $\AA$ in $\R^d$. Then every tiling in $X_\sigma$ is non-FLC and consists of tiles of $|\AA|$ types that appear in a dense set of scales. The system $(X_\sigma, \R^d)$ is minimal and uniquely ergodic.
\end{thm}

A substitution matrix for a multiscale substitution rule $\sigma$ could be defined as in Definition \ref{Def:SubMatrix+Primitive}. However, it would not carry any information about the different scales appearing in $\omega_\sigma$. To overcome this, we define a directed weighted graph that serves to carry both the combinatorial information of the multiscale substitution rule and the participating scales. In addition, it is convenient to view the substitution semi-flow as being modeled by a semi-flow along the edges of an underlying graph.

Let $G = (V, E, \ell)$ denote a directed weighted multigraph, where each edge $e\in E$ is assigned a positive weight, interpreted as its length and denote by $\ell(e)$. 
By convention, an edge is assumed to include its terminal vertex but not its initial one, so that every vertex is regarded as a point contained in each of its incoming edges. We consider \emph{directed walks} in $G$, which need not begin or end at vertices in $V$. A walk that starts and ends in vertices in $V$ is called a \emph{path}. Every edge $e\in E$ is equipped with a linear parametrization by the interval $(0,\ell(e)]$, and the length $\ell(\gamma)$ of a walk in $G$ is defined accordingly. The path of length zero with initial vertex $i\in V$ is taken to consist solely of the vertex $i$.

We can now define the graph associated with a multiscale substitution rule. 

\begin{definition} \label{def:associated_graph}
Let $\sigma$ be a multiscale substitution rule on $\AA=\{T_1,\ldots,T_n\}$. The \emph{associated graph} $G_\sigma$ is a directed weighted graph with vertices $V=\left\{1,\ldots,n\right\}$, where the vertex $i\in V$ is associated with the tile $T_i\in \AA$. Edges with initial vertex $i\in V$ are defined according to the elements in $\omega_\sigma(T_i)$, where for each $\alpha T_j\in\omega_\sigma(T_i)$, an edge $ e\in E$ is associated, with terminal vertex $j$, and length 
	\begin{equation*}
	\ell( e)=\log\tfrac{1}{\alpha}.
	\end{equation*}	
\end{definition}

Note that $G_\sigma$ depends only on the elements in $\omega_\sigma(T_i)$, and not on the specific geometry in which they appear in the patches $\sigma(T_i)$. Thus, as already mentioned, $G_\sigma$ can indeed be thought of as the abelianization of $\sigma$, with a role similar to that of the substitution matrix $M_\rho$ in \S\ref{sec:substitution}. 

 The associated graph offers the simplest method for verifying incommensurability, which we leave as Exercise \ref{exe:incom_in_graphs}. For a detailed discussion of incommensurability and equivalent formulations, see \cite[\S 3.1]{SmilanskySolomon-2021}.

\begin{exe}\label{exe:incom_in_graphs}
    Show that in terms of the associated graph $G_\sigma$, irreducibility and incommensurability are equivalent to $G_\sigma$ being strongly connected and containing two closed paths with lengths $a$ and $b$ satisfying $a \notin \Q b$, respectively. 
\end{exe}

\begin{example} \label{ex:Kak_rule_graph}
    Let $\alpha\in(0,1)$. The $\alpha$-Kakutani multiscale substitution rule is defined on the unit interval $I=[0,1]$ in $\R$, which is substituted into two rescaled copies of the interval, one of length $\alpha$ and one of length $1-\alpha$. The rule and its associated graph are illustrated in Figure \ref{fig:Kakgraph}. Incommensurability holds for any $\alpha\in(0,1)$ for which $\frac{\log \alpha}{\log(1-\alpha)}\notin\Q$, implying it can be viewed as a typical property.
    
    \begin{figure}[h!]
    \centering
	\includegraphics[scale=0.38]{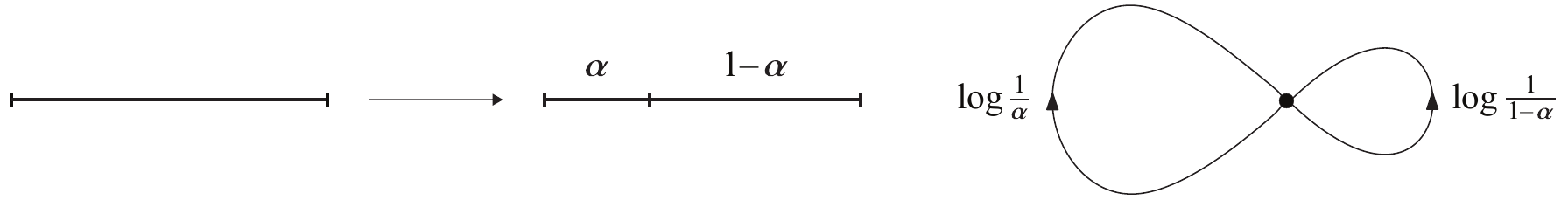}\caption{The $\alpha$-Kakutani multiscale substitution rule and its associated graph with a single vertex and two loops corresponding to the single tile $I$ and the two rescaled copies in its substitution rule. }	\label{fig:Kakgraph}	
    \end{figure}
    
\end{example}

\begin{exe}
    Sketch the graph associated with the substitution scheme described in Example \ref{ex:square_scheme} and verify that it is irreducible and incommensurable. Do the same for the multiscale substitution scheme illustrated in Figure \ref{fig:triangle_scheme} below, where, as in \S\ref{sec:substitution}, we allow rotations and consider $\AA$ to consist of two elements. 
    
\begin{figure}[h!]
\centering
	\includegraphics[scale=0.5]{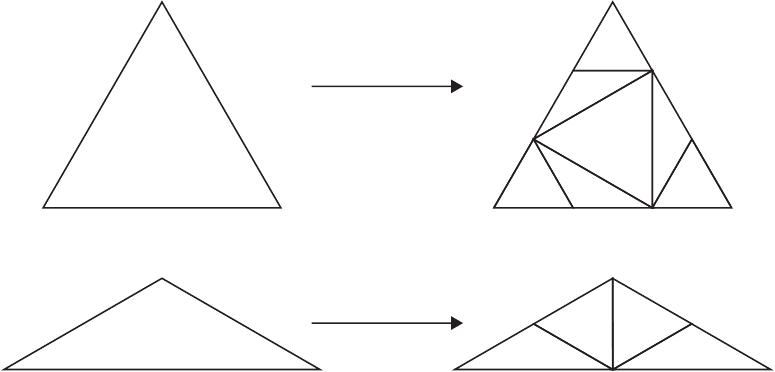}\caption{A multiscale substitution rule on two triangles in $\R^2$.
	}	\label{fig:triangle_scheme}
\end{figure}
\end{exe}

A very useful property of $G_\sigma$ is the correspondence between walks in $G_\sigma$ and tiles in the generating patches $\mathcal{P}_\sigma$, described in Lemma \ref{lem:scales_and_paths} below. For full details about this correspondence, see \cite[\S2]{SmilanskySolomon-2021}.

\begin{lem}\label{lem:scales_and_paths}
    Let $t>0$ and $T_i\in \AA$. Then every tile in $F_t(T_i)$ is in one-to-one correspondence with a walk in $G_\sigma$ of length $t$ that starts at the vertex $i\in V$. A tile $T\in F_t(T_i)$ of type $j$ corresponds to a walk $\gamma_T$ with termination point contained in an edge that ends at the vertex $j\in V$. The volume of $T$ is then determined only by the distance between the termination point of $\gamma_T$ and the vertex $j$.
\end{lem}

\begin{example}
     Figure \ref{fig:Kakwalks} is an illustration of three patches in $F_t(I)$, generated by the $\alpha$-Kakutani multiscale substitution rule described in Example \ref{ex:Kak_rule_graph}, and their corresponding walks on the graph. 

\begin{figure}[h!]
\centering
	\includegraphics[scale=0.38]{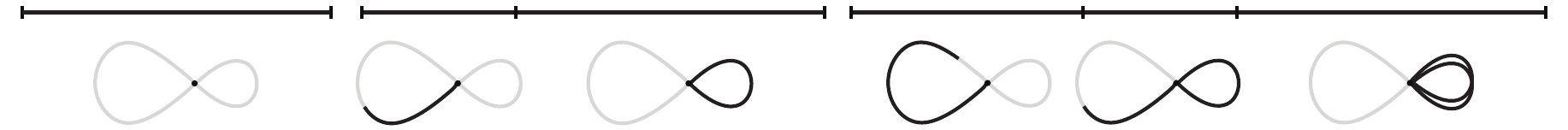}\caption{Generating patches and their corresponding walks on the associated graph for the $\alpha$-Kakutani multiscale substitution rule from Example \ref{ex:Kak_rule_graph}.}	\label{fig:Kakwalks}	
\end{figure}
\end{example}

Examples of incommensurable multiscale substitution rules with multiple tile-types can easily be constructed in every dimension $d\ge 1$, see \cite{Smilansky-2020}, \cite{SmilanskySolomon-2021}, \cite{SmilanskySolomon2-2022}.
Commensurable substitution rules are, of course, interesting in their own right. They include all ``fixed-scale'' substitution rules, namely those discussed in \S \ref{sec:substitution}. The Rauzy fractal rule introduced in \cite{Rauzy-1982} provides an example of a commensurable multiscale substitution rule on a single tile with multiple distinct scales. 
For further discussion and additional examples of commensurable multiscale substitution rules, see \cite{Smilansky-2020}.

\subsection{Tile-counting estimates and implications for BL and BD equivalence}
As we have seen in previous sections, a key step in the study of BD and BL equivalence relations on Delone sets is to establish estimates of the density and the discrepancy. In Delone sets associated with tilings, this means estimates on the main term and bounds on the error term of tile-counting functions. For primitive substitution tilings, it was shown in \S\ref{sec:substitution} that such results can be achieved using elementary linear algebra. This is no longer the case for incommensurable multiscale substitution tilings, which require estimates concerning the distribution of walks in the associated graph.

Given a substitution rule $\sigma$ on a set $\AA=\{T_1,\ldots,T_n\}$ of unit-volume tiles in $\R^d$, let $S_\sigma\in M_n(\Z)$ and $V_\sigma, H_\sigma\in M_n(\R)$ denote the matrices defined by

		\begin{equation*}
		\left(S_\sigma\right)_{ij}=\sum\limits_{\substack{T\in\omega_\sigma(T_j) \\ T\text { of type }i}}1 ~;~ \,
			\left(V_\sigma\right)_{ij}=\sum\limits_{\substack{T\in\omega_\sigma(T_j) \\ T\text { of type }i}}\vol T~; ~\,
				\left(H_\sigma\right)_{ij}=\sum\limits_{\substack{T\in\omega_\sigma(T_j) \\ T\text { of type }i}}-\vol T\cdot\log(\vol T).
		\end{equation*} 
        The matrices $S_\sigma, V_\sigma$, and $H_\sigma$ hold the combinatorial, volume, and entropy information associated with $\sigma$, respectively. Since tiles in $\AA$ are of unit volume, the vector $u_{\vol} = \mathbf{1}\in \R^n $ is a left Perron--Frobenius eigenvector of $V_\sigma$, with eigenvalue $\mu=1$.

\begin{thm}\label{thm:multiscale_tile_counting}
	Let $\sigma$
	be an irreducible incommensurable substitution rule. Then
	\[
	\frac{\#\{\text{Tiles in }F_t(T)\}}{\vol(F_{t}(T))}=\frac{{\bf 1}^t(S_\sigma - V_\sigma )v_1}{{\bf 1}^tH_\sigma v_1}+o\left(1\right),\quad t\rightarrow\infty,
	\]
    independently of the choice of $T\in A$, where $v_1\in\R^n$ is any right Perron--Frobenius eigenvector of $V_\sigma$.
    \end{thm}
It follows from the arguments in \cite[\S9]{SmilanskySolomon-2021} that the asymptotic density $\alpha$ of any Delone set associated with a tiling in $X_\sigma$ is equal to the limit in Theorem \ref{thm:multiscale_tile_counting}. 
Note that here the Perron-Frobenius vectors are taken with respect to the volume matrix $V_\sigma$, compare this with \eqref{eq:alpha}, the asymptotic density for primitive substitution tilings.
Theorem \ref{thm:multiscale_tile_counting} is a special case of \cite[Theorem~2.3]{Smilansky-2022}, which establishes formulas for the asymptotic frequencies of tiles of given types and within a given interval of scales in generating patches. The proof involves considerably heavier analytical machinery than that provided by linear algebra in \S\ref{sec:substitution}, and relies on path-counting results established in \cite{KiroSmilanskySmilansky-2020}, in which the Wiener-Ikehara Tauberian theorem plays a central role. We omit the proof here and refer the reader to \cite{Smilansky-2022} and to \cite[\S 7]{SmilanskySolomon-2021}.

\begin{exe}
    Compute the asymptotic density of a multiscale substitution tiling generated by an incommensurable $\alpha$-Kakutani rule, as described in Example \ref{ex:Kak_rule_graph}.
\end{exe}

While Theorem \ref{thm:multiscale_tile_counting} clarifies the existence and value of the density $\alpha$, as the next result shows, the discrepancy in the incommensurable setting may remain large. We say that two edges of $G_\sigma$ are \emph{parallel edges} if they have the same length and the same initial and terminal vertices. Denote by $C = C_\sigma$ the size of a maximal set of non-parallel edges of $G_\sigma$.

\begin{thm}\label{thm:multiscale_large_discrepancy}
    Let $\sigma$ be an irreducible incommensurable substitution rule,
    and let $\alpha$ be the asymptotic tile density from Theorem
    \ref{thm:multiscale_tile_counting}. Then there exist
    $T\in\AA$ and arbitrarily large $t>0$ such that
    \begin{equation}\label{eq:multiscale_lower_bound}
    \absolute{
        \#F_t(T)-\alpha\vol(F_t(T))
    }
    =
    \Omega\left(
        \frac{\vol(F_t(T))}
             {\log^{C-1}(\vol(F_t(T)))}
    \right).
    \end{equation}
\end{thm}

The theorem follows from the combinatorial argument below. We note
that this is a slight improvement on the bounds appearing in
\cite[\S8]{SmilanskySolomon-2021}, although the argument is similar.

\begin{lem}\label{lem:poly(log)_bound}
    The maximal size of a set of tiles in $F_t(T_i)$ consisting of
    tiles of pairwise distinct volumes is
    $\OO\left(t^{C-1}\right)$,
    independently of $T_i\in\AA$.
\end{lem}

\begin{proof}
    Let
    $
    \mathcal{E}_1,\ldots,\mathcal{E}_C
    $
    denote the equivalence classes of parallel edges of $G_\sigma$,
    and let $\beta_\ell>0$ be the common length of the edges in
    $\mathcal E_\ell$. Set
    \[
    \beta_{\max}=\max_{1\le\ell\le C}\beta_\ell.
    \]

    Let $\gamma$ be a walk of length $t$ corresponding, by Lemma
    \ref{lem:scales_and_paths}, to a tile in $F_t(T_i)$.
    For every $\ell\in\{1,\ldots,C\}$, let $x_\ell(\gamma)$ be the
    number of edges from $\mathcal E_\ell$ that are fully traversed
    by $\gamma$.

    If $\gamma$ terminates in the interior of an edge, then
    \begin{equation}\label{eq:number_of_walks}
    t-\beta_{\max}
    <
    \sum_{\ell=1}^C\beta_\ell x_\ell(\gamma)
    <
    t.
    \end{equation}
    Note that walks terminating at a vertex correspond to tiles of unit volume and therefore contribute at most one additional volume.

    We claim that the number of vectors
    $
    x=(x_1,\ldots,x_C)\in\Z_{\ge0}^C
    $
    satisfying \eqref{eq:number_of_walks} is $\OO(t^{C-1})$. Indeed, each of
    $x_1,\ldots,x_{C-1}$ has $\OO(t)$ possible values. Once these
    coordinates are fixed, $x_C$ must lie in an interval of length
    at most
    \[
    \frac{\beta_{\max}}{\beta_C},
    \]
    and hence there are only $\OO(1)$ possible values of $x_C$, which proves the claim.

    Finally, once the vector $x$ and the equivalence class of the
    terminal edge are fixed, the distance of the terminal point of
    the walk from the terminal vertex is determined. By Lemma
    \ref{lem:scales_and_paths}, this determines the type and
    volume of the corresponding tile. Since there are only $C$
    possible terminal-edge classes, the total number of distinct
    tile volumes is $\OO(t^{C-1})$.
\end{proof}

As a consequence, the pigeonhole principle gives the following quantitative jump estimate.

\begin{cor}\label{cor:counting_gaps}
    For every $T_i\in\AA$ and every $t_0\ge1$, there exist
    $t\ge t_0$ and $\varepsilon_0>0$ such that, for every
    $\varepsilon\in(0,\varepsilon_0]$,
    \[
    \#F_{t+\varepsilon}(T_i)-\#F_t(T_i)
    \ge
    c_\sigma\frac{e^{dt}}{t^{C-1}},
    \]
    where $c_\sigma>0$ depends only on the parameters of $\sigma$.
\end{cor}

\begin{exe}
    Deduce Theorem \ref{thm:multiscale_large_discrepancy} and Corollary \ref{cor:counting_gaps} from Lemma \ref{lem:poly(log)_bound}.
\end{exe}


To deduce non-uniform spreadness from Theorem
\ref{thm:multiscale_large_discrepancy}, we require quantitative
control on the boundary neighborhoods of the prototiles.

\begin{cor}\label{cor:multiscale_not_uniformly_spread}
    Let $\sigma$ be an irreducible incommensurable multiscale
    substitution rule in $\R^d$. Suppose that there exists
    $\eta>0$ such that, for every $T_i\in\AA$,
    \begin{equation}\label{eq:multiscale_boundary_condition}
    \vol\left(
        \left(\partial\supp(F_t(T_i))\right)^{+1}
    \right)
    =
    \OO\left(e^{(d-\eta)t}\right).
    \end{equation}
    Then every Delone set associated with a tiling in $X_\sigma$ is
    not uniformly spread.
\end{cor}

\begin{proof}
    Suppose, for contradiction, that one of the associated Delone
    sets is uniformly spread. Since $(X_\sigma,\R^d)$ is minimal by
    Theorem \ref{thm:multiscaledynamics}, and uniform
    spreadness is inherited by every element of the hull, all the
    associated Delone sets are uniformly spread. Moreover, the usual
    proof of this inheritance gives a BD bound that is uniform over
    the hull. Therefore, by Laczkovich's criterion, there exists a
    constant $c>0$ such that
    \begin{equation}\label{eq:uniform_Lacz_multiscale}
    \absolute{\#(\Lambda\cap U)-\alpha\vol(U)}
    \le
    c\vol\left((\partial U)^{+1}\right)
    \end{equation}
    for every associated Delone set $\Lambda$ and every bounded
    measurable set $U$.

    Every generating patch extends to a tiling in $X_\sigma$.
    We may therefore apply
    \eqref{eq:uniform_Lacz_multiscale} to the generating
    patches supplied by Theorem
    \ref{thm:multiscale_large_discrepancy}. For these patches,
    the discrepancy is bounded from below by
    \[
    c_1\frac{e^{dt}}{t^{C-1}},
    \]
    whereas \eqref{eq:multiscale_boundary_condition} gives
    \[
    \vol\left(
        \left(\partial\supp(F_t(T_i))\right)^{+1}
    \right)
    =
    \OO\left(e^{(d-\eta)t}\right).
    \]
    Since
    \[
    \frac{e^{dt}/t^{C-1}}{e^{(d-\eta)t}}
    =
    \frac{e^{\eta t}}{t^{C-1}}
    \longrightarrow\infty,
    \]
    inequality \eqref{eq:uniform_Lacz_multiscale} is
    contradicted.
\end{proof}

Condition \eqref{eq:multiscale_boundary_condition} holds, for
example, for polygonal prototiles, and more generally for prototiles
with Lipschitz boundary. In dimension one it holds automatically for
interval prototiles. In particular, every Delone set arising from an
incommensurable $\alpha$-Kakutani multiscale substitution tiling of
$\R$ is not uniformly spread.

\begin{cor}
    Let $\Lambda$ be a Delone set that arises from an incommensurable multiscale substitution tiling, whose prototiles
have Lipschitz boundary. Then $\Lambda$ is not uniformly spread. 
\end{cor}

In particular, any Delone set that arises from an incommensurable $\alpha$-Kakutani multiscale substitution tiling of $\R$ is not uniformly spread. Moreover, in the commensurable case, there are precisely five distinct values of $\min\{\alpha,1-\alpha\}$
that give rise to uniformly spread Delone sets. This exceptional set of values is related to a certain family of Pisot-Vijayaraghavan numbers, see \cite{Smilansky-2025}.

With regard to BL equivalence, obtaining a bound of the form \eqref{eq:multiscale_lower_bound} with $C-1 \le 1$ implies the failure of the Burago--Kleiner sufficient condition for rectifiability, stated in Theorem \ref{thm:BK_sufficient_condition}. Indeed, the substitution rule in Example \ref{ex:square_scheme} is of this kind, since the two participating scales correspond to two lengths of edges in the associated graph, and thus $C=2$ (see also \cite[Proposition~5.4]{SmilanskySolomon-2021} for a more careful analysis of the square substitution rule). 
This makes Delone sets arising from incommensurable multiscale substitution schemes particularly interesting candidates for potential non-rectifiable examples, although the following question remains open:

\begin{open}
    Does there exist a non-rectifiable Delone set arising from a multiscale substitution scheme?
\end{open}

\section{Cut-and-project sets}\label{sec:C&P} 

A central class of aperiodic Delone sets in $\R^d$ is constructed via the projection method. 
These cut-and-project sets, or model sets, originated in the work of Meyer~\cite{Meyer-1970} and have since received considerable attention, both as mathematical models for quasicrystals and as fundamental examples in the study of aperiodic order. 
They provide a rich and flexible framework that naturally yields Delone sets with long-range order properties such as FLC and repetitivity, but without translational periodicity.

Fix $d,m,n\in\N$ with $n=d+m$, and a direct sum decomposition $\R^n=\R^d\oplus \R^m$. We refer to $\R^d$ as the \emph{physical space} and to $\R^m$ as the \emph{internal space}, and denote the corresponding orthogonal projections from the \emph{ambient space} $\R^n$ by $\pi_{\rm{phys}}$ and $\pi_{\rm{int}}$, respectively. Given a lattice $\LL\subset \R^n$ and a \emph{window} $\WW\subset \R^m$, we define the associated \emph{cut-and-project set} $\Lambda=\Lambda(\LL,\WW)\subset \R^d$ by
\begin{equation*}
\Lambda(\LL,\WW)=\pi_{\rm{phys}}(\LL\cap \pi_{\rm{int}}^{-1}(\WW)),
\end{equation*}
namely the physical projection of the lattice points with internal projection in $\WW$. We note that we also allow $\LL$ to be a translated lattice, that is, a set of the form $x_0+\LL_0$ for some lattice $\LL_0$. 

The literature contains several regularity assumptions on cut-and-project schemes. 
Typically, one assumes that $\pi_{\mathrm{int}}(\LL)$ is dense in $\R^m$ and that the restriction $\pi_{\mathrm{phys}}|_{\LL}$ is injective, ensuring that $\Lambda$ cannot be realized within a similar setup involving smaller dimensions $n$ and $m$. 
In addition, the window $\WW$ is often assumed to be Borel measurable, bounded, with nonempty interior and boundary of measure zero. Under such standard assumptions, the resulting cut-and-project set $\Lambda(\LL,\WW)$ is an aperiodic Delone set in the physical space $\R^d$ with asymptotic density $\vol(\WW)/\covol(\LL)$, see \cite{Moody-2002}.

In contrast with the geometric definition above, cut-and-project sets can also be defined using a dynamical setup, as the set of return times to a section that is transverse to an $\R^d$-flow on the $n$-dimensional torus. More precisely, let $\mathbb{T}^n = \R^n/\Z^n$ denote the $n$-dimensional torus, and let $\pi:\R^n\to\mathbb{T}^n$ be the canonical projection. For $d<n$, a $d$-dimensional subspace $V\cong \R^d$ of $\R^n$ acts naturally on $\mathbb{T}^n$ by 
\[
\forall v\in V, x\in\R^n:\quad v.\pi(x) = \pi(x+v).
\]
We refer to this action as the $V$-flow. Let $\SS \subset \mathbb{T}^n$ be a Poincar\'{e} section of complementary dimension, that is, an $(n-d)$-dimensional submanifold everywhere transverse to the $V$-flow, and let $x_0 \in \mathbb{T}^n$. We then define the \emph{toral dynamics Delone set} by

\[
\Lambda(V, \SS,x_0) = \{v\in V ~\mid~ v.x_0\in \SS\}.
\]
As in the geometric setting, certain regularity assumptions are imposed on $V$ and $\SS$, which depend on the context, see, for instance, \cite{AdiceamSolomonWeiss-2022}, \cite{HaynesKellyWeiss-2014} and \cite{HaynesKoivusalo-2016}. 
Under the assumption that $\SS$ is a \emph{linear section}, i.e., $\SS = \pi(K)$ for some bounded set $K \subset U$ with nonempty interior, where $U$ is an $(n-d)$-dimensional affine subspace transverse to $V$, the definition of $\Lambda(V, \SS, x_0)$ coincides with that of the cut-and-project set $\Lambda(\LL, \WW)$. In the proof of this equivalence, $V$ and $U$ play the roles of the physical and internal spaces, respectively, with $\WW = -K$ and $\LL = x_0 + \Z^n$. The reader is encouraged to verify the equivalence between the two definitions, which can also be found in \cite[Proposition~2.3]{AdiceamSolomonWeiss-2022}.

In the context of BL and BD equivalence relations, when it comes to discrepancy estimates, each of these two formulations offers distinct advantages and relies on different mathematical tools. 
The dynamical viewpoint recasts the counting problem as that of studying return times of a linear flow on a torus, and thus draws on methods from ergodic theory and homogeneous dynamics. 
Quantitative equidistribution results for such flows, together with Diophantine properties of the subspace $V$, provide effective control on the uniformity of returns, and hence on discrepancy bounds. 
In contrast, the geometric formulation interprets the construction as a lattice-point counting problem in acceptance domains, which are certain regions in the window $\WW$. 
This viewpoint naturally invokes tools from Fourier analysis and the geometry of numbers, such as Poisson summation formulas, estimates on lattice-point counting errors in expanding regions, and Diophantine properties of the lattice $\LL$. 
Since these approaches involve advanced analytical and dynamical methods, we shall not include the proofs here, but rather state some central results concerning BL and BD equivalence in cut-and-project sets in a concise, survey-oriented manner.

The study of the BL equivalence relation for cut-and-project sets was initiated by Burago and Kleiner, who considered the case $(n,d) = (3,2)$ in their paper \cite[\S 1]{BuragoKleiner-2002}, where the sufficient condition from Theorem \ref{thm:BK_sufficient_condition} was established. They asked whether, for particular constructions of the form $\Lambda(V,\SS,x_0)$, a suitable Diophantine condition on the slope of $V$ suffices to imply rectifiability. This question was later investigated in much greater generality by Haynes, Kelly and Weiss in \cite{HaynesKellyWeiss-2014}, who considered arbitrary dimensions $(n,d)$ and obtained results concerning both BL and BD equivalence of toral dynamics Delone sets. Their analysis assumes that $\SS$ is a linear section satisfying certain regularity conditions on $\partial \SS$, expressed in terms of its Minkowski dimension $\dim_M(\partial \SS)$ (see, e.g., \cite{Mattila-1995}).

\begin{thm}\label{thm:HKW_BL}
    For almost every $d$-dimensional subspace $V\subset\R^n$, every $x_0\in\mathbb{T}^n$ and every linear section $\SS\subset\mathbb{T}^n$ satisfying $\dim_M(\partial \SS)<n-d$, the Delone set $\Lambda(V,\SS,x_0)$ is rectifiable. 
\end{thm}

We remark that the ``almost every'' statement in the theorem relates to an explicit Diophantine condition regarding the slope of the subspace $V$. 
Sharp quantitative bounds on the discrepancy of cut-and-project sets were recently established in \cite{KoivusaloLagace-2025}, clarifying the dependence on the regularity of the window $\WW$ and on explicit Diophantine conditions for the underlying lattice $\LL$. Even so, the following problem remains open: 

\begin{open}
    Does a non-rectifiable cut-and-project set exist?
\end{open}

As mentioned above, \cite{HaynesKellyWeiss-2014} also contained a study of BD equivalence for cut-and-project sets. Unlike the case of substitution tilings, here no dichotomy is known that determines whether $\Lambda$ is BD equivalent to a lattice based on the parameters of the construction. Nonetheless, the following results constitute the content of \cite[Theorem~1.2]{HaynesKellyWeiss-2014} regarding toral dynamics Delone sets with associated dimensions $(n,d)$.

\begin{thm}\label{thm:HKW_BD1} 
    Let $\frac{n+1}2<d<n$. Then for almost every $V\subset\R^n$, every $x_0\in\mathbb{T}^n$ and every $(n-d)$-dimensionally open linear section $\SS$ that satisfies $\dim_M(\SS)=n-d-1$, the set $\Lambda(V,\SS,x_0)$ is BD equivalent to a lattice. 
\end{thm}

Under stronger assumptions on the geometry of $\SS$, they also obtained the following: 

\begin{thm}\label{thm:HKW_BD2} 
Let $2\le d<n$. 
    \begin{enumerate}
        \item 
        For almost every $V$, every $x_0\in\mathbb{T}^n$ and every linear section $\SS$ that is an axis-parallel box, the set $\Lambda(V,\SS,x_0)$ is BD equivalent to a lattice. 
        \item\label{item_thm2:HKW_BD2}
        For almost every parallelotope linear section $\SS$, there is a residual set of subspaces $V$ for which $\Lambda(V,\SS,x_0)$ is not BD equivalent to a lattice for every $x_0$.
    \end{enumerate}
\end{thm}

In connection with the discussion in \S\ref{subsec:dichotomy}, Theorem \ref{thm:BD_dichotomy} implies that the hull of any $\Lambda(V,\SS,x_0)$ from \eqref{item_thm2:HKW_BD2} of Theorem \ref{thm:HKW_BD2} contains representatives of continuously many distinct BD classes. Earlier work by Frettl\"oh and Garber explicitly showed that BD equivalence need not be preserved within a hull of cut-and-project sets. They presented a concrete one-dimensional example of the phenomenon of two BD non-equivalent cut-and-project sets from the same hull (see \cite[Theorem~6.4]{FrettlohGarber-2018}). In their construction, they introduced the \emph{half-Fibonacci} sequences $F_1$ and $F_2$, obtained by splitting the standard Fibonacci window into two halves, and proved that $F_1$ and $F_2$ lie in the same hull but are not BD equivalent. A posteriori, as a consequence of Theorem \ref{thm:BD_dichotomy}, this example also yields a cut-and-project set whose hull contains representatives of continuously many distinct BD classes.

Another fruitful link in the study of BD equivalence for cut-and-project sets is the notion of \emph{bounded remainder sets (BRS)} for toral dynamics Delone sets, developed in \cite{HaynesKellyKoivusalo-2017}, \cite{HaynesKoivusalo-2016}. Roughly speaking, the BRS property concerns the boundedness of a similar type of discrepancy, now associated with the action of a toral translation, hence its close connection to BD equivalence. Haynes and Koivusalo \cite{HaynesKoivusalo-2016} provided an explicit geometric construction of an infinite family of windows whose associated cut-and-project sets satisfy the BRS property, and are therefore BD equivalent to a lattice.
In a related work \cite{EtkindGrepstadKolountzakisLev-2026}, the connection between BD equivalence of cut-and-project sets and equidecomposability of the corresponding windows is explored. 
Further study of BD equivalence of $1$-dimensional cut-and-project sets can be found in \cite{AmbrozMasakovaPelantova-2021}.

\section{Beyond BL and BD equivalence}\label{sec:beyond BL and BD}

In this section, we give a brief overview of some recent results that extend the study of BL and BD equivalence. We do not include proofs here and only provide the general framework and references for further reading.

\subsection{Generalized regularity and displacement equivalence}

The notions of BD and BL equivalence for Delone sets are defined in terms of bounded perturbations and biLipschitz distortions of point sets. By manipulating and relaxing these definitions, Dymond and Kalu{\v z}a \cite{DymondKaluza-2023}, \cite{DymondKaluza-2024} introduced and explored additional equivalence relations that allow a finer distinction between Delone sets.

Let $\omega:(0,a)\rightarrow(0,\infty)$ be a strictly increasing, concave function with $\lim_{t\rightarrow0}\omega(t)=0$, and let $\Lambda_1,\Lambda_2\subset\R^d$ be Delone sets. The \emph{modulus of regularity} $\omega$ was used in \cite{DymondKaluza-2023} to extend the notion of a biLipschitz map and BL equivalence in the following way.
\begin{definition}
A map $\psi:\Lambda_1\rightarrow\Lambda_2$ is a \emph{homogeneous $\omega$-mapping} if there exist constants $K>0$ and $c\in(0,1)$ such that 
\[
\|\psi(x_2)-\psi(x_1)\|\leq K\cdot R\cdot \omega\left(\frac{\|x_1-x_2\|}{R}\right),
\]
for all $R>0$ and $x_1,x_2\in B(0,R)$ with $\|x_2-x_1\|<cR$. We say that $\Lambda_1$ is \emph{$\omega$-regular} with respect to $\Lambda_2$ if there exists a bijection $\psi:\Lambda_1\rightarrow\Lambda_2$ so that both $\psi$ and $\psi^{-1}$ are homogeneous $\omega$-mappings.
\end{definition}

Note that for $\omega(t)=t$, this definition reduces to BL equivalence. McMullen has shown in \cite{McMullen-1998} that given any pair of Delone sets in $\R^d$ and the weaker H{\" o}lder moduli of continuity $\omega(t)=t^\beta$, each set is $\omega$-regular with respect to the other for some choice of $\beta\in(0,1)$. Intermediate forms of regularity, which correspond to moduli of regularity situated between the Lipschitz and the H{\" o}lder moduli, were considered by Dymond and Kalu{\v z}a, who established the following result in \cite[Theorem~1.2]{DymondKaluza-2023}.

\begin{thm} \label{thm:DK intermediate}
    There exist $\alpha>0$ and a Delone set in $\R^d$ that is $\omega$-irregular with respect to $\Z^d$ for any modulus of continuity $\omega$ of the form $\omega(t)=t\left(\log\frac{1}{t}\right)^\alpha$ for $t>0$ sufficiently small.
\end{thm}

It is important to note that the set of Delone sets that are $\omega$-regular with respect to $\Z^d$ strictly contains the set of rectifiable Delone sets in $\R^d$ \cite[Theorem~7.1]{DymondKaluza-2024}. This implies that Theorem \ref{thm:DK intermediate} does indeed strengthen McMullen's and Burago and Kleiner's result about the existence of non-rectifiable sets in $\R^d$ for $d\geq 2$, as described in \S \ref{subsec:non-rectifiable sets}. Only three such regularity relations are currently known to be pairwise distinct, namely H{\" o}lder regularity, the $\omega$-regularity of Theorem \ref{thm:DK intermediate}, and BL equivalence, though it is conjectured in \cite{DymondKaluza-2024} that the aforementioned intermediate moduli of continuity provide a fine hierarchy of distinct equivalence relations. Intermediate regularities and their relation to various rates of repetitivity were recently studied in \cite{Inoquio-RenteriaViera-2024}, further extending the discussion in \cite{Aliste-PrietoCoronelGambaudo-2013}.

Next, let $\varphi:(0,\infty)\rightarrow (0,\infty)$ be an increasing, concave function and let $\Lambda_1,\Lambda_2\subset\R^d$ be Delone sets. The asymptotic growth rate of $\varphi$ at $\infty$ was used in \cite{DymondKaluza-2024} to extend the notion of a bounded displacement map and BD equivalence in the following way.
\begin{definition}
Let $\Lambda\subset \R^d$ be a Delone set. For $R$ greater than the relative denseness constant of $\Lambda$, the \emph{$R$-displacement} of a function  $f:\Lambda\rightarrow \R^d$ is given by
\[
{\rm disp}(f,R):=
\sup\left\{\norm{f(x)-x} ~\mid~  x\in \Lambda\cap \overline{B_R(0)} \right\}.
\] 
We say that $\Lambda_1$ and $\Lambda_2$ are \emph{$\varphi$-displacement equivalent} if there exists a bijection $\phi:\Lambda_1\rightarrow\Lambda_2$ such that ${\rm disp}(\phi,R)=\OO(\varphi(R))$. 
\end{definition}

Clearly, if $\varphi$ is bounded, then $\varphi$-displacement equivalence reduces to BD equivalence. In contrast, any two Delone sets admit a bijection with linear $R$-displacement, and so $\varphi$-displacement equivalence does not distinguish between Delone sets for $\varphi(R)=\Omega(R)$. However, in the unbounded case and when $\varphi(R)=o(R)$, this definition distinguishes between Delone sets according to the rates at which their displacement functions diverge in a very strong way, giving rise to a continuous spectrum of equivalence relations on Delone sets. The following result, proved in \cite[Corollary~1.7]{DymondKaluza-2024}, shows that the family of $\varphi$-displacement equivalence relations forms a strict hierarchy.

\begin{thm}
    Let $\varphi_1,\varphi_2:(0,\infty)\rightarrow (0,\infty)$ be increasing, concave functions such that $\varphi_1(R)=o(\varphi_2(R))$. Then $\varphi_2$-displacement equivalence is a strictly weaker notion than $\varphi_1$-displacement equivalence.
\end{thm}

While BL equivalence does not precisely belong to the $\varphi$-displacement equivalence spectrum, it is positioned between the bounded and the unbounded equivalence relations. The following theorem contains the content of \cite[Theorems~1.9, 5.1, 6.1]{DymondKaluza-2024}.

\begin{thm}
    Let $d\geq 2$. Then $\varphi$-displacement equivalence of Delone sets in $\R^d$ is stronger than BL equivalence if and only if $\varphi$ is bounded, that is, when $\varphi$-displacement equivalence is BD equivalence, in which case every BL equivalence class contains representatives from continuously many BD classes. If $\varphi$ is unbounded, then every $\varphi$-displacement equivalence class contains representatives from continuously many BL classes.
\end{thm}

We conclude this section with an open problem, see \cite[Conjecture~1.8]{DymondKaluza-2024}, about the way in which the notions of $\omega$-regularity and $\varphi$-displacement equivalence are related.

\begin{open}
    Let $d\geq 2$. Is it true that there exists a Delone set in $\R^d$ that is $\varphi$-displacement equivalent to $\Z^d$ but not $\omega$-regular with respect to $\Z^d$ if and only if $R\cdot\omega(1/R)=o(\varphi(R))$?
\end{open}

\subsection{A short remark about random point processes}

Delone sets can be constructed randomly in various ways. This can be done by adding randomness to existing constructions, for example, by randomly choosing a lattice in the cut-and-project construction or by applying random substitution rules on a set of prototiles. Another method is to randomly perturb a Delone set. For example, a \emph{random perturbation of $\Z^d$} is a random point process of the form $\Pi=\{x+v_x\mid x\in\Z^d\}$, where the random translation vectors $(v_x)_{x\in\Z^d}$ are usually assumed to be
independent and identically distributed. When chosen uniformly inside a ball of radius $1/2$, or in a way similar to part (2) of Exercise \ref{ex:repetitive_Delone_sets}, for example, they give rise to point sets that are almost surely both uniformly discrete and relatively dense. Although arbitrary bounded perturbations need not preserve
uniform discreteness, and standard i.i.d.\ Gaussian perturbations are
almost surely not Delone, they are nevertheless locally finite random
point processes and admit meaningful matching problems with the
original lattice. A natural question, which can be viewed as related to the various equivalence relations described above, is to find matchings with optimal displacement, see for example \cite{ElboimSpinkaYakir-2026}, \cite{Yakir-2022} and references therein.

A random point process is \emph{hyperuniform} if the variance of the random variable of counting points inside a large ball grows strictly slower than the volume of the ball. Hyperuniformity can be viewed as a measure of long-range order in random point processes. We refer to Torquato's survey \cite{Torquato-2018} for background, applications, and further references. The definition is comparable to the Laczkovich and the Burago--Kleiner sufficient conditions for BD and BL equivalence in that it similarly relies on the growth rate of the discrepancy. For recent results about hyperuniformity in cut-and-project random point processes and further references on hyperuniformity in aperiodic order, we refer the reader to the recent work of Bj\"{o}rklund and Hartnick \cite{BjorklundHartnick-2024}.

\subsection*{Acknowledgements}
We are deeply grateful to Val\'{e}rie Berth\'{e}, Michel Rigo, and Manon Stipulanti for their thoughtful organization and warm hospitality at the ``Combinatorics, Automata, and Number Theory (CANT)'' school at CIRM, where the lectures on which this contribution is based were given. Their efforts made the school both enjoyable and fruitful, and we are thankful for the invitation to expand our notes into a chapter for the accompanying volume.
We would like to thank the school's participants for their insightful questions and stimulating discussions, which helped clarify and deepen several aspects of this survey. We are especially grateful to Val\'{e}rie Berth\'{e}, Samuel Petite, Ville Salo, Christopher-Lloyd Simon, and L\'{e}o Vivion for their particularly valuable comments and questions offered both during and after the talks.

\bibliographystyle{abbrv}
\bibliography{bib}

 \end{document}